\documentclass{amsart} %{amsart}%{article}%{elsart}
\usepackage{amssymb, amsmath}
\usepackage{amsfonts}
\usepackage{amssymb}
\usepackage{color}
\usepackage{comment}
\usepackage{epsfig}
\usepackage{float}
\usepackage{graphicx}
\usepackage[pagewise]{lineno}%\linenumbers
\newtheorem{theorem}{Theorem}[section]
\newtheorem*{thmA*}{Theorem A}
\newtheorem*{thmB*}{Theorem B}
\newtheorem*{lemA*}{Lemma A}
\newtheorem{corollary}[theorem]{Corollary}

\newtheorem{definition}{Definition}[section]

\newtheorem{lemma}{Lemma}[section]

\newtheorem{proposition}[theorem]{Proposition}
\newtheorem{remark}{Remark}[section]

\newtheorem*{assumptionA*}{Assumption A}
\newtheorem*{assumptionB*}{Assumption B}
\newtheorem*{assumptionC*}{Assumption C}

\newcommand{\diam}{\text{\rm diam}}

\newcommand{\parallelsum}{\mathrel{\!/\mkern-5mu/\!}}
\def\be{\begin{equation}}
\def\ee{\end{equation}}

\def\tilde{\widetilde}

\numberwithin{equation}{section} \numberwithin{theorem}{section}
\numberwithin{figure}{section}
\vspace{1cm}

\usepackage{array,makecell}
\usepackage{graphicx, array, blindtext}

\allowdisplaybreaks

\begin{document}
\bibliographystyle{siam}

\title[]
{On the quantitative lower bounds for solutions to the Boltzmann equation in nonconvex domains}

\author{Jhe-kuan SU}

\date{\today}

\begin{abstract}
In this article, we study the continuous mild solutions to the Boltzmann equation in a bounded spatial domain, under either angular cutoff assumption or non-cutoff assumption. Without assuming convexity of the spatial domain, we establish a Maxwellian lower bound in the cutoff case, and a weaker-than-Maxwellian lower bound for the non-cutoff case. This extends the results of \cite{Bri1,Bri2}, where the convexity of the domain was required.
\end{abstract}

\maketitle

\section{Introduction}
%%%%%%%%%%%%%%%%%%%%%%%%%%%%%%%%%%%%%%%%%%%%%%%%%%%%%%

In this paper, we investigate the lower bounds of the mild solutions to the Boltzmann equation on a spatial domain $\Omega \subset \mathbb{R}^3$ satisfying%with $C^2$ boundary:

\begin{assumptionA*}
    $\Omega$ is a connected, bounded open set with $C^2$ boundary.
\end{assumptionA*}

\noindent
The Boltzmann equation under consideration reads
\begin{equation}\label{Boltzmann equation}
     \partial_t f(t,x,v)+ v \cdot \nabla f(t,x,v)=Q[f,f](t,x,v),\ \forall (t,x,v) \in [0,T) \times \Omega \times \mathbb{R}^3,
\end{equation}
\begin{equation}\label{Boltzmann equation initial condition}
     f(0,x,v)=f_0(x,v),\ \forall (x,v) \in \Omega \times \mathbb{R}^3,
\end{equation}
where $T>0$.
We consider the mixed boundary condition: for any $(x,v) \in \Gamma^-$, we have
\begin{equation}\label{boundary condition}
f(x,v)= \alpha f(t,x,R(x,v)) +(1-\alpha)\left( \int_{w\cdot n(x)>0} f(t,x,w)(w\cdot n(x)) dw \right)\frac{1}{2\pi T_B^2}e^{-\frac{|v|^2}{2T_B}}
\end{equation}
for some $T_B>0$ and $\alpha \in [0,1]$. Here, $\Gamma^-:=\{ (x,v) \in \partial \Omega \times \mathbb{R}^3 | n(x)\cdot v<0 \}$, $R(x,v):=v-2(v\cdot n(x))n(x)$, where $n(x)$ is the outward unit normal vector at the boundary point $x \in \partial\Omega$.  

The collision operator $Q$ is defined as 

\begin{equation}
Q[h_1,h_2](v):=\int_{\mathbb{R}^3\times \mathbb{S}^2}B(|v-v_*|,\cos{\theta})[h_2(v')h_1(v'_*)-h_2(v)h_1(v_*)]\,dv_*\,d\sigma,
\end{equation}
where $v'$, $v_*$, $v'_*$ are defined as:
\begin{equation*}
    v':=\frac{v+v_*}{2}+\frac{|v-v_*|}{2}\sigma,\ v'_*:=\frac{v+v_*}{2}-\frac{|v-v_*|}{2}\sigma
\end{equation*}
with $\cos{\theta}:= \left\langle \frac{v-v_*}{|v-v_*|},\sigma \right\rangle$.
We assume that the collision kernel $B \geq 0$ and satisfies:
\begin{assumptionB*}
 \begin{equation}
    B(|v-v_*|, \cos{\theta})=\Phi(|v-v_*|)b(\cos{\theta}),
\end{equation}
\end{assumptionB*}
\noindent where $\Phi:=\Phi(|v-v_*|)$ is a function defined on $[0,\infty)$ satisfying

\begin{equation} \label{Phi estimate 1}
 c_{\Phi}r^{\gamma}  \leq \Phi(r) \leq C_{\Phi} r^{\gamma},\ \forall r \in [0,\infty),
\end{equation}
or 
\begin{equation}\label{Phi estimate 2}
\begin{cases}  c_{\Phi}r^{\gamma}  \leq \Phi(r) \leq C_{\Phi} r^{\gamma},\ &\forall r \in [1,\infty), \\ 
c_{\Phi}  \leq \Phi(r) \leq C_{\Phi} ,\ &\forall r \in [0,1]
\end{cases} 
\end{equation}
for some positive constants $c_{\Phi}$, $C_{\Phi}$ and $\gamma \in (-3, 1]$. Here, we assume that the function $b: \theta \mapsto b(\cos{\theta})$ is continuous on
$(0, \pi]$, positive near $\frac{\pi}{2}$, and satisfies

\begin{equation} \label{b estimate}
   \lim\limits_{\theta \rightarrow 0^+} \frac{b(\cos{\theta})\sin{\theta}}{\theta^{-(1+\nu)}} =b_0,
\end{equation}
for some $b_0>0$ and $\nu \in (-\infty ,2)$.

The problems of quantifying the positivity of the solutions to \eqref{Boltzmann equation}-\eqref{boundary condition} have attracted considerable attention from many authors. Beyond interest in physics, the problem also plays a crucial role in kinetic theory. For example, the exponential lower bound is crucial in the study of the behavior of the solution to the related Laudau equation in \cite{Des1,Des2}. In \cite{Gua1}, the Maxwellian lower bound is applied to derive the uniqueness of the solution to the Boltzmann equation in $L_v^1L_x^{\infty}(1+|v|^{2+0})$. 
In 1933, Carleman \cite{Car1} showed the existence of exponential lower bounds on the radially symmetric solution of the spatially homogeneous Boltzmann equation with angular cutoff hard potentials. Since then, many results have been derived. For example, in \cite{Pul1} Pulvirenti and Wennberg improved the result by proving the Maxwellian lower bound for non-radially symmetric solutions of the spatially homogeneous Boltzmann equation with angular cutoff hard potentials. In \cite{Mou 1}, Mouhot removed the homogeneous assumption and obtained the Maxwellian lower bound for the torus domain.
This result was extended to the bounded convex domain with $C^2$ boundary by Briant in \cite{Bri1,Bri2}. We note that in the articles \cite{Bri1,Bri2,Mou 1}, they also derived the "less than Maxwellian" lower bounds for the non-cutoff case. Later, this result was improved to Maxwellian lower bounds in the case of hard and moderately soft potentials with the assumption that $\Omega$ is a torus by Imbert, Mouhot, Silvestre in \cite{Imb1} and in the case where $\Omega$ is $\mathbb{R}^3$ by Henderson, Snelson, and Tarfulea in \cite{Hen1}.

However, the problem of Maxwellian lower bounds on a non-convex spatial domain remains an open problem in both the cutoff and non-cutoff cases. In this article, we derive a Maxwellian lower bound on an open bounded connected but not necessarily convex domain with a $C^2$ boundary for the cutoff case. For the non-cutoff case, we also derive a weaker than Maxwellian lower bound.

In \cite{Bri1,Bri2,Mou 1}, the following spreading property of the collision operator
\begin{equation}
    Q^+\left[\mathbf{1}_{B(v_0,\delta)},\mathbf{1}_{B(v_0,\delta)}\right](v) \geq C\xi^{\frac{1}{2}}\mathbf{1}_{B(v_0,\delta\sqrt{2}(1-\xi))}(v),
\end{equation}
for some $C>0$ depending on the cross section, and any $\xi \in (0,1)$,
is a key ingredient to deriving a lower bound around one point, which could then be extended to any other point using the convexity of the domain. 

In this article, we focus on the geometry of a $C^2$ connected domain and show that for any two points of $\Omega$ there always exists a suitable zigzag between these two points, which allows the spreading of the lower bound from one point to the other point like in \cite{Bri1,Bri2}. One major difficulty arises from the requirement that these zigzags in the domain must remain sufficiently distant from the boundary of $\Omega$. In fact, the shape structure affects the amount of gas molecules passing through the boundary. In \cite{Bri1, Bri2}, the lower bounds depend on: (1) hydrodynamic quantities. (2) the modulus of continuity of $f_0$. In our result, the lower bounds also depend on $\operatorname{Conn}_d(\Omega)$, the maximum number of "good" zigzag segments needed to connect any two points in the spatial domain. We put the definition of $\operatorname{Conn}_d(\Omega)$ in \eqref{2025 06 09 02:40} and address the detailed analysis concerning the geometry in Section 2.  Our main contribution lies in showing that the Maxwellian lower bound can be propagated throughout a non-convex domain by exploiting geometric connectivity properties, quantified through $\operatorname{Conn}_d(\Omega)$.

\bigskip

Throughout Sections 1--5, we assume that $\nu <0$ (cutoff case), where $\nu$ is defined in \eqref{b estimate}. To classify the type of solutions to \eqref{Boltzmann equation}--\eqref{boundary condition}, we first decompose the collision operator $Q$ as follows:

\begin{equation}
    \begin{split}
        &Q[h_1,h_2](v)\\
        =&\int_{\mathbb{R}^3\times \mathbb{S}^2}B(|v-v_*|,\cos{\theta})[h_2(v')h_1(v'_*)-h_2(v)h_1(v_*)]\,dv_*\,d\sigma\\
        =&\int_{\mathbb{R}^3\times \mathbb{S}^2}B(|v-v_*|,\cos{\theta})[h_2(v')h_1(v'_*)]\,dv_*\,d\sigma\\
        &-h_2(v)\int_{\mathbb{R}^3\times \mathbb{S}^2}B(|v-v_*|,\cos{\theta})h_1(v_*)\,dv_*\,d\sigma\\
        =&:Q^+[h_1,h_2](v)-h_2(v)L[h_1](v),
    \end{split}
\end{equation}
where

\begin{align*}
    &Q^+[h_1,h_2](v):=\int_{\mathbb{R}^3\times \mathbb{S}^2}B(|v-v_*|,\cos{\theta})[h_2(v')h_1(v'_*)]\,dv_*\,d\sigma,\\
    &L[h_1](v):=\int_{\mathbb{R}^3\times \mathbb{S}^2}B(|v-v_*|,\cos{\theta})h_1(v_*)\,dv_*\,d\sigma.
\end{align*}

Next, we define the following notations:
\begin{align*}
    &t_b(x,v):=\sup\left\{\{0\} \cup \{t>0\mid x-sv \in \Omega,\, \forall \,0<s<t\}\right\},\\
 &x_b(x,v):=x-t_b(x,v)v,\\
    &\Gamma^0:=\{(x,v) \in \partial\Omega \times \mathbb{R}^3 \mid n(x)\cdot v =0 \},\\
    &\Gamma^+:=\{(x,v) \in \partial\Omega \times \mathbb{R}^3 \mid n(x)\cdot v >0 \},\\
    &\Gamma^-:=\{(x,v) \in \partial\Omega \times \mathbb{R}^3 \mid n(x)\cdot v <0 \},\\
    &\Gamma_-^0:=\{(x,v) \in \Gamma^0 \mid t_b(x,v)=0,t_b(x,-v)>0,\, \exists \, \delta >0  \ni  \, x-s'v \in (\overline{\Omega})^c\, \forall \, s' \in (0,\delta) \},
\end{align*}

\begin{equation}
\begin{split}&\Gamma_{conti}:= \left\{\{0\}\times \overline{\Omega} \times \mathbb{R}^3\right\}  \cup  \left\{ (0,\infty)\times(\Gamma^-\cup \Gamma^0_-)\right \}\\
&\cup \left\{ (t,x,v)\in (0,T) \times \{ \Omega \times \mathbb{R}^3 \cup \Gamma^+ \} \ | t<t_b(x,v) \, \mathrm{ or } \, (x_b(x,v),v) \in \Gamma^-\cup \Gamma^0_-) \right\}.
    \end{split}
\end{equation}
Using the above notations and decomposition, we are ready to introduce the continuous mild solutions of the Boltzmann equation:
\begin{definition} \label{3/9 01:05}
We assume that the domain $\Omega \subset \mathbb{R}^3$ satisfies \textbf{Assumption A} and the collision kernel $B$ satisfies \textbf{Assumption B} with $\nu <0$, where $\nu$ is introduced in \eqref{b estimate}.   Given a nonnegative, continuous function $f_0$ on $\overline{\Omega} \times \mathbb{R}^3$, we call a nonnegative function $f$ defined on $[0,T) \times (\overline{\Omega} \times \mathbb{R}^3)$ with $|f(t,x,v)| \leq C(1+|v|)^{-r}$ for some constant $C>0$ and $r>3$ for any $0< t\leq T$, $(x,v) \in \overline{\Omega} \times \mathbb{R}^3$ a "continuous mild" solution to \eqref{Boltzmann equation}--\eqref{boundary condition} with initial data $f_0$ if $f$ is continuous on $\Gamma_{conti}$ and for $(t,x,v) \in [0,T) \times \Omega \times \mathbb{R}^3$ and the function $f(t,x,v)$ satisfies the following integral identities:
\begin{equation}\label{Duhamel formula}
\begin{split}
f(t,x,v)&= f_{0}(X_{0,t}(x,v),v)\exp\left(-\int_0^tL[f(s,X_{s,t}(x,v),\cdot)](v)\, ds\right)\\
&+\int_0^t \exp\left(-\int_s^t L[f(s',X_{s',t}(x,v),\cdot)](v)\, ds'\right)\\
&\hspace{1cm}Q^+[f(s,X_{s,t}(x,v),\cdot),f(s,X_{s,t}(x,v),\cdot)](v)\,ds,
\end{split}
\end{equation}
when $t \leq t_{\partial}(x,v):=\sup\{t\geq 0|x-vs\in \overline{\Omega},\ \forall s \in [0,t]\}$, and
\begin{equation}\label{Duhamel formula B}
\begin{split}
f(t,x,v)&= \alpha f(t-t_{\partial}(x,v),X_{t-t_{\partial}(x,v),t}(x,v),R(X_{t-t_{\partial}(x,v),t}(x,v),v))\\
&\hspace{0.5cm}\exp\left(-\int_{t-t_{\partial}(x,v)}^tL[f(s,X_{s,t}(x,v),\cdot)](v)\,ds\right)\\
&+(1-\alpha)\left( \int_{w\cdot n(X_{t-t_{\partial}(x,v),t}(x,v))>0} f(t,X_{t-t_{\partial}(x,v),t}(x,v),w)(w\cdot n(X_{t-t_{\partial}(x,v),t}(x,v))) dw \right)\\
&\hspace{1.8cm} \frac{1}{2\pi T_B^2}e^{-\frac{|v|^2}{2T_B}} \exp\left(-\int_{t-t_{\partial}(x,v)}^tL[f(s,X_{s,t}(x,v),\cdot)](v)\,ds\right)\\
&+\int_{t-t_{\partial}(x,v)}^t \exp\left(-\int_s^t L[f(s',X_{s',t}(x,v),\cdot)](v)\,ds'\right)\\
&\hspace{1.4cm}Q^+[f(s,X_{s,t}(x,v),\cdot),f(s,X_{s,t}(x,v),\cdot)](v)\,ds,
\end{split}
\end{equation}
when $t \geq t_{\partial}(x,v)$.
Here, a detailed definition of the characteristic line $X_{s,t}(x,v)$ is provided in the \textbf{Appendix}.

\end{definition}

\begin{remark}
   In the case of specular reflection ($\alpha = 1$), the continuous mild solution of \eqref{Boltzmann equation}--\eqref{boundary condition} can also be written as the continuous function on $\Gamma_{conti}$, $f$ such that for all $(t,x,v) \in [0,T]\times \overline{\Omega} \times \mathbb{R}^3$,
\begin{equation}\label{Duhamel formula specular}
\begin{split}
f(t,x,v)&= f_{0}(X_{0,t}(x,v),V_{0,t}(x,v))\exp\left(-\int_0^tL[f(s,X_{s,t}(x,v),\cdot)](V_{s,t}(x,v))\,ds\right)\\
&+\int_0^t \exp\left(-\int_s^t L[f(s',X_{s',t}(x,v),\cdot)](V_{s',t}(x,v))\,ds'\right)\\
&\hspace{1cm}Q^+[f(s,X_{s,t}(x,v),\cdot),f(s,X_{s,t}(x,v),\cdot)](V_{s,t}(x,v))\,ds,
\end{split}
\end{equation}
where $V_{s,t}$ is defined in the \textbf{Appendix}.

\end{remark}

\begin{remark}
The continuity assumption of the solution $f(t,x,v)$ is technical, as the proof only involves the continuity of the initial condition. 
Nonetheless, in the diffusive reflection boundary condition ($\alpha =0$), the existence of a solution of bounded variation is established in \cite{Guo1}. However, as shown in \cite{Kim1}, unlike in the convex case, the solution loses continuity, with singularities propagating along the grazing set. Motivated by \cite{Kim1}, we therefore consider the set %Thus, unlike the continuity assumption of \cite{Bri1,Bri2}, 
 $\Gamma_{conti}$. In the same work, the authors also showed that the continuity of solutions on $\Gamma_{conti}$ is ensured under diffuse reflection boundary conditions, with continuity of the initial data and addition conditions. 
\end{remark}

Before stating our main result, we introduce the following hydrodynamic quantities:

\begin{equation}
    \varrho_f(t.x):=\int_{v \in \mathbb{R}^3}f^2(t,x,v)\,dv
\end{equation}

\begin{equation} \label{Energy constant 1}
    e_f(t,x):= \int_{v \in \mathbb{R}^3}|v|^2f(t,x,v)\,dv, \,E_f:= \sup\limits_{[0,T)\times \Omega}(e_f(t,x)+\mathcal{\varrho}_f(t.x)),
\end{equation}

\begin{equation} \label{Energy constant 3}
    l_{f,p}(t,x):= \left(\int_{v \in \mathbb{R}^3}f^p(t,x,v)\,dv\right)^{\frac{1}{p}}, \,
    L_{f,p}:= \sup\limits_{[0,T)\times \Omega}l_{f,p}(t,x),
\end{equation}
It is also known that the quantity (total mass)
\begin{equation}
    M:=\int_{\Omega \times \mathbb{R}^3} f(t,x,v) \,dx\,dv
\end{equation}
does not depend on $t$, where $f$ is a (continuous mild) solution of \eqref{Boltzmann equation}--\eqref{boundary condition}.

\bigskip

We can now state our main result as follows:

\begin{theorem} \label{Main theorem}
     Suppose that $\Omega \subset \mathbb{R}^3$ satisfies \textbf{Assumption A} and that the kernel $B$ satisfies \textbf{Assumption B} with $\nu<0$. We consider a non-negative continuous function $f_0$ on $\overline{\Omega} \times \mathbb{R}^3$. Let $f(t,x,v)$ be a continuous mild solution of \eqref{Boltzmann equation}--\eqref{boundary condition} on $[0,T) \times \overline{\Omega} \times \mathbb{R}^3$, with initial condition $f_0$ for some $T>0$, and let $\alpha \in [0,1]$. We assume that $f(t,x,v)$ satisfies the following properties:
  \begin{enumerate}
        \item $M>0$;
        %\item $f$ is a continuous function on $[0,T]\times ( \overline{\Omega} \times \mathbb{R}^3 - \Gamma^0)$;
        \item $E_f<\infty$ if $\gamma \geq 0$ and $\max\{ E_f, L_{f,{p_{\gamma}}} \}< \infty$, where $p_\gamma > \frac{3}{3+\gamma} >0$, if $-3 < \gamma <0$.
    \end{enumerate}

    Then, the following lower bound holds:
    There exists $0 <\tau_0\leq T$ such that for any $\tau \in (0,\tau_0)$, there exist $\Delta_{\tau_0}>0$, $\rho>0$ and $\theta>0$ depending on $\tau_0$, $M$, $C_{\Phi},c_{\Phi}, \gamma, b_0, \nu, E_f$ (and $L_{f,p}$ if $\gamma <0$ ), $\tau$, $\alpha$, $\Omega$ and on the modulus of continuity of $f_0$, such that

    \begin{equation}
        f(t,x,v) \geq \frac{\rho}{(2\pi\theta)^{\frac{3}{2}}}e^{-\frac{|v|^2}{2\theta}},\ \forall \, t \in [\tau,\Delta_{\tau_0}),\ \forall \, (x,v) \in \overline{\Omega}\times \mathbb{R}^3.
    \end{equation}

\end{theorem}

The next theorem can be derived from \textbf{Theorem \ref{Main theorem}}:
\begin{theorem} \label{Main theorem 2}
     Suppose that $\Omega \subset \mathbb{R}^3$ satisfies \textbf{Assumption A} and that the kernel $B$ satisfies \textbf{Assumption B} with $\nu<0$. We consider a non-negative continuous function $f_0$ on $\overline{\Omega} \times \mathbb{R}^3$. Let $f(t,x,v)$ be a continuous mild solution of \eqref{Boltzmann equation}--\eqref{boundary condition} on $[0,T) \times \overline{\Omega} \times \mathbb{R}^3$, with initial condition $f_0$ for some $T>0$, and let $\alpha \in [0,1]$. We assume that $f(t,x,v)$ satisfies the following properties:
  \begin{enumerate}
        \item $M>0$;
        %\item $f$ is a continuous function on $[0,T]\times ( \overline{\Omega} \times \mathbb{R}^3 - \Gamma^0)$;
        \item $E_f<\infty$ if $\gamma \geq 0$ and $\max\{ E_f, L_{f,{p_{\gamma}}} \}< \infty$, where $p_\gamma > \frac{3}{3+\gamma} >0$, if $-3 < \gamma <0$.
    \end{enumerate}

    Then, the following lower bound holds:
    For any $\tau \in (0,T)$, there exist $\rho>0$ and $\theta>0$ depending on $M$, $C_{\Phi},c_{\Phi}, \gamma, b_0, \nu, E_f$ (and $L_{f,p}$ if $\gamma <0$ ), $\tau$, $\alpha$, $\Omega$ and on the modulus of continuity of $f_0$, such that

    \begin{equation}
        f(t,x,v) \geq \frac{\rho}{(2\pi\theta)^{\frac{3}{2}}}e^{-\frac{|v|^2}{2\theta}},\ \forall \, t \in [\tau,T),\ \forall \, (x,v) \in \overline{\Omega}\times \mathbb{R}^3.
    \end{equation}

\end{theorem}

For the non-cutoff case, the result is stated in Chapter 6. We note that the numbers $\rho$ and $\theta$ are computable thanks to the special zigzag which will be introduced in \textbf{Proposition \ref{zigzag lemma}}, through which we can guarantee that a certain amount and range of the lower bound can be generated through the propagation in \textbf{Proposition \ref{redoable translation proposition}} and derive a computable uniform lower bound (which depends on $\Omega$ since the construction of the zigzag is affected by the geometric properties of $\Omega$) as in \textbf{Proposition \ref{initial point lower bound near boundary}}. 

The remainder of the article is organized as follows. In Section 2, we analyze the boundary of non-convex domain to derive some key properties. In Section 3, we focus on the derivation of a series of diluting initial lower bounds and the propagation effect of lower bounds. In Section 4 and Section 5, we provide a detailed proof of the Maxwellian lower bound on mild solution to the Boltzmann equation with non-fully specular ($0 \leq \alpha < 1$) and fully specular reflection boundary condition ($\alpha=1$), respectively. In Section 6, a weaker exponential lower bound of solutions to the Boltzmann equation for non-cutoff case is provided. 

\section{Geometric properties near the boundary}

In this section, we introduce some useful lemmas.

\begin{lemma} \label{initial cover over boundary}
    Given $\Omega$ which satisfies \textbf{Assumption A}, there exists $\delta:=\delta(\Omega)>0$, such that for any $0<d<\min\{1,\delta\}$, one can find $m_1 \in \mathbb{N}$, $\{x^0_i\}_{i=1}^{m_1} \in \partial\Omega$ depending on $d$ such that
    \begin{equation} \label{initial cover}  
   \partial \Omega \subset  \underset{1 \leq i \leq m_1}{\bigcup} B\left(x_i^{0},\frac{d}{8}\right)
  \end{equation}
    so that, for each $1\leq i \leq m_1$, after orthogonal transformation of variables, the boundary near $x_i^0$ can be represented as a graph of a $C^2$ scalar function on $\mathbb{R}^2$. More precisely, for each $1 \leq i \leq m_1$, there exists a orthonormal basis $\{e_i^1, e_i^2, -n(x_i^0)\}$ and a $C^2$ function $\phi_i:\mathbb{R}^2 \rightarrow \mathbb{R}$ such that
    \begin{align*}
        \phi_i(0,0)=0,  \nabla\phi_i(0,0)=0,
    \end{align*}
    \begin{equation}\label{circle presentation inner}
    \begin{split}
        &\partial \Omega \cap B(x_i^0,3d)=\{ y \in B(x_i^0,3d)  \mid y=x_i^0+u_1e_i^1+u_2e_i^2-\phi_i(u_1,u_2)n(x^0_i)  \},\\
        &\partial \Omega \cap \overline{B(x_i^0,3d)}=\{  y \in \overline{B(x_i^0,3d)} \mid y=x_i^0+u_1e_i^1+u_2e_i^2-\phi_i(u_1,u_2)n(x^0_i)\},\\
        &\Omega \cap B(x_i^0,3d)=\{  y \in B(x_i^0,3d) \mid \\ & \hspace{3.5cm}y= x_i^0+u_1e_i^1+u_2e_i^2-u_3n(x^0_i), u_3> \phi_i(u_1,u_2) \}.
    \end{split}
    \end{equation}
    Furthermore, we have $|\nabla \phi_i| < \frac{1}{100}$, $|\nabla^2 \phi_i| < \tilde{C}(\Omega)$ for some $\tilde{C}(\Omega)>\max\left\{\frac{4}{d},\frac{1}{2d_r}\right\}$, and for any $1\leq i\leq m_1$,

    \begin{equation} \label{additional initial cover}  
   B\left(x_i^0-\frac{d}{2}n(x_i^0),\frac{d}{2}\right) \subset  B\left(x_i^{0},d\right)\cap \Omega.
  \end{equation}
    Here, $d_r$ is introduced in \textbf{Remark \ref{2025/08/19 02:54}}.

\end{lemma}

\begin{proof}
   This is a direct consequence of the application of the Heine-Borel theorem.
\end{proof}
\begin{remark}\label{2025/08/19 02:54}
It is known that any bounded domain with $C^2$ boundary satisfies the uniform interior sphere condition. That is, there exists $d_r$ such that for any $x \in \partial \Omega$, one can find a ball with radius $d_r>0$ in $\Omega$ that intersects with $\partial \Omega$ only on $x$. In this article, we choose $d$ from \textbf{Lemma \ref{initial cover over boundary}} to be smaller than $d_r$.
    
\end{remark}
\begin{corollary}\label{small coro regarding distance}
   Let $\Omega$ satisfy \textbf{Assumption A} with $\delta:=\delta(\Omega)>0$. For any $0<d<\min\{1,\delta \}$, which is given in \textbf{Lemma \ref{initial cover over boundary}}, given  $1\leq i \leq m_1$, if $x \in B(x^0_i,\frac{3d}{2})$ with coordinate representation
   \begin{equation*}
       x= x_i^{0}+x_1e^1_{i}+x_2e^2_{i}-x_3n(x_i^{0}),
   \end{equation*}
 then
   \begin{equation}
       d(x, \partial\Omega \cap B(x^0_i,3d)) \geq \frac{1}{2}(x_3-\phi_{i}(x_1,x_2)).
   \end{equation}
\end{corollary}

\begin{proof}
    This is a direct consequence of the fact that $|\nabla\phi_i| <\frac{1}{100}$ and 
 the mean value theorem. Indeed, given $x= x_i^{0}+x_1e^1_{i}+x_2e^2_{i}-x_3n(x_i^{0})$, if there exists $x'= x_i^{0}+x'_1e^1_{i}+x'_2e^2_{i}-x'_3n(x_i^{0}) \in B(x,\frac{1}{2}(x_3-\phi_{i}(x_1,x_2))) $ such that $x' \in \partial\Omega \cap B(x^0_i,3d))$, we have by \textbf{Lemma \ref{initial cover over boundary}} that 
 \begin{align*}
     \phi_i(x'_1,x'_2)=x'_3.
 \end{align*} 
 By the mean value theorem, there exists $\tilde{x}_1$ between $x_1$ and $x'_1$, $\tilde{x}_2$ between $x_2$ and $x'_2$ such that

 \begin{equation}
     |\nabla\phi_i(\tilde{x}_1,\tilde{x}_2)|=\frac{|\phi_i(x_1,x_2)-\phi_i(x'_1,x'_2)|}{|(x_1,x_2)-(x'_1,x'_2)|} \geq \frac{\frac{1}{2}(x_3-\phi_{i}(x_1,x_2))}{\frac{1}{2}(x_3-\phi_{i}(x_1,x_2))}=1, 
 \end{equation}
 from which we deduce a contradiction to the fact that $|\nabla\phi_i(\tilde{x}_1,\tilde{x}_2)|<\frac{1}{100}$.
\end{proof}

\begin{remark}\label{remark 2}
    Notice that by the triangle inequality we have $\Omega_{\frac{d}{8}} \subset \bigcup\limits_{i=1}^{m_1}B(x_i^{0},\frac{d}{4})$, where $\Omega_{\epsilon }:=\{ x\in \Omega | \ d(x,\partial\Omega)< \epsilon \}$.
\end{remark}

Next, we define $y_i^0:=x_i^0-\frac{3d}{4}n(x_i^0)$. Observe that $y_i^0 \in \Omega-\Omega_{\frac{d}{8}}$, since $B(y_i^0,\frac{d}{8}) \subset B\left(x_i^0-\frac{d}{2}n(x_i^0),\frac{d}{2}\right) \subset  B\left(x_i^{0},d\right)\cap \Omega $.

\begin{lemma} \label{initial cover general}
    Given $\Omega$ which satisfies \textbf{Assumption A}. Let $ 0< d< \min(1,\delta)$ and $y_i^0 \in \Omega$ as defined in \textbf{Lemma \ref{initial cover over boundary}}. 
    \smallskip Then, there exist $m_2 \in \mathbb{N}$, $\{ y_{i}^0 \}_{i=m_1+1}^{m_1+m_2} \subset \Omega-\Omega_{\frac{d}{8}}$ such that
\begin{enumerate}
    \item  For any $x \in \Omega_{\frac{d}{8}} $, there exists $1 \leq i \leq m_1$ such that $\overline{xy_i^0} \subset \Omega$. 
    \item For any $x \in \Omega-\Omega_{\frac{d}{8}}$, there exists $m_1+1\leq i\leq m_1+m_2$ with the property that $x \in B(y_i^0,\frac{d}{16})$ and $\overline{xy_i^0} \subset \Omega$.
\end{enumerate}

\noindent
    Here, $\overline{xy}:=\{sx+(1-s)y \mid 0\leq s \leq 1 \}$.
    
\end{lemma}

\begin{proof}
By \textbf{Lemma \ref{initial cover over boundary}} and \textbf{Remark \ref{remark 2}}, there exist $m_1$ and $\{x_i^0\}_{i=1}^{m_1}$ such that 
\begin{equation}
    \Omega_{\frac{d}{8}} \subset \bigcup_{i=1}^{m_1} B \left (x_i^0,\frac{d}{4}\right).
\end{equation}
 Because of the representation of $\partial \Omega \cap B(x^0_i,d)$ in \textbf{Lemma \ref{initial cover over boundary}}, we see that
 \begin{equation}
     \overline{xy_i^0}\in \Omega,\ \forall x \in B\left(y_i^0,\frac{d}{4}\right).
 \end{equation}
 Next, we cover $\Omega-\Omega_{\frac{d}{8}}$ by $\{ B(x,\frac{d}{16})  \}_{x \in \Omega-\Omega_{\frac{d}{8}}}$. Since $\Omega-\Omega_{\frac{d}{8}}$ is compact. The Heine-Borel covering theorem ensures the existence of a finite sub-cover of this collection (Here, we choose the sub-cover with the minimum number of covers). We then denote the corresponding centers by $\{ y_i^0 \}_{i=m_1+1}^{m_1+m_2}$. Finally, we conclude the proof by observing that 
 \begin{equation*}
B\left(y_i^0,\frac{d}{16}\right) \subset \Omega -\Omega_{\frac{d}{16}} \subset \Omega,\ \forall  \, m_1+1 \leq i \leq m_1+m_2.  
 \end{equation*}
    
\end{proof}

\begin{lemma} \label{line near point lemma}
    Suppose $\Omega$ which satisfies \textbf{Assumption A}, with $\delta$ and $d$ as given in \textbf{Lemma \ref{initial cover over boundary}}. For any $1\leq i \leq m_1$, given
    \begin{align*}
    &x=x^0_i+x_{1}e^1_i+x_{2}e^2_i-x_{3}n(x^0_i),\\
    &y=x^0_i+y_{1}e^1_i+y_{2}e^2_i-y_{3}n(x^0_i)
\end{align*}
    with $x \in \overline{B(x^0_i,\frac{d}{4})}, \,y \in \overline{ B(x^0_i,\frac{3d}{4})}$, $(y_1^2+y_2^2)^{\frac{1}{2}} \leq \frac{d}{4}$, and $\frac{x-y}{|x-y|} \cdot n(x^0_i) \geq \frac{1}{2}$, \\ we have
    \begin{equation}
        d(\overline{xy},\partial\Omega \cap \overline{B(x^0_i,d)})=d(x,\partial\Omega \cap \overline{B(x^0_i,d)}).
    \end{equation}
\end{lemma}

\begin{proof}

Note that it suffices to show that for any $0\leq q_1 < 1$, there exists $\epsilon'(q_1)>0$ such that for any $q_1<q_2<q_1+\epsilon'(q_1)$, we have $d(q_1x+(1-q_1)y, \partial \Omega \cap \overline{B(x^0_i,d)})> d(q_2x+(1-q_2)y, \partial \Omega \cap \overline{B(x^0_i,d)})$.

\begin{figure}[ht]
\centering
\includegraphics[width=0.65\linewidth]{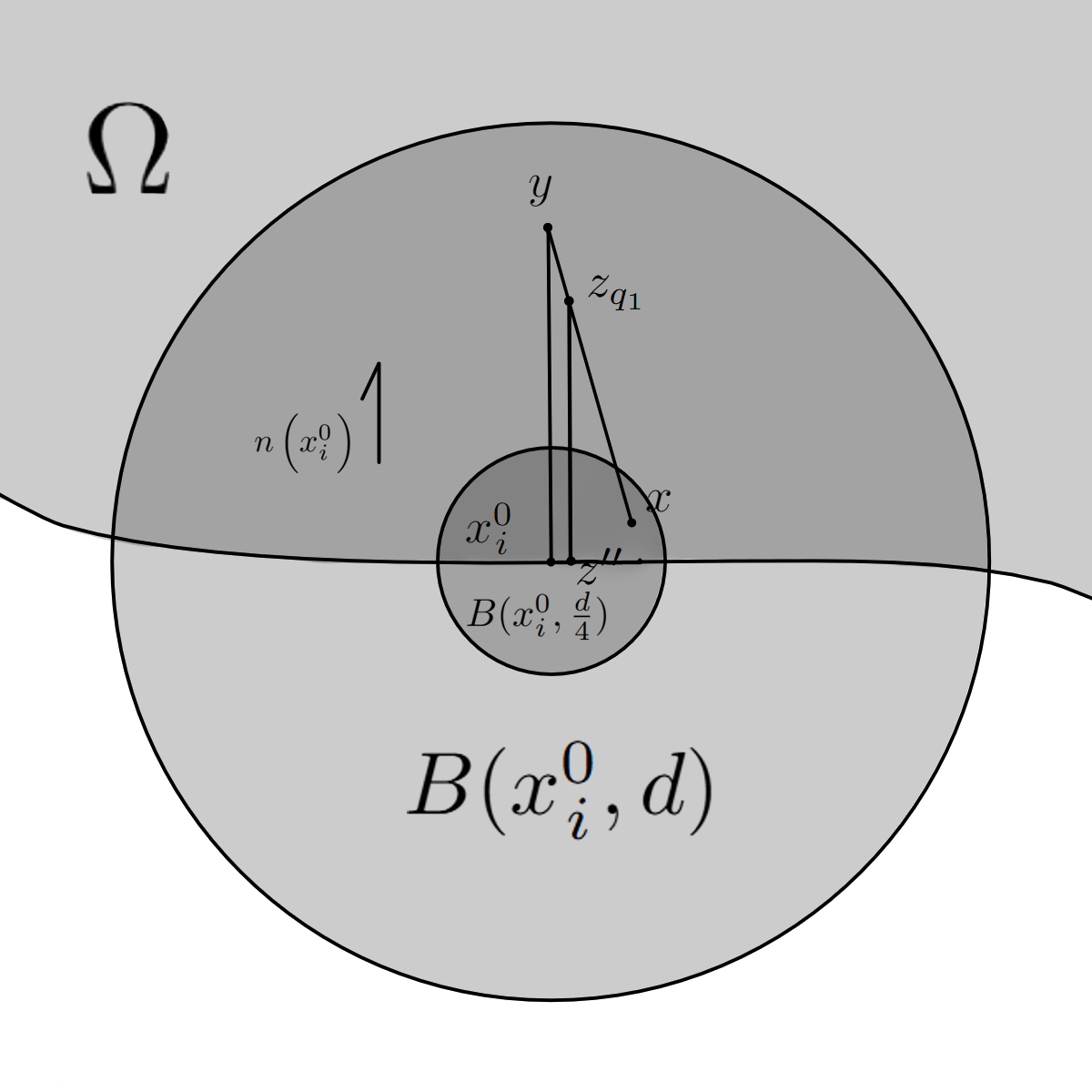}  
\caption{A figure illustrating how the points in the proof of Lemma \ref{line near point lemma} is located.}
\label{Figure 2.png}
\end{figure}

To show that, we first notice that by non-emptiness and compactness of $\partial \Omega \cap \overline{B(x^0_i,d)}$, there exists a point $z'' \in \partial \Omega \cap \overline{B(x^0_i,d)}$ such that 
\begin{align*}
    |z''-q_1x-(1-q_1)y|=d(q_1x+(1-q_1)y, \partial \Omega \cap \overline{B(x^0_i,d)}).
\end{align*}

We parameterize $x$, $y$, $z''$ as follows:
\begin{align*}
    &x=x^0_i+x_{1}e^1_i+x_{2}e^2_i-x_{3}n(x^0_i),\\
    &y=x^0_i+y_{1}e^1_i+y_{2}e^2_i-y_{3}n(x^0_i),\\
    &z''=x^0_i+z''_1e^1_i+z''_2e^2_i-z''_3n(x^0_i).
\end{align*}

We define the following quantities:
\begin{align}
    \label{Z_1 eq}&Z_1:=\frac{(z''-q_1x-(1-q_1)y)}{|z''-q_1x-(1-q_1)y|}=\frac{z''-z_{q_1}}{|z''-z_{q_1}|},\\
    \label{Z_2 eq}&Z_2:=\frac{x-q_1x-(1-q_1)y}{|x-q_1x-(1-q_1)y|}=\frac{x-y}{|x-y|}.
\end{align}
Before we proceed, we also define $z_{q_1}:=q_1x+(1-q_1)y$, and we notice that 
\begin{equation*}
    Z_1 = n(z'').
\end{equation*}
By \textbf{Lemma \ref{initial cover over boundary}}, we have (we recall $Z_1$ from \eqref{Z_1 eq})
\begin{align*}
    &Z_1\cdot n(x^0_i)\\
    =&n(z'')\cdot n(x^0_i) \\
     =& \frac{1}{\sqrt{|\partial_1\phi_{i}(z''_1,z''_2)|^2+|\partial_2\phi_{i}(z''_1,z''_2)|^2+1}}\\
    >&\sqrt{\frac{10000}{10001}}.
\end{align*}

Notice that we have from the assumption of \textbf{Lemma \ref{line near point lemma}} that (we recall $Z_2$ from \eqref{Z_2 eq})
\begin{align*}
&n(x^0_i)\cdot Z_2\geq\frac{1}{2}.
\end{align*}

Hence, we deduce that
\begin{equation}\label{2025 08/13 03:22}
\begin{split}
    &Z_1 \cdot Z_2\\
    =&n(x^0_i) \cdot Z_2+(Z_1-n(x^0_i)) \cdot Z_2\\
    \geq& \frac{1}{2} -|Z_1-n(x^0_i)||Z_2|\\
    =&\frac{1}{2}-\sqrt{|Z_1|^2+|n(x^0_i)|^2-2Z_1\cdot n(x^0_i)}\\
    >&\frac{1}{2} -\sqrt{2\left(1-\sqrt{\frac{10000}{10001}}\right) }\\
    >&\frac{\sqrt{2}}{3}.
    \end{split}
\end{equation}
Now, we define
\begin{align*}
    &W_{1,//}:=((z''-z_{q_1}) \cdot Z_2) Z_2=|z''-z_{q_1}|(Z_1\cdot Z_2)Z_2,\\
    &W_{1,\perp}:=(z''-z_{q_1})-W_{1,//}.
\end{align*}
Clearly, we have $z''-z_{q_1}=W_{1,//}+W_{1,\perp}$, $W_{1,//} \parallelsum Z_2 \parallelsum x-y$, and $W_{1,\perp} \perp Z_2$.

Finally, by taking $\epsilon'(q_1) :=\min\left\{  \frac{\sqrt{2}|z''-z_{q_1}|}{6|x-y|}, 1-q_1 \right\}$, we have 
\begin{align*}
    &z''-q_2x-(1-q_2)y \in \overline{xy},
    \end{align*}
    and 
    \begin{align*}
    0&<|x-y|(q_2-q_1) < |z''-q_1x-(1-q_1)y| \frac{\sqrt{2}}{3}\\
    &<|z''-z_{q_1}| Z_1 \cdot Z_2\\
    &=|W_{1,//}|,
\end{align*}
for any $q_1<q_2<q_1+\epsilon'(q_1)$.

Hence, we have 
\begin{align*}
    &|W_{1,//}-|x-y|(q_2-q_1)Z_2|<|W_{1,//}|,
\end{align*}
and we deduce that
\begin{align*}
    &d\bigg(q_2x+(1-q_2)y, \partial \Omega \cap \overline{B(x^0_i,d)}\bigg)<d\left(q_1x+(1-q_1)y, \partial \Omega \cap \overline{B(x^0_i,d)}\right)
\end{align*}
for any $q_1<q_2<q_1+\epsilon'(q_1)$.
\end{proof}

Now, we introduce one more property regarding the geometry of $\Omega$:
\begin{lemma} \label{path wise connected inner}
    Suppose that $\Omega$ satisfies \textbf{Assumption A} with $\delta$, $d$ as given in \textbf{Lemma \ref{initial cover over boundary}}. Then, $\Omega-\Omega_{\frac{d}{8}}$ is pathwise connected; that is, given $x'_1, x'_2 \in \Omega-\Omega_{\frac{d}{8}}$, there exists a continuous curve in $\Omega-\Omega_{\frac{d}{8}}$ joining $x'_1$ and $x'_2$.
\end{lemma}

\begin{proof}
   Given $x'_1 , x'_2 \in \Omega-\Omega_{\frac{d}{8}}$. Since the domain $\Omega$ is connected, $\Omega$ is pathwise connected. Hence, there exist $l \in \mathbb{N} ,\{\tilde{x}_i\}_{i=1}^{l} \in \Omega$ such that the following curve 
\begin{align*}
    &\phi(s)
    =\begin{cases}
        [1-(l+1)s]x'_1+(l+1)s\tilde{x}_1, \ &\mathrm{for}\ 0 \leq s \leq \frac{1}{l+1},\\
        [1-(l+1)s+j]\tilde{x}_j+[(l+1)s-j]\tilde{x}_{j+1},\  &\mathrm{for}\ \frac{j}{l+1} \leq s \leq \frac{j+1}{l+1},\, 1\leq j \leq l-1, \\
        [1-(l+1)s+l]\tilde{x}_l+[(l+1)s-l]x'_2,\  &\mathrm{for}\ \frac{l}{l+1} \leq s \leq 1,
    \end{cases}
\end{align*}
lies within $\Omega$. 

As \textbf{Remark \ref{remark 2}} shows, we have
\begin{equation*}
    \{ \phi(s)\mid 0 \leq s\leq 1 \} \cap \Omega_{\frac{d}{8}} \subset \bigcup_{i=1}^{m_1}B(x^0_{i},\frac{d}{4}).
\end{equation*}

Without loss of generality, we can assume that there is a function $k:\mathbb{N}\cap[1,l]\longrightarrow \mathbb{N}\cap [1,m_1]$ such that $k(l-1)=k(l)$,

$\left\{
\begin{aligned}
        \tilde{x}_1 &\in \partial B(x^0_{k(1)},\frac{d}{4}) ,\\ 
    \overline{\tilde{x}_{j}\tilde{x}_{j+1}} &\subset \overline{ B(x^0_{k(j)},\frac{d}{4})},\, |\tilde{x}_{j}-\tilde{x}_{j+1}| \leq \frac{d}{10},\quad \forall \, j=1,...,l-1, \\
    \tilde{x}_{l} &\in \partial B(x^0_{k(l)},\frac{d}{4}) ,
\end{aligned}
\right.$

and that
\begin{align*}
  \overline{x'_1\tilde{x}_1} \subset \Omega-\Omega_{\frac{d}{8}},\ \overline{\tilde{x}_l x'_2} \subset \Omega-\Omega_{\frac{d}{8}}.
\end{align*}
Indeed, given $0\leq j \leq l$ a segment $\overline{\tilde{x}_{j}\tilde{x}_{j+1}}$ (here we consider $\tilde{x}_{0}:=x'_1$, $\tilde{x}_{l+1}:=x'_2$), and a number $1 \leq k \leq m_1$, the set $\overline{\tilde{x}_{j}\tilde{x}_{j+1}} \cap \partial B(x_k^0,\frac{d}{4})$ is finite. In fact, after adding finite subsections of the line, we can assume that either the entire line segment $\overline{\tilde{x}_{j}\tilde{x}_{j+1}}$ lies within a ball $\overline{B(x_{k(j)}^0, \frac{d}{4})}$ for some $1 \leq k(j) \leq m_1$ or belongs to $\cap_{k=1}^{m_1}B^c(x_k^0,\frac{d}{4}) \subset \Omega-\Omega_{\frac{d}{8}}$.

\begin{figure}[ht]
\centering
\includegraphics[width=1\linewidth]{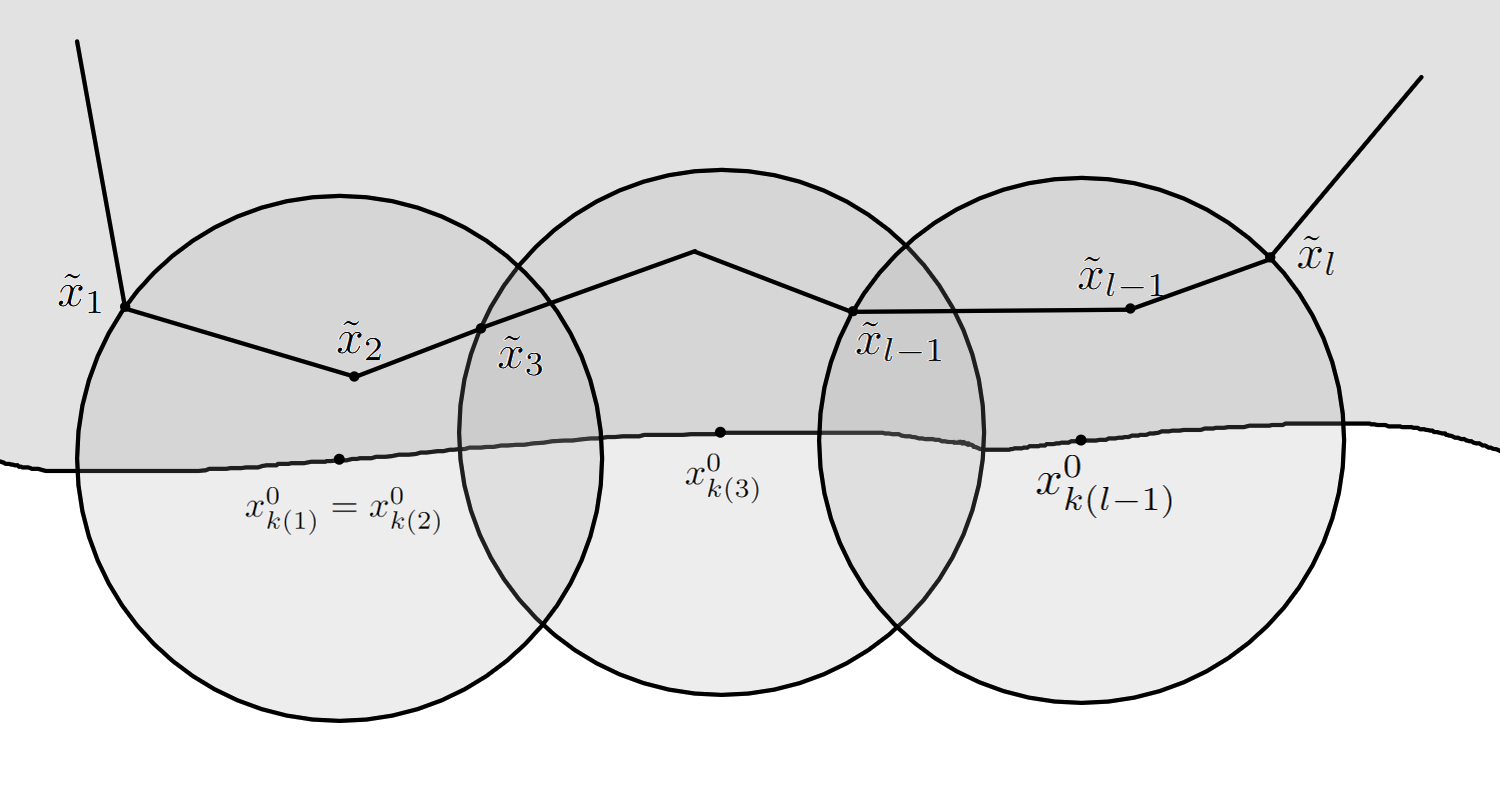}  
\caption{Example of construction of the re-indexation of k}
\label{Figure 5.png}
\end{figure}

Now, we define 
\begin{align*}
    &\tilde{\phi}(s)
    =\begin{cases}
        [1-2(l+1)s]x'_1+2(l+1)s\tilde{x}_1, \ &\mathrm{for}\ 0 \leq s \leq \frac{1}{2(l+1)},\\
        [2-2(l+1)s]\tilde{x}_1+[2(l+1)s-1]\tilde{x}'_1, \ &\mathrm{for}\ \frac{1}{2(l+1)} \leq s \leq \frac{1}{l+1},\\
        [1-(l+1)s+j]\tilde{x}'_j+[(l+1)s-j]\tilde{x}'_{j+1},\  &\mathrm{for}\ \frac{j}{l+1} \leq s \leq \frac{j+1}{l+1},\ 1 \leq j \leq l-1,\\
        [1-2(l+1)s+2l]\tilde{x}'_l+[2(l+1)s-2l]\tilde{x}_l,\  &\mathrm{for}\ \frac{l}{l+1} \leq s \leq \frac{2l+1}{2(l+1)},\\
        [2-2(l+1)s+2l]\tilde{x}_l+[2(l+1)s-2l-1]x'_2,\  &\mathrm{for}\ \frac{2l+1}{2(l+1)} \leq s \leq 1,
    \end{cases}
\end{align*}
where 
\begin{equation}
       \tilde{x}'_j:=  \tilde{x}_j-\frac{d}{2}n(x^0_{k(j)})\ ,\ 1 \leq j \leq l.
\end{equation}

Now, we show that this new curve belongs to $\Omega-\Omega_{\frac{d}{8}}$. To do this, it suffices to show that for $1 \leq j \leq l-1$,
\begin{align*}
    \overline{\tilde{x}'_j\tilde{x}'_{j+1}}, \overline{\tilde{x}_1\tilde{x}'_{1}},\overline{\tilde{x}_l\tilde{x}'_{l}} \subset \Omega-\Omega_{\frac{d}{8}}.
\end{align*}
To show that $\overline{\tilde{x}'_j\tilde{x}'_{j+1}} \subset \Omega-\Omega_{\frac{d}{8}}$ for $1\leq j \leq l-1$, we re-parametrize $\tilde{x}_j$ as follows:

\begin{align*}
\tilde{x}_j=x^0_{k(j)}+\tilde{x}_{j,1}e^1_{k(j)}+\tilde{x}_{j,2}e^2_{k(j)}-\tilde{x}_{j,3}n(x^0_{k(j)}).
\end{align*}
We deduce that 
\begin{align*}
\tilde{x}'_j=x^0_{k(j)}+\tilde{x}_{j,1}e^1_{k(j)}+\tilde{x}_{j,2}e^2_{k(j)}-\left(\tilde{x}_{j,3}+\frac{d}{2}\right)n(x^0_{k(j)}).
\end{align*}
Here, we recall that $\{e^1_{k(j)},e^2_{k(j)},n(x^0_{k(j)})\}$ is an orthonormal basis of $\mathbb{R}^3$.
Notice that $\overline{\tilde{x}'_j\tilde{x}'_{j+1}} \subset \Omega \cap B(x^0_{k(j)},d)$. Therefore, for $0 \leq s \leq 1$, we have 
\begin{align*}
    &d((1-s)\tilde{x}'_j+s\tilde{x}'_{j+1}, \partial \Omega)\\
     \geq &d((1-s)\tilde{x}'_j+s\tilde{x}'_{j+1}, \partial (\Omega \cap B(x^0_{k(j)},3d) ))\\
    \geq &\min \{ d((1-s)\tilde{x}'_j+s\tilde{x}'_{j+1},  \overline{\Omega} \cap \partial B(x^0_{k(j)},3d) ),d((1-s)\tilde{x}'_j+s\tilde{x}'_{j+1}, \partial \Omega \cap \overline{B(x^0_{k(j)},3d)} ) \}.
    \end{align*}
Notice that 

\begin{equation*}
    d((1-s)\tilde{x}'_j+s\tilde{x}'_{j+1},  \overline{\Omega} \cap \partial B(x^0_{k(j)},3d) ) \geq d( B(x^0_{k(j)},d),\partial B(x^0_{k(j)},3d)) =2d.
\end{equation*}
    %since we have $\tilde{x}'_j,\tilde{x}'_{j+1} \in B(x^0_{k(j)},d)$ and $B(x^0_{k(j)},d)$ is convex.

Since $|\tilde{x}'_j-x^0_{k(j)}| \leq |\tilde{x}'_j-\tilde{x}_j|+|\tilde{x}_j-x^0_{k(j)}| \leq \frac{3d}{4}$, we have  $\tilde{x}'_j \in B(x^0_{k(j)},\frac{3d}{2})$.   By \textbf{Corollary \ref{small coro regarding distance}} we have 
\begin{equation}
    d(\tilde{x}'_j, \partial \Omega \cap \overline{B(x^0_{k(j)},3d)} ) \geq \frac{d}{4} 
\end{equation}

 We also notice that, using the same argument as \eqref{2025 08/13 03:22}, we have $|n(x^0_{k(j)}) - n(x^0_{k(j+1)})|< \frac{1}{50}$ for $1 \leq j \leq l-1$, from which we deduce that 
 \begin{equation}
     |\tilde{x}'_j-\tilde{x}'_{j+1}| \leq |\tilde{x}_j-\tilde{x}_{j+1}| +\frac{d}{2}|n(x^0_{k(j)})-n(x^0_{k(j+1)})| \leq \frac{d}{10}+\frac{d}{100}< \frac{d}{9}.
 \end{equation}
Hence, for any $0 \leq s \leq 1$
\begin{align*}
    &d((1-s)\tilde{x}'_j+s\tilde{x}'_{j+1}, \partial \Omega \cap \overline{B(x^0_{k(j)},3d)} )
    \geq  \frac{d}{4}-\frac{sd}{9}\geq \frac{d}{8}.
\end{align*}
    Therefore, we conclude that for any $0 \leq s \leq 1$
    \begin{align*}
    d((1-s)\tilde{x}'_j+s\tilde{x}'_{j+1}, \partial \Omega)\geq \frac{d}{8}.
\end{align*}
We conclude that 
\begin{equation}\label{non near boundary final result}
    \overline{\tilde{x}'_j\tilde{x}'_{j+1}} \in \Omega-\Omega_{\frac{d}{8}},
\end{equation}
for $ 2 \leq j \leq l-2$.

For the case $j=1$ or $l$, we use \textbf{Lemma \ref{line near point lemma}}. Notice that for $j=1$ or $l$, we have $\tilde{x}_j \in \overline{B(x^0_{k(j)},\frac{d}{4})}, \,\tilde{x}'_{j} \in \overline{B(x^0_{k(j)},\frac{3d}{4})}$.  Since for $i=1$ or $l$ 
\begin{equation*}
    \sqrt{\tilde{x}^2_{j,1}+\tilde{x}^2_{j,2}} \leq \frac{d}{4},
\end{equation*}
and \begin{equation*}
    \frac{\tilde{x}_j-\tilde{x}'_j}{|\tilde{x}_j-\tilde{x}'_j|} \cdot n(x^0_{k(j)}) = 1,
\end{equation*}
all the assumptions of \textbf{Lemma \ref{line near point lemma}} are satisfied. Hence, we deduce that 
\begin{equation*}
    d(\overline{\tilde{x}_j\tilde{x}'_{j}}, \partial \Omega ) = d(\overline{\tilde{x}_j\tilde{x}'_{j}}, \partial \Omega \cap \overline{B(x^0_{k(j)},d)})= d(\tilde{x}_j, \partial \Omega \cap \overline{B(x^0_{k(j)},d)}) \geq \frac{d}{8},
\end{equation*}
for $j=1$ or $l$. We conclude that the curve $\tilde{\phi}(s)$ is the desired curve between $x_1'$ and $x_2'$ which belongs to $\Omega- \Omega_{\frac{d}{8}}$. 

\end{proof}

Before we proceed, for given $\delta$, $d$, we define $\mathbb{Y}:=\mathbb{Y}_1 \cup \mathbb{Y}_2$, where
\begin{equation*}
    \mathbb{Y}_1:=\{ 
y_i^0 \mid 1\leq i \leq m_1 \},\, \mathbb{Y}_2:=\{ 
y_i^0 \mid m_1+1\leq i \leq m_1+m_2 \}.
\end{equation*} 
Now, we introduce a useful property of $\mathbb{Y}$:
\begin{lemma} \label{zigzag lemma}
    Suppose that $\Omega$ satisfies \textbf{Assumption A} with $\delta$, $d$ as given in \textbf{Lemma \ref{initial cover over boundary}}. Given $x'_1, x'_2 \in \Omega$. Then, there exist $\{y_i\}_{i=1}^{N(x'_1,x'_2)} \in \mathbb{Y}$ such that $\{ x'_1, y_1, y_2,..., y_{N(x'_1,x'_2)}, x'_2 \}$ forms a zigzag in $\Omega$, that is,

\begin{align*}
    \overline{x'_1y_1} \in \Omega,\ \overline{y_{N(x'_1,x'_2)}x'_2} \in \Omega,\ 
    \overline{y_iy_{i+1}} \in \Omega,\ \forall 1 \leq i \leq N(x'_1,x'_2)-1.
\end{align*}
Moreover, we have 
\begin{align*}
    &d(\overline{x'_1y_1},\partial \Omega) \geq \min \left\{d(x'_1,\partial \Omega),\frac{d}{10} \right\},\\
    &d(\overline{x'_2y_{N(x'_1,x'_2)}},\partial \Omega) \geq \min \left\{d(x'_2,\partial \Omega),\frac{d}{10} \right\},\\
    &d(\overline{y_i,y_{i+1}}, \partial \Omega)>\frac{d}{10},
\end{align*}
for $1 \leq i \leq N(x'_1,x'_2)-1$.
We call such a zigzag a ``good zigzag" from $x'_1$ to $x'_2$ with number of segments $N(x'_1,x'_2)+1$.
    
\end{lemma}

\begin{proof}
Given $x'_1, x'_2 \in \Omega$, we shall show that there exist $y'_1, y'_2 \in \mathbb{Y}$ such that 

\begin{equation}\label{geo lemma first eq}
    d(\overline{x'_iy'_i}, \partial \Omega) \geq \min \left\{d(x'_i, \partial \Omega), \frac{d}{10} \right\},
\end{equation}
for $i=1,2$. Indeed, if $x'_i \in \Omega_{\frac{d}{8}}$, then by \textbf{Lemma \ref{initial cover over boundary}}, there exists $1 \leq l(i) \leq m_1$ such that $x'_i \in B(x_{l(i)}^{0},\frac{d}{4})$ and we can parameterize the boundary of $\Omega$ near $x^0_{l(i)}$ as follows:
\begin{small}
\begin{align*}
    & \partial \Omega \cap B(x_{l(i)}^0,d)\\
    =&\{ x_{l(i)}^0+u_{1}e_{l(i)}^1+u_{2}e_{l(i)}^2-\phi_{l(i)}(u_{1},u_{2})n(x^0_{l(i)})\mid \\
    & \hspace{0.2cm}x_{l(i)}^0+u_{1}e_{l(i)}^1+u_{2}e_{l(i)}^2-\phi_{l(i)}(u_{1},u_{2})n(x^0_{l(i)}) \in B(x_{l(i)}^0,d) \}
\end{align*}
\end{small}
with $|\nabla \phi_{l(i)}| < \frac{1}{100}$.
We shall show that
\begin{equation}\label{Lemma 2.3 first include}
    \overline{x'_iy^0_{l(i)}}\subset B(x_{l(i)}^0,d) \cap \Omega.
\end{equation}
Observe that, $\overline{x'_iy^0_{l(i)}} \subset B(x_{l(i)}^0,d)$ by the convexity of $B(x_{l(i)}^0,d)$, so it suffices to show that $\overline{x'_iy^0_{l(i)}}\subset  \Omega$. To show this, we define $s' := \inf{\{s \mid sx'_i+(1-s)y^0_{l(i)}\notin \Omega\}}$. If $\overline{x'_iy^0_{l(i)}}\not\subset  \Omega$, we have $0\leq s'<1$. Notice that $0<s'$ since $x'_i\in \Omega$ and $\Omega$ are open. Next, we observe that $z:=s'x'_i+(1-s')y^0_{l(i)}\in \partial\Omega$. Now, we recall the \textbf{Lemma \ref{initial cover over boundary}}. We can re-parametrize $y^0_{l(i)}$, $z$ and $x'_i$ as follows:
\begin{align*}
   y^0_{l(i)}&=x^0_{l(i)}+0u_1+0u_2-\frac{3d}{4}n(x^0_{l(i)}),\\
   z&=x^0_{l(i)}+z_1u_1+z_2u_2-z_3n(x^0_{l(i)}),\\    x'_i&=x^0_{l(i)}+x'_{i,1}u_1+x'_{i,2}u_2-x'_{i,3}n(x^0_{l(i)}).
\end{align*} Since $x'_i \in  B(x_{l(i)}^0,\frac{d}{4})$, we have 
\begin{align*}
    |x'_{i,1}|^2+|x'_{i,2}|^2 \leq |x'_{i,1}|^2+|x'_{i,2}|^2+|x'_{i,3}|^2 =|x'_i-x^0_{l(i)}|^2\leq \frac{d^2}{16}.
\end{align*}
Now, we consider the following function $\psi:[0,\sqrt{|z_1-x'_{i,1}|^2+|z_2-x'_{i,2}|^2}] \rightarrow \mathbb{R}$:
\begin{align*}
    \psi(\zeta):=\phi_{l(i)}(z_1+\frac{x'_{i,1}-z_1}{\sqrt{|z_1-x'_{i,1}|^2+|z_2-x'_{i,2}|^2}}\zeta,z_2+\frac{x'_{i,2}-z_2}{\sqrt{|z_1-x'_{i,1}|^2+|z_2-x'_{i,2}|^2}}\zeta).
\end{align*}
By the mean value theorem, there exists $\zeta'$ such that 
\begin{align*}
    \psi'(\zeta')&=\frac{\psi\left(\sqrt{|z_1-x'_{i,1}|^2+|z_2-x'_{i,2}|^2}\right)-\psi(0)}{\sqrt{|z_1-x'_{i,1}|^2+|z_2-x'_{i,2}|^2}-0}\\
    &\leq \frac{x'_{i,3}-\phi_{l(i)}(z_1,z_2)}{\sqrt{|z_1-x'_{i,1}|^2+|z_2-x'_{i,2}|^2}}.\\
\end{align*}
Here we used the fact that $x'_{i,3} \geq \phi_{l(i)}(x'_{i,1},x'_{i,2})$ from the fact that $x'_i \in \Omega$ and \eqref{circle presentation inner}.

Note that $y^0_{l(i)}$, $z$ and $x'_i$ belong to a straight line, so we have

\begin{align*}
    \frac{x'_{i,3}-\phi_{l(i)}(z_1,z_2)}{\sqrt{|z_1-x'_{i,1}|^2+|z_2-x'_{i,2}|^2}}&=\frac{\phi_{l(i)}(z_1,z_2)-\frac{3d}{4}}{\sqrt{|z_1|^2+|z_2|^2}}\\
    & \leq \frac{1}{100}-\frac{3}{4} = -\frac{74}{100}.
\end{align*}
Here we use the fact that $z \in \overline{B(x^0_{l(i)},d}).$

Finally, we get a contradiction to $\overline{x'_iy^0_{l(i)}}\not\subset  \Omega$ by computing:
\begin{align*}
    &\nabla \phi_{l(i)}\left(z_1+\frac{x'_{i,1}-z_1}{\sqrt{|z_1-x'_{i,1}|^2+|z_2-x'_{i,2}|^2}}\zeta',z_2+\frac{x'_{i,2}-z_2}{\sqrt{|z_1-x'_{i,1}|^2+|z_2-x'_{i,2}|^2}}\zeta'\right)\\
    &\hspace{0.75cm}\cdot \left(\frac{x'_{i,1}-z_1}{\sqrt{|z_1-x'_{i,1}|^2+|z_2-x'_{i,2}|^2}},\frac{x'_{i,2}-z_2}{\sqrt{|z_1-x'_{i,1}|^2+|z_2-x'_{i,2}|^2}}\right)\\ =&\psi'(\zeta') \leq -\frac{74}{100},
\end{align*}
which implies that $|\nabla \phi_{l(i)}| \geq \frac{74}{100}> \frac{1}{100}$.
Hence, we deduce that \eqref{Lemma 2.3 first include} holds.

\begin{figure}[ht]
\centering
\includegraphics[width=0.65\linewidth]{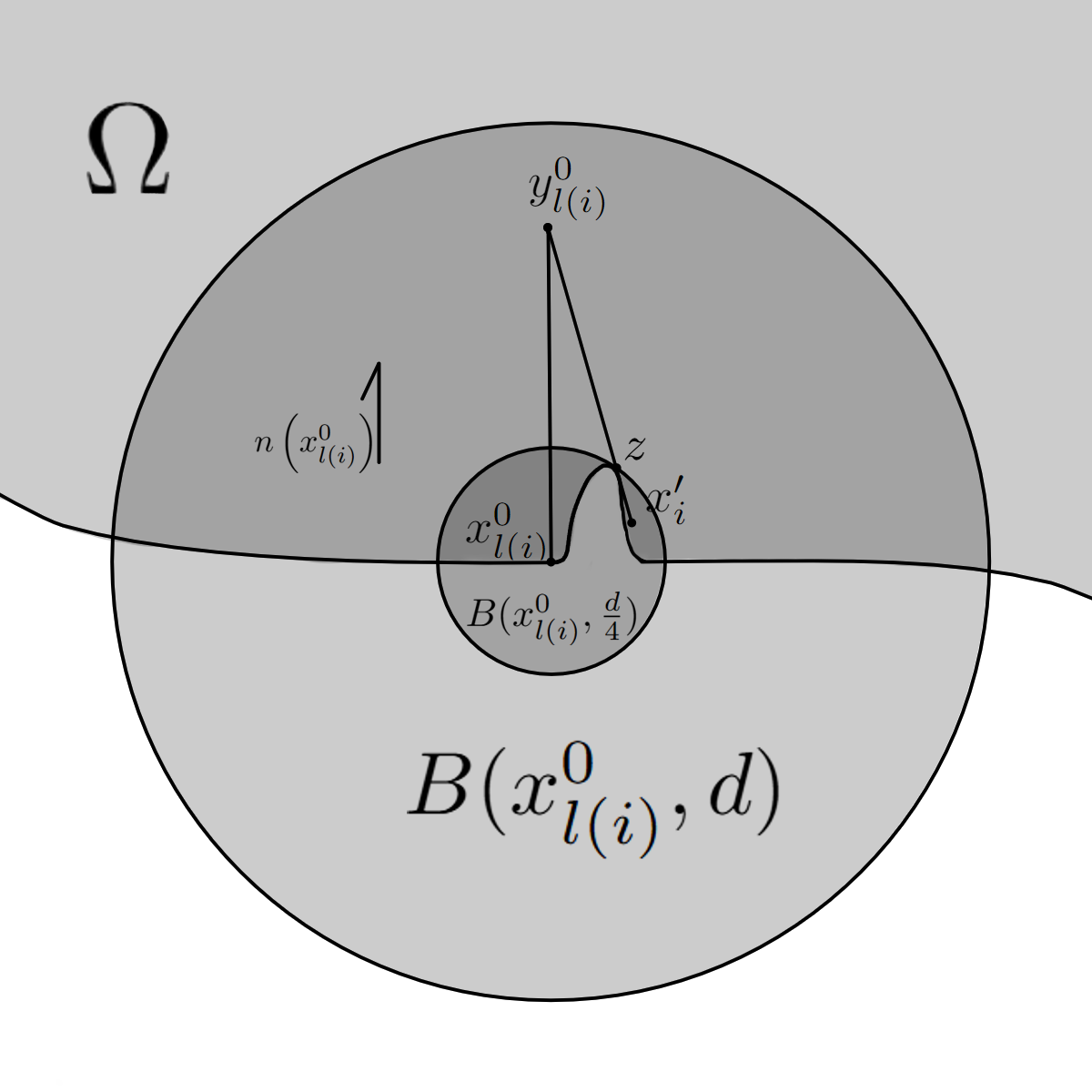}  
%\caption{XXXXX.}
\label{Figure 1.png}
\end{figure}

\

\

\

Hence, by \eqref{Lemma 2.3 first include}, we have
\begin{equation*}
    d(\overline{x'_iy^0_{l(i)}}, \partial \Omega) \geq d(\overline{x'_iy^0_{l(i)}}, \partial (\Omega \cap B(x^0_{l(i)},d))).
\end{equation*}

Next, we show that

\begin{equation*}
     d(\overline{x'_iy^0_{l(i)}}, \partial (\Omega \cap B(x^0_{l(i)},d))) \geq \min \{d(x'_i, \partial \Omega), \frac{d}{10} \}.
\end{equation*}
Since $\partial (\Omega \cap B(x^0_{l(i)},d))\subset[\partial B(x^0_{l(i)},d)\cap \overline{\Omega}] \cup [\partial \Omega \cap \overline{B(x^0_{l(i)},d)})]$, we have
\begin{equation*}
\begin{split}
    &d(\overline{x'_iy^0_{l(i)}}, \partial (\Omega \cap B(x^0_{l(i)},d)))\\
    \geq& d(\overline{x'_iy^0_{l(i)}}, [\partial B(x^0_{l(i)},d)\cap \overline{\Omega}] \cup [\partial \Omega \cap \overline{B(x^0_{l(i)},d)})])\\
      =& \min \{d(\overline{x'_iy^0_{l(i)}}, \partial B(x^0_{l(i)},d)\cap \overline{\Omega}) , d(\overline{x'_iy^0_{l(i)}}, \partial \Omega \cap 
 \overline{B(x^0_{l(i)},d)}  ) \}.
\end{split}
\end{equation*}
By the convexity of $B(x^0_{l(i)},d)$, we have
\begin{align*}
    d(\overline{x'_iy^0_{l(i)}}, \partial B(x^0_{l(i)},d)\cap \overline{\Omega}) \geq \frac{d}{4}
    \geq \min \{d(x'_i, \partial \Omega), \frac{d}{10} \}.
\end{align*}
Next, we notice that
\begin{equation}\label{2025 05/05 07:52}
    \frac{x'_i-y^0_{l(i)}}{|x'_i-y^0_{l(i)}|} \cdot n(x^0_{l(i)}) \geq \frac{2\sqrt{2}}{3}.
\end{equation}
Now, we use \textbf{Lemma \ref{line near point lemma}}. To show that all the assumption of \textbf{Lemma \ref{line near point lemma}} is satisfied, we notice that we have $x'_i \in B(x^0_{l(i)},\frac{d}{4})$, $y^0_{l(i)}=x^0_{l(i)}-\frac{3d}{4}n(x^0_{l(i)}) \in \overline{B(x^0_{l(i)}, \frac{3d}{4})}$, and \eqref{2025 05/05 07:52}.
As a result, we deduce that 
\begin{equation}\label{geo lemma sec eq}
\begin{split}
    d(\overline{x'_iy^0_{l(i)}}, \partial \Omega \cap \overline{B(x^0_{l(i)},d)}) = d(x'_i, \partial \Omega \cap \overline{B(x^0_{l(i)},d)}) \geq \min \{d(x'_i, \partial \Omega), \frac{d}{10} \}.
    \end{split}
\end{equation}

We showed that \eqref{geo lemma first eq} holds when $x'_i \in \Omega_{\frac{d}{8}}$.
\medskip
In the other case, namely when $x'_i \notin \Omega_{\frac{d}{8}}$, by \textbf{Lemma \ref{initial cover general}}, there exists $y\in \mathbb{Y}_2$ s.t. $x'_i \in B(y,\frac{d}{16})$. By using the fact that $B\left(x'_i,\frac{d}{8}\right)\cup B\left(y,\frac{d}{8}\right) \subset \Omega$, and that 

\begin{equation*}
    \overline{x'_iy} \subset B\left(x'_i,\frac{d}{8}\right)\cup B\left(y,\frac{d}{8}\right),
\end{equation*}
we have

\begin{align*}
    d(\overline{x'_iy},\partial \Omega) &\geq d\left(\overline{x'_iy}, \partial \left[B\left(x'_i,\frac{d}{8}\right)\cup B\left(y,\frac{d}{8}\right)\right]\right)\\
    &\geq \sqrt{(\frac{d}{8})^2-(\frac{d}{32})^2} >\frac{d}{10}.
\end{align*}

Next, we show that for any $w_1, w_2 \in \mathbb{Y}$, there exists a zigzag which satisfies the condition mentioned in \textbf{Lemma \ref{zigzag lemma}}.

We consider $w_1,w_2 \in \mathbb{Y}$ and define 
\begin{equation}
    \begin{split}
        &\mathbb{O}(w_1):=\bigg\{ y \in \mathbb{Y} \mid \mathrm{There \ exists \ a \ zigzag} \, \ \{w_1, u_1,...,u_s,y \}
         \ni \, u_i \in \mathbb{Y} \, \forall \, i=1,...,s, \\
         &d(\overline{w_1u_1}, \partial \Omega) > \frac{d}{10} ,d(\overline{u_iu_{i+1}}, \partial \Omega) > \frac{d}{10} \forall \, 1\leq i \leq s-1 , d(\overline{u_sy}, \partial \Omega) \> \frac{d}{10}\bigg\}.
    \end{split}
\end{equation}
If $B(\mathbb{O}(w_1),\frac{d}{16}):=\{\tilde{x} \mid \exists \, \tilde{y} \in \mathbb{O}(w_1) \ \ni \ |\tilde{x}-\tilde{y}|<\frac{d}{16}\}$ covers  $\Omega-\Omega_{\frac{d}{8}}$, then $\mathbb{O}(w_1)$ covers $\mathbb{Y}$. We notice that $\mathbb{Y} \subset \Omega-\Omega_{\frac{d}{8}}$. For a given point $y \in \mathbb{Y}$, there exists a point $ x \in  \mathbb{O}(w_1) \subset \mathbb{Y}$ with $|x-y| < \frac{d}{16}$ and a zigzag $\{w_1, u_1,...,u_s,x \}$ such that 
\begin{align*}
    d(\overline{w_1u_1}, \partial \Omega) > \frac{d}{10},\ d(\overline{u_iu_{i+1}}, \partial \Omega) > \frac{d}{10} ,\ \forall 1\leq i \leq s-1,\ d(\overline{u_sx}, \partial \Omega) > \frac{d}{10}.
\end{align*}
We notice that 
\begin{equation*}
    d(\overline{xy},\partial \Omega) \geq d\left(\overline{xy}, \partial \left[B\left(x,\frac{d}{8}\right)\cup B\left(y,\frac{d}{8}\right)\right]\right) \geq \sqrt{(\frac{d}{8})^2-(\frac{d}{32})^2} >\frac{d}{10}.
\end{equation*}
Thus, the zigzag $\{ w_1, u_1,...,u_s,x,y \}$ satisfies the condition mentioned in \textbf{Lemma \ref{zigzag lemma}}, and 
\begin{align*}
    y \in \mathbb{O}(w_1).
\end{align*}

On the other hand, suppose that $\Omega-\Omega_{\frac{d}{8}} \nsubseteq B\left(\mathbb{O}(w_1),\frac{d}{16}\right)$. We consider a point $\overline{w}_3 \in \Omega-\Omega_{\frac{d}{8}}$ such that $d(\mathbb{O}(w_1),\overline{w}_3) \geq d/16$.  By \textbf{Lemma \ref{path wise connected inner}}, there exists a continuous function $\overline{\phi}:[0,1]\longrightarrow \Omega-\Omega_{\frac{d}{8}}$ such that
$\overline{\phi}(0)=w_1$, $\overline{\phi}(1)=\overline{w}_3$.

Notice that $d(\mathbb{O}(w_1),\overline{\phi}(\cdot))$ is a continuous function from $[0,1]$ to $\mathbb{R}$ and that $d(\mathbb{O}(w_1),\overline{\phi}(0))=0,\ d(\mathbb{O}(w_1),\overline{\phi}(1))\geq\frac{d}{16}$. We define 
\begin{align*}
    \overline{\xi}:=\sup\left\{0 \leq s\leq 1 \mid d(\mathbb{O}(w_1),\overline{\phi}(s)) \leq \frac{d}{15}\right\}.
\end{align*}
 Now, we define $w_3:=\overline{\phi}(\overline{\xi})$. Notice that $\frac{d}{16} \leq d(\mathbb{O}(w_1),w_3)\leq\frac{d}{15}$. Since $\mathbb{O}(w_1)$ is finite, there exists a point $z$ in $\mathbb{O}(w_1)$ such that $\frac{d}{16} \leq  d(z,w_3) \leq \frac{d}{15}$. 

By \textbf{Lemma \ref{initial cover general}}, since $w_3 \in \Omega-\Omega_{\frac{d}{8}}$, there exists $w_4 \in \mathbb{Y}_2$ such that $d(w_3,w_4)< \frac{d}{16}$, which implies $d(z,w_4)\leq d(w_3,w_4)+d(z,w_3) <\frac{31}{240}d$. Hence, we deduce that 
\begin{equation*}
   d(\overline{zw_4},\partial \Omega) \geq d\left(\overline{zw_4}, \partial \left(B\left(z,\frac{d}{8}\right)\cup  B\left(w_4,\frac{d}{8}\right)\right)\right) \geq \sqrt{(\frac{d}{8})^2-(\frac{31d}{480})^2} >\frac{d}{10}.
\end{equation*}
Thus, we have that $w_4 \in \mathbb{O}(w_1)$. This implies that $\frac{d}{16}\leq d(\mathbb{O}(w_1),w_3) \leq d(w_4,w_3)<\frac{d}{16}$, which is a contradiction. We conclude that $\Omega-\Omega_{\frac{d}{8}} \subset B\left(\mathbb{O}(w_1),\frac{d}{16}\right)$. Therefore, $\mathbb{O}(w_1)$ contains $\mathbb{Y}$, and $w_2 \in \mathbb{O}(w_1)$, which shows that there exists a zigzag between $w_1$ and $w_2$ satisfying the conditions mentioned in \textbf{Lemma \ref{zigzag lemma}}.
\end{proof}

Next, we introduce a lemma which is similar to \textbf{Proposition 5.1} in \cite{Bri2}:
\begin{lemma}\label{grazing geo lemma}
     Suppose that $\Omega$ satisfies \textbf{Assumption A}. Given $\delta$,  $d$ in \textbf{Lemma \ref{initial cover over boundary}}, $v_M>0 \in \mathbb{R}$, and $\epsilon>0$. Then there exists $t_{\epsilon}(v_M)>0$ such that given $0<\tau_2\leq t_{\epsilon}(v_M)$, one can find $l_{\epsilon}(\tau_2)>0$ such that for all $1\leq i \leq m_1$, $(x,v) \in B(x^0_i,d) \times \overline{B(0,v_M)}$, 
\begin{equation}
    \forall \, s \in [0,\tau_2],\, X_s(x,v) \in \Omega_{l_{\epsilon}(\tau_2)}\cap B(x_i^0,d)
\end{equation}
implies
\begin{equation}
    \forall \, s \in [0,t_{\epsilon}(v_M)],\, |V_s(x,v)-v| < \epsilon.
\end{equation}
\end{lemma}
In fact, the numbers $t_{\epsilon}(v_M)$ and $l_{\epsilon}(\tau_2)$ can be chosen as follows.
\begin{equation}\label{t es and l es}
    t_{\epsilon}(v_M):=\frac{\epsilon}{8\tilde{C}(\Omega)v_M^2}, \, l_{\epsilon}(\tau_2):=\frac{\tau_2\epsilon}{64},
\end{equation}
where $\tilde{C}(\Omega)$ is defined in \textbf{Lemma \ref{initial cover over boundary}}.

\begin{remark}
    Unlike \textbf{Proposition 5.1},\cite{Bri2}, we do not need the lower bound of the velocity $v$. 
\end{remark}

\begin{proof}
Let $\epsilon>0$ and $0 \leq \tau_2 \leq t_{\epsilon}(v_M)$. Assume that
\begin{equation}\label{a small assumption in lemma}
    \forall s \in [0,\tau_2], X_s(x,v) \in \Omega_{l_{\epsilon}(\tau_2)}\cap B(x_i^0,d),
\end{equation}
where $t_{\epsilon}(v_M)$ and $l_{\epsilon}(\tau_2)$ are defined in \eqref{t es and l es}. 

For almost every $(x,v) \in B(x^0_i,d) \times \overline{B(0,v_M)}$, the trajectory $X_s(x,v)$ admits finitely many rebounds at times $$0<t_1<t_2<...<t_{r(x,v)} \leq t_{\epsilon}(v_M),$$
with corresponding reflection angles  $$\theta_i:=\frac{\pi}{2}-\arccos{|n(X_{t_i}(x,v))\cdot V_{t_i}(x,v)|} ,\, \forall 1 \leq i \leq r(x,v).$$  If $r(x,v)=0$, then we observe that $V_s(x,v)=v$ for all $s \in [0,t_{\epsilon}(v_M)]$. Therefore, we assume that $r(x,v)\geq 1$.

We notice that the estimate
\begin{equation}
    \forall s \in [0,t_{\epsilon}(v_M)],\ |V_s(x,v)-v| < \epsilon
\end{equation}
holds if
\begin{equation}
    \sum_{i=1}^{r(x,v)}2\sin{\theta_i} <\frac{\epsilon}{|v|}.
\end{equation}

Next, we show that for any $1 \leq i \leq r(x,v)-1$
\begin{equation}\label{2025/08/19 04:34}
    |X_{t_i}(x,v)-X_{t_{i-1}}(x,v)| \geq \frac{\sin{\theta_i}}{\tilde{C}(\Omega)}.
\end{equation}
We invoke the \textbf{Remark \ref{2025/08/19 02:54}} and use the uniform interior sphere condition to deduce that for any $0 \leq s \leq \frac{\sin{\theta_i}}{\tilde{C}(\Omega)}$, we have
\begin{equation}
    X_{t_i}(x,v)-s\frac{v}{|v|} \in B\left(X_{t_i}(x,v),\frac{1}{2\tilde{C}(\Omega)}\right) \subset \Omega,
\end{equation}
which directly implies \eqref{2025/08/19 04:34}.

Therefore, it suffices to show that
\begin{equation}\label{2025 09/02 03:42}
    \sum_{i=2}^{r(x,v)}|X_{t_i}(x,v)-X_{t_{i-1}}(x,v)| \leq \frac{\epsilon}{8|v|\tilde{C}(\Omega)},
\end{equation}
and that 
\begin{equation}\label{2025 09/02 03:43}
    \sin{\theta_1} <\frac{\epsilon}{4|v|}.
\end{equation}

The \eqref{2025 09/02 03:42} holds since we have
\begin{equation}
    \sum_{i=2}^{r(x,v)}|X_{t_i}(x,v)-X_{t_{i-1}}(x,v)| \leq t_{\epsilon}(v_M)|v| \leq \frac{\epsilon|v|}{8\tilde{C}(\Omega)v_M^2} \leq \frac{\epsilon}{8|v|\tilde{C}(\Omega)}.
\end{equation}
Here, we use the fact that $|v| \leq v_M$.
For the \eqref{2025 09/02 03:43}, we again use the uniform interior sphere condition to deduce that for any $y \in \partial \Omega$ and $w \in \mathbb{R}^3$ with $n(y) \cdot w <0$, $\theta_w:=\frac{\pi}{2}-\arccos{|n(y)\cdot w|}$, we have
\begin{equation}\label{2025/08/19 05:40}
    d\left(y+sw,\partial \Omega \right) \geq d\left(y+sw, \partial B\left(y-\frac{1}{2\tilde{C}(\Omega)}n(y),\frac{1}{2\tilde{C}(\Omega)}\right)  \right)\geq \frac{s\sin{\theta_w}}{2}|w|,
\end{equation}
for any $0 \leq s \leq \frac{\sin{\theta_w}}{2\tilde{C}(\Omega)|w|}$.

Now, suppose that $\sin{\theta_1} \geq \frac{\epsilon}{4|v|}$. We consider the following estimate
\begin{equation}\label{2025/08/19 05:39}
   \tau_2 \leq t_{\epsilon}(v_M):=\frac{\epsilon}{8\tilde{C}(\Omega)v_M^2} \leq \frac{1}{2\tilde{C}(\Omega)|v|} \times \frac{\epsilon}{4|v|} \leq \frac{\sin{\theta_1}}{2\tilde{C}(\Omega)|v|}.
\end{equation} Notice that if $r(x,v) \geq 2$, then $t_2 \geq \frac{\sin{\theta_1}}{\tilde{C}(\Omega)|v|} > t_{\epsilon}$, which is a contradiction. 

For the case $r(x,v)=1$, since $X_{t_1}(x,v) \in \partial\Omega $, we deduce from \eqref{2025/08/19 05:40} and \eqref{2025/08/19 05:39} , by setting $y:=X_{t_1}(x,v)$, that

\begin{equation}
d(X_{s}(x,v),\partial\Omega) \geq \frac{|s-t_1|\sin{\theta_1}}{8}|v|  \geq \frac{|s-t_1|\epsilon}{32},  
\end{equation}
for any $0\leq s \leq \tau_2$.
Hence, we obtain 
\begin{equation}
\sup_{0\leq s \leq \tau_2}d(X_{s}(x,v),\partial\Omega) \geq \sup_{0\leq s \leq \tau_2}\frac{|s-t_1|\epsilon}{32} \geq \frac{\frac{\tau_2}{2}\epsilon}{32}  =l_{\epsilon}(\tau_2),
\end{equation}
which contradicts \eqref{a small assumption in lemma}. We conclude the proof of \textbf{Lemma \ref{grazing geo lemma}}.

\end{proof}
    
Now, we introduce the following notation:
\begin{definition}
Given $x ,y \in \Omega$ and $d\in (0,1]$, we define
 \begin{equation}
    \operatorname{Conn}_d(x,y):=\min\{ N \mid \,  \textit{There exist a good zigzag with length N from x to y} \},
\end{equation} and
\begin{equation}\label{2025 06 09 02:40}
    \operatorname{Conn}_d(\Omega):=\max\{ \operatorname{Conn}_d(x,y) \mid x,y \in \Omega \}<+\infty.
\end{equation}
The last inequality is due to the fact that $\mathbb{Y}$ is finite.
\end{definition}

\section{Operator estimates and spreading properties}

\bigskip
We introduce a new constant $n_b:=\int_{\mathbb{S}^2}b(\cos{\theta})\,d\sigma$ and quote a key lemma (\textbf{Corollary 2.2}, \cite{Mou 1}), which is very useful throughout this article:

\begin{lemma}\label{up bound of L lemma}
    Given a measurable function g on $\mathbb{R}^3$, suppose that the collision operator satisfies \textbf{Assumption B} with $\nu<0$. Then, there exists $C_g^L>0$ which depends only on $n_b$, $C_{\Phi}$, $\rho_g$, and $e_g$ (and $l_{g,p}$, where $p>\frac{3}{3+\gamma}$, if $\gamma<0$) such that
\begin{equation}
    |L[g](v)| \leq C_g^L\left\langle v \right\rangle^{\gamma^+}.
\end{equation}
\end{lemma}
\begin{definition}\label{2025 10/03 03:50}
    We define $C_L:=\sup_{(t,x) \in [0,T] \times \mathbb{R}^3} C^L_{f(t,x,\cdot)}$.
\end{definition}
We also quote a very useful lemma which describes the spreading property of the operator $Q$ (\textbf{Lemma 2.4}, \cite{Mou 1}):

\begin{lemma}\label{spreading property lemma}
Suppose that the collision operator satisfies \textbf{Assumption B} with $\nu<0$. Then, there exists a constant $C>0$, which depends only on $\gamma,\ \nu$, and $\ b_0$, such that for any $v_0 \in \mathbb{R}^3$, $0<r \leq R$, $\xi \in (0,1)$, we have

\begin{equation}
    Q^+\left[\mathbf{1}_{B(v_0,R)},\mathbf{1}_{B(v_0,r)}\right] \geq C l_b c_{\Phi} R^{3+\gamma}\xi^{\frac{1}{2}}\mathbf{1}_{B(v_0,\sqrt{r^2+R^2}(1-\xi))},
\end{equation}
where $l_b:=\inf\limits_{\frac{1}{4}\pi\leq \theta \leq \frac{3}{4}\pi}b(\cos{\theta})$.

In particular, when $ r=R=\delta$, we have
\begin{equation}
    Q^+\left[\mathbf{1}_{B(v_0,\delta)},\mathbf{1}_{B(v_0,\delta)}\right] \geq C l_b c_{\Phi} \delta^{3+\gamma}\xi^{\frac{1}{2}}\mathbf{1}_{B(v_0,\delta\sqrt{2}(1-\xi))}.
\end{equation}
\end{lemma}
\noindent For simplicity, we define $C_Q:=C l_b c_{\Phi}$ 
\bigskip

Now, we will adapt an important lemma introduced in \textbf{Lemma 3.3} in \cite{Bri1}:

\begin{proposition} \label{initial point lower bound on a single point}
 Let $\Omega \subset \mathbb{R}^3$ satisfy \textbf{Assumption A} and let the collision kernel $B$ satisfy \textbf{Assumption B} with $\nu<0$. Suppose that $f(t,x,v)$ is a continuous mild solution to \eqref{Boltzmann equation}--\eqref{boundary condition} with $M>0$ and $0<E_f<\infty$.
Then, there exist $x_I \in \Omega$ with $\Delta^0\in \left(0,\;\min\left\{\frac{d}{10},\, \frac{1}{2}d(x_I,\partial\Omega)\right\}\right] $, $v_I \in \mathbb{R}^3$, $\alpha'_0>0$, which depend on $\Omega, M, E_f$ and the modulus of continuity of $f_0$, 
such that for all $t \in [0,\Delta^0]$, $x \in B(x_I,\Delta^0) \cap \Omega$, we have 

\begin{equation}\label{}
    \forall v \in \mathbb{R}^3,\ f(t,x,v) \geq \alpha'_0 \mathbf{1}_{ B(v_I,\Delta^0)}(v). 
\end{equation}

\end{proposition}

\begin{remark}
The assumption here is weaker than that in \textbf{Lemma 3.3} of \cite{Bri1}: we require 
$f(t,x,v)$ to be continuous only on $\Gamma_{conti}$, while \textbf{Lemma 3.3} of \cite{Bri1} assumes continuity on the larger set 
$[0,T) \times (\overline{\Omega} \times \mathbb{R}^3 - \Gamma_0)$.
\end{remark}

\begin{proof}

   The following proof is an adaptation of the proof of \textbf{Lemma 3.3} in \cite{Bri1}. 
   
   First, we construct a lower bound at a point in $\Omega$. Note that by translation, we can always assume that $0 \in \Omega$.
Since $\Omega$ is bounded, we have for any $t>0$

\begin{equation}\label{a simple estimate}
    \int_{\mathbb{R}^3}\int_{\Omega}(|x|^2+|v|^2)f(t,x,v)\,dx\,dv \leq M(\diam(\Omega))^2+|\Omega|E_f<\infty.
\end{equation}
Define $R_0:=\max\left\{1,\sqrt{\frac{2(M(\diam(\Omega))^2+|\Omega|E_f)}{M}}\right\}$. Notice that $\Omega \subset B(0,R_0)$. Then, we have
\begin{equation}
    \int_{B(0,R_0)}\int_{B(0,R_0)\cap \Omega}f_0(x,v)\,dx\,dv=\int_{B(0,R_0)}\int_{ \Omega}f_0(x,v)\,dx\,dv \geq \frac{M}{2}>0.
\end{equation}
Otherwise, we would have
\begin{equation}
\begin{split}
     &\int\int_{\Omega \times \mathbb{R}^3-(B(0,R_0) \cap \Omega)\times B(0,R_0)}(|x|^2+|v|^2)f_0(x,v)\,dx\,dv \\
     >& R_0^2\frac{M}{2}\geq M(\diam(\Omega))^2+|\Omega|E_f,
\end{split}
\end{equation}
which contradicts \eqref{a simple estimate}.

Hence, we can take a point $(x_I,v_I) \in \Omega \times B(0,R_0)$  such that $f(0,x_I,v_I)> \frac{M}{3|B(0,R_0)|| \Omega|}$. By the uniform continuity of $f_0(x,v)$ on $\overline{\Omega} \times \overline{B(0,2R_0)}$, there exist $0<\Delta_I< \min\left\{\frac{d(x_I,\partial \Omega)}{2},R_0\right\}$ such that $$f_0(x,v) \geq \frac{M}{6|B(0,R_0)|| \Omega|}$$ for any $(x,v) \in B(x_I,\Delta_I)  \times B(v_I,\Delta_I) $.

Next, we observe that for any $(t,x,v) \in \left[0,\frac{\Delta_I}{4R_0}\right]\times B\left(x_I,\frac{\Delta_I}{2}\right) \times B(v_I,\Delta_I)$, we have 
$$(x-tv,v) \in B(x_I,\Delta_I) \times B(v_I,\Delta_I) $$.
As a result, we deduce from \eqref{Duhamel formula} and  \textbf{Lemma \ref{up bound of L lemma}} that
\begin{equation}
\begin{split}
f(t,x,v)&\geq f_{0}(X_{0,t}(x,v),v)\exp\left(-\int_0^tL[f(s,X_{s,t}(x,v),\cdot)](v)\, ds\right)\\
&\geq f_{0}(x-tv,v)e^{-tC_L\left\langle v \right\rangle^{\gamma^+}}\\
&\geq \frac{M}{6|B(0,R_0)|| \Omega|}e^{-\frac{\Delta_IC_L}{4R_0}\left\langle 2R_0 \right\rangle^{\gamma^+}}
\end{split}
\end{equation} 
for any $(t,x,v) \in \left[0,\frac{\Delta_I}{4R_0}\right]\times B\left(x_I,\frac{\Delta_I}{2}\right)  \times B(v_I,\Delta_I)$, where the constant $C_L$ depends only on $n_b$,$C_{\Phi}$ and $E_f$  (and $L_{f,p}$ if $\gamma<0$).

We conclude the proof by defining $$\Delta^0:= \min\left\{ \frac{\Delta_I}{4R_0}, \frac{\Delta_I}{2}, \frac{d}{10}\right\}.$$

%We define $\tilde\Delta^0:= \frac{d(x_I,\partial\Omega)}{8R_0}$. By continuity of $f(t,x,v)$ on $[0,\tilde\Delta^0] \times \overline{B(x_I, \frac{d(x_I,\partial\Omega)}{4})} \times \overline{B(v_I,R_0)} \subset \Gamma_{conti}$, there exists $\Delta^0\in (0,\min\{\frac{d}{10}, \frac{1}{2}d(x_I,\partial\Omega)]$ which depends on the modulus of continuity of $f_0$ such that $B(x_I,\Delta^0) \subset \Omega$ and

%\begin{equation}
%    f(t,x,v) \geq \frac{M}{4|B(0,R_0)||B(0,R_0) \cap \Omega|},\ \forall (t,x,v) \in [0,\Delta^0] \times B(x_I,\Delta^0) \times B(v_I,\Delta^0).
%\end{equation} 

\end{proof}

To demonstrate the fact that given an initial lower bound around a point $x \in \Omega$, multiple diluting lower bounds can be generated at the same point, we quote \textbf{Lemma 3.3}, \cite{Bri1}:
\begin{proposition} \label{ini to multi}
    Suppose that $\Omega \subset \mathbb{R}^3$ satisfies \textbf{Assumption A}. Let kernel $B$ satisfy \textbf{Assumption B} with $\nu<0$ and $\alpha \in [0,1]$.  We consider a continuous mild solution $f$ of \eqref{Boltzmann equation}--\eqref{boundary condition}. Suppose that there exist $A > 0$, $\Delta_1,\ \Delta_2>0$, $(\tau, x' ,v') \in [0, T) \times \Omega \times \mathbb{R}^3$ such that $B(x',\Delta_2) \subset \Omega$ and
    \begin{equation}\label{ini to multi assume eq}
        f(t,x,v) \geq A,\ \forall (t,x,v) \in [\tau, \tau+\Delta_1] \times B(x',\Delta_2) \times B(v', \Delta_2).
    \end{equation}
    Then, we have for $n \in \mathbb{N}\cup \{ 0 \}$, $t \in [\tau, \tau+\Delta_1]$, $x \in B(x',\frac{\Delta_2}{2^n}) $, 
\begin{equation}\label{}
     f(t,x,v) \geq \alpha_n(\tau,t,\Delta_2,A,|v'|)\mathbf{1}_{ B(v',r_n(\Delta_2))}(v), \,\forall \, v \in \mathbb{R}^3,
\end{equation}
where the numbers $\{r_n(\Delta_2)\}_{n=0}^{\infty} \in \mathbb{R}$, $\{t_n(t,\Delta_2,|v'|)\}_{n=0}^{\infty}\in \mathbb{R}$ and $\{\alpha_n(\tau, t,\Delta_2,A,|v'|)\}_{n=1}^{\infty}\in \mathbb{R}$ are defined as below:
\begin{align*}
&r_0(\Delta_2):=\Delta_2,\ r_{n+1}(\Delta_2):=\frac{3\sqrt{2}}{4}r_{n}(\Delta_2),\\    
&t_n(\tau,t,\Delta_2,|v'|):=\max\left\{ \tau, t-\frac{\Delta_2}{2^{n+1}(2r_n(\Delta_2)+|v'|)} \right\},\\
 &   \alpha_0:=A,
\end{align*}
\begin{equation}\label{alpha second line}
\begin{split}
&\alpha_{n+1}(\tau,t,\Delta_2,A,|v'|)\\
:=&\frac{C_Q}{2}r_n^{3+\gamma}(\Delta_2)\int_{t_n(\tau,t,\Delta_2,|v'|)}^{t}e^{-(t-s)C_L\left\langle2r_n(\Delta_2)+|v'|\right\rangle^{\gamma^{+}}}\alpha_n(\tau,s,\Delta_2,A,|v'|)^2ds\\
=&\frac{C_Q}{2}r_n^{3+\gamma}(\Delta_2)\int_{0}^{t-t_n(\tau,t,\Delta_2,|v'|)}e^{-uC_L\left\langle2r_n(\Delta_2)+|v'|\right\rangle^{\gamma^{+}}}\alpha_n(\tau,t-u,\Delta_2,A,|v'|)^2du.
\end{split}
\end{equation}
    
\end{proposition}
\begin{proof}
    %We refer to \textbf{Lemma 3.3}, \cite{Bri1}.
     The following proof is also a summary of the proof of \textbf{Lemma 3.3} in \cite{Bri1}. 
We prove the proposition by induction on $n$. The case $n=0$ is exactly the assumption.  Assume that \textbf{Proposition \ref{ini to multi}} holds for $n=k$. Given $t \in [\tau,\tau+\Delta_1]$, $x \in B(x', \frac{\Delta_2}{2^{k+1}})$, $v \in B(0,|v'|+2r_{k}(\Delta_2))$, we first notice that when $s \in \left[\max\left\{ \tau, t-\frac{\Delta_2}{2^{k+1}(|v'|+2r_k(\Delta_2))} \right\},t\right]$, we have
\begin{equation*}
    |x'-X_{s,t}(x,v)|=|x'-x+tv-sv|\leq \frac{\Delta_2}{2^{k+1}}+|t-s|(|v'|+2r_k(\Delta_2)) \leq  \frac{\Delta_2}{2^{k}},
\end{equation*}
which implies that $X_{s,t}(x,v) \in B(x', \frac{\Delta_2}{2^k}) \subset \Omega$.

We consider the second term of the right hand side of \eqref{Duhamel formula} to obtain the following lower bound:

\begin{small}
\begin{equation}
    \begin{split}
        &f(t,x,v) \\
         \geq &\int_{\tau}^t \exp\left(-\int_s^t L[f(s',X_{s',t}(x,v),\cdot)](v))\,ds'\right)Q^+[f(s,X_{s,t}(x,v),\cdot),f(s,X_{s,t}(x,v),\cdot)](v)\,ds.
    \end{split}
\end{equation}
\end{small}

Furthermore, by \textbf{Lemma \ref{up bound of L lemma}}, we have

\begin{equation}
|L[f(s',X_{s',t}(x,v),\cdot)](v)| \leq C_L\langle  v \rangle^{\gamma^+} \leq C_L\left\langle |v'|+2r_k(\Delta_2) \right\rangle^{\gamma^+},
\end{equation}
where the constant $C_L$ depends only on $n_b$,$C_{\Phi}$ and $E_f$  (and $L_{f,p}$ if $\gamma<0$).

Hence, the following estimate:

\begin{small}
\begin{equation}
    \begin{split}
        &f(t,x,v) \\
         \geq &\int_{\tau}^t \exp\left(-(t-s)C_L\left\langle |v'|+2r_k(\Delta_2) \right\rangle^{\gamma^+}\right)Q^+[f(s,X_{s,t}(x,v),\cdot),f(s,X_{s,t}(x,v),\cdot)](v)\,ds.
    \end{split}
\end{equation}
\end{small}

Since we have $|X_{s,t}(x,v)-x'| \leq \frac{\Delta_2}{2^{k}}$, we can use the induction hypothesis:

\begin{equation}
    f(s,X_{s,t}(x,v),w) \geq \alpha_k(\tau,s,\Delta_2,A,|v'|)1_{ B(v',r_k(\Delta_2))}(w) ,\ \forall w \in \mathbb{R}^3 
\end{equation}
to deduce that

\begin{equation}
    \begin{split}
        &f(t,x,v) \\
         \geq& \int_{\max\left\{ \tau, t-\frac{\Delta_2}{2^{k+1}(2r_k(\Delta_2)+|v'|)} \right\}}^t \exp\left(-(t-s)C_L\left\langle |v'|+2r_k(\Delta_2)\right\rangle^{\gamma^+}\right)\\
        &\alpha^2_k(\tau,s,\Delta_2,A,|v'|)Q^+[1_{B(v',r_k(\Delta_2))}(\cdot),1_{B(v',r_k(\Delta_2))}(\cdot)](v)\,ds.
    \end{split}
\end{equation}

Then we use \textbf{Lemma \ref{spreading property lemma}} to "spread" the lower bound:

For any $\xi \in (0,1)$ and $w \in \mathbb{R}^3$
\begin{equation}
    Q^+\left[1_{B(v',r_k(\Delta_2))},1_{B(v',r_k(\Delta_2))}\right](w)\geq C_Q (r_k(\Delta_2))^{3+\gamma}\xi^{\frac{1}{2}}1_{B(v',r_k(\Delta_2)\sqrt{2}(1-\xi))}(w).
\end{equation}
Hence, for any $\xi \in (0,1)$ 
\begin{equation}
    \begin{split}
        &f(t,x,v) \\
         \geq& \int_{\max\left\{ \tau, t-\frac{\Delta_2}{2^{k+1}(2r_k(\Delta_2)+|v'|)} \right\}}^t \exp\left(-(t-s)C_L\left\langle |v'|+2r_k(\Delta_2) \right\rangle^{\gamma^+}\right)\\
        &\alpha^2_k(\tau, s,\Delta_2,A,|v'|)C_Q (r_k(\Delta_2))^{3+\gamma}\xi^{\frac{1}{2}}1_{B(v',r_k(\Delta_2)\sqrt{2}(1-\xi))}(v)\,ds.
    \end{split}
\end{equation}

Set $\xi =\frac{1}{4}$ and notice that $B(v',r_{k+1}(\Delta_2))) \subset B(0,|v'|+2r_{k}(\Delta_2))$. We conclude thus the proof of \textbf{Proposition \ref{ini to multi}}.

\end{proof}
\begin{remark}\label{alpha monotone thing}
    Notice that by the last line of \eqref{alpha second line} and the fact that for any $\tau < t < t'< \infty$
    \begin{align*}
        &t-t_n(\tau,t,\Delta_2,|v'|)\leq t'-t_n(\tau,t',\Delta_2,|v'|).
    \end{align*}
    Hence, we notice that by induction on $n$, we can show that when $\tau<t$, $\alpha_{n+1}(\tau,t,\Delta_2,A,|v'|)$ is strictly decreasing with respect to $|v'|$ and strictly increasing with respect to $t$. 
\end{remark}

Now, we show that we can generate lower bounds starting from one point $x' \in \Omega$ and reaching another point $y \in \Omega$, as long as $\overline{x'y} \subset \Omega$. Without loss of generality, we may assume that $\Delta_2 \leq 2$.

\begin{proposition} \label{translation proposition}
   Let $\Omega \subset \mathbb{R}^3$ satisfy \textbf{Assumption A} and let the collision kernel $B$ satisfy \textbf{Assumption B} with $\nu<0$, $\alpha \in [0,1]$. We consider a continuous mild solution $f$ of \eqref{Boltzmann equation}--\eqref{boundary condition}. Suppose that there exists $A > 0$, $\Delta_1, \Delta_2 \in (0,1]$, $(\tau, x' ,v') \in [0, T) \times \Omega \times \mathbb{R}^3$ with $B(x',\Delta_2) \subset \Omega$ such that

    \begin{equation*}
        f(t,x,v) \geq A,\ \forall (t,x,v) \in [\tau, \tau+\Delta_1] \times B(x',\Delta_2) \times B(v', \Delta_2).
    \end{equation*}
Then, given $\tau' \in (\tau,\tau+\Delta_1)$, $y \in \Omega$ with $\overline{x'y} \subset \Omega$, we have that $B\left(y, \frac{\Delta_2}{2^{m+1}}\right) \subset \Omega$. Moreover, for all $x \in \Omega$, $v \in \mathbb{R}^3$
\begin{align}\label{P2.3 main eq}
    f(t,x,v) \geq \mathbb{B}'(m,\tau,\tau',\Delta_2,A,|v'|) \mathbf{1}_{B\left(y,\frac{\Delta_2}{2^{m+1}}\right) \times B\left(\mathbb{V}'(\tau,\tau',y,x'), \frac{\Delta_2}{2^{m+1}}\right)}(x,v)
\end{align}
for any $m \geq m'(\tau,\tau',\Delta_2,|v'|,d(\overline{x'y},\partial \Omega))$ and $\tau' \leq t \leq \min\{ \tau+\Delta_1,\tau' +\mathbb{D}'(m,\tau,\tau',\Delta_2)\}$.

Here, the functions $m'(\tau,\tau', \Delta_2,\mathfrak{y}_1,\mathfrak{y}_2)$, $\mathbb{B}'(m,\tau,\tau',\Delta_2,A,\mathfrak{y}_3)$, $\mathbb{V}'(\tau,\tau',y,x')$, and $\mathbb{D}'(m,\tau,\tau',\Delta_2)$ are defined as

\begin{align*}
    &m'(\tau,\tau',\Delta_2,\mathfrak{y}_1,\mathfrak{y}_2):=\max\left\{1,\left\lceil \log_{\frac{3\sqrt{2}}{4}}\left(\frac{\frac{2d_{\Omega}}{\tau'-\tau}+\mathfrak{a}_1+1}{\Delta_2}\right)\right\rceil 
,\left\lceil \log_{2}{\frac{\Delta_2+1}{\mathfrak{y}_2}} \right\rceil \right\},\\
    &\mathbb{B}'(m,\tau,\tau',\Delta_2,A,\mathfrak{y}_3):= \alpha_{m}\left(\tau,\frac{\tau'+\tau}{2},\Delta_2, A,\mathfrak{y}_3\right) e^{-C_L\left \langle \frac{2d_{\Omega}}{\tau'-\tau}+1\right \rangle ^{\gamma^+}},\\
&\mathbb{V}'(\tau,\tau',y,x'):=\frac{2(y-x')}{\tau'-\tau},\\
&\mathbb{D}'(m,\tau,\tau',\Delta_2):=\frac{\Delta_2}{2^{m+1}\left(\frac{2d_{\Omega}}{\tau'-\tau}+1\right)}=\frac{\Delta_2(\tau'-\tau)}{2^{m+1}(2d_{\Omega}+\tau'-\tau)}.
\end{align*}
Here, $C_L$ is defined in \textbf{Definition \ref{2025 10/03 03:50}}, $\alpha_{m}\left(\tau,\frac{\tau'+\tau}{2},\Delta_2, A,\mathfrak{y}_3\right)$ is as defined in \textbf{Proposition \ref{ini to multi}}, and we denote $d_{\Omega}:=\diam(\Omega)$ as the diameter of $\Omega$.

\end{proposition}
 
\begin{proof}

For $0\leq\tau<\tau'< \tau+\Delta_1$, we define $v'':=\frac{2(y-x')}{\tau'-\tau}$ and take $m \in \mathbb{N} \cup \{ 0 \}$ with $$m \geq \max \left \{1,  \left\lceil \log_{\frac{3\sqrt{2}}{4}}\left(\frac{\frac{2d_{\Omega}}{\tau'-\tau}+|v'|+1}{\Delta_2}\right)\right\rceil 
, \left\lceil \log_{2}{\frac{\Delta_2+1}{d(\overline{x'y},\partial \Omega)}} \right\rceil \right \},$$ then we see that $$r_m(\Delta_2):=(\frac{3\sqrt{2}}{4})^m\Delta_2 \geq \frac{2d_{\Omega}}{\tau'-\tau}+|v'|+1 $$ and $ \frac{\Delta_2}{2^m} <d(\overline{x'y},\partial\Omega)$. Here, the number $r_m(\Delta_2)$ is as defined in \textbf{Proposition \ref{ini to multi}} and we notice that $\frac{3\sqrt{2}}{4}>1$. 
Observe that we have
\begin{align*}
\frac{\Delta_2}{2^{m+1}} \leq  \frac{1}{2}d(\overline{x'y},\partial \Omega)<d(y,\partial \Omega),
\end{align*}
so that $B\left(y, \frac{\Delta_2}{2^{m+1}}\right) \subset \Omega$.

Now, we prove the lower bound \eqref{P2.3 main eq}.
As long as $x \in B\left(y,\frac{\Delta_2}{2^{m+1}}\right)$ and $v \in B\left(v'', \min \left\{ \frac{\Delta_2}{2^{m+1}(t-\tau)},1 \right\}\right)$, for $$\tau' \leq t \leq \min \left\{ \tau+\Delta_1, \tau'+   \frac{\Delta_2}{2^{m+1}\left(\frac{2d_{\Omega}}{\tau'-\tau}+1\right)}\right\},$$ we have

\begin{equation*}
\begin{split}
       & \left|x-\frac{t-\tau}{2}v-x'\right| \\
    \leq& |x-y|+\left|\frac{t-\tau}{2}v-(y-x')\right| < \frac{\Delta_2}{2^{m+1}}+\left|\frac{t-\tau}{2}v-\frac{\tau'-\tau}{2}v''\right|\\
    \leq& \frac{\Delta_2}{2^{m+1}}+ \left|\frac{t-\tau}{2}v-\frac{\tau'-\tau}{2}v\right|+ \left|\frac{\tau'-\tau}{2}v-\frac{\tau'-\tau}{2}v''\right|\\
    < &\frac{\Delta_2}{2^{m+1}}+ \left|\frac{t-\tau'}{2}v\right|+ \frac{\Delta_2}{2^{m+2}}\\
    < &\frac{\Delta_2}{2^{m+1}}+\frac{\Delta_2}{2^{m+2}}+ \frac{\Delta_2}{2^{m+2}}\\
   = &\frac{\Delta_2}{2^{m}},
\end{split}
\end{equation*}
by which we deduce that
$x-\frac{t-\tau}{2}v\in B\left(x',\frac{\Delta_2}{2^m}\right)$.
Notice that we have 
\begin{equation*}
r_m(\Delta_2) \geq |v'| +\frac{2d_{\Omega}}{\tau'-\tau}+1\geq |v'|+|v''|+1 > |v'|+|v|,
\end{equation*}
so we have $v \in B(v', r_m(\Delta_2))$.

Note that for any $0 \leq s \leq 1$, we have 
\begin{equation*}
\begin{split}
    &|sx+(1-s)\left(x-\frac{t-\tau}{2}v\right)-(sy+(1-s)x')|\\
     \leq& s|x-y|+(1-s)\left|x-\frac{t-\tau}{2}v-x'\right|\\
     \leq& \frac{\Delta_2}{2^{m}} < d(\overline{x'y},\partial\Omega).
\end{split}
\end{equation*}

Hence, the segment $\{ x-\sigma v|\ 0 \leq \sigma \leq \frac{t-\tau}{2} \}$ also lies within $\Omega$. From this, we deduce by using the first term of the right hand side of \eqref{Duhamel formula},   \textbf{Lemma \ref{up bound of L lemma}}, and \textbf{Proposition \ref{ini to multi}} that
\begin{equation}\label{Long repeat 2}
\begin{split}
f(t,x,v)&\geq f\left (\frac{t+\tau}{2},x-\frac{t-\tau}{2}v,v\right )e^{-\frac{t-\tau}{2}C_L\left \langle v\right \rangle ^{\gamma^+}},\\
&\geq \alpha_m\left(\tau,\frac{t+\tau}{2},\Delta_2, A,|v'|\right)\mathbf{1}_{B\left(x',\frac{\Delta_2}{2^m}\right) \times B(v',r_m(\Delta_2))}(x-\frac{t-\tau}{2}v,v) e^{-\frac{t-\tau}{2}C_L\left \langle v\right \rangle ^{\gamma^+}}\\
&= \alpha_m\left(\tau,\frac{t+\tau}{2},\Delta_2, A,|v'|\right) e^{-\frac{t-\tau}{2}C_L\left \langle  
v \right \rangle ^{\gamma^+}},
\end{split}
\end{equation}
where we use the fact that
$x-\frac{t-\tau}{2}v\in B(x',\frac{\Delta_2}{2^m})$ and $v \in B(v', r_m(\Delta_2))$.
Hence, we have the following inequality for $\tau' \leq t \leq \min\bigg\{\tau+\Delta_1, \tau' +\frac{\Delta_2}{2^{m+1}(\frac{2d_{\Omega}}{\tau'-\tau}+1)}\bigg\}$, $x \in \Omega$, $v \in \mathbb{R}^3$ :

\begin{align}
    &\nonumber f(t,x,v) \\ 
    \nonumber  \geq &\alpha_m\left(\tau,\frac{t+\tau}{2},\Delta_2, A,|v'|\right) e^{-\frac{t-\tau}{2}C_L\left \langle v\right \rangle ^{\gamma^+}} \mathbf{1}_{B\left(y,\frac{\Delta_2}{2^{m+1}}\right) \times B\left(v'', \min \left\{ \frac{\Delta_2}{2^{m+1}(t-\tau)},1\right\} \right)}(x,v)\\ 
      \label{SSS} \geq &\alpha_m\left(\tau,\frac{\tau'+\tau}{2},\Delta_2, A,|v'|\right) e^{-C_L\left \langle \frac{2d_{\Omega}}{\tau'-\tau}+1\right \rangle ^{\gamma^+}} \mathbf{1}_{B\left(y,\frac{\Delta_2}{2^{m+1}}\right) \times B\left(v'', \frac{\Delta_2}{2^{m+1}}\right)}(x,v).   
\end{align}
Here, we use \textbf{Remark \ref{alpha monotone thing}} in the last inequality above.
The line \eqref{SSS} corresponds to \eqref{P2.3 main eq}. Thus, we have shown the statement of \textbf{Proposition \ref{translation proposition}}.

\end{proof}

The following corollary can be derived by modifying the starting time $\tau'$ in \textbf{Proposition \ref{translation proposition}}:

\begin{corollary} \label{redoable translation proposition}
    Let $\Omega \subset \mathbb{R}^3$ satisfy \textbf{Assumption A} and let the collision kernel $B$ satisfy \textbf{Assumption B} with $\nu<0$, $\alpha \in [0,1]$. We consider a continuous mild solution $f$ of \eqref{Boltzmann equation}--\eqref{boundary condition}. Suppose that there exists $A > 0$, $\Delta_1, \Delta_1',  \Delta_2 \in (0,\frac{1}{2}]$, $(\tau, x' ,v') \in [0, T) \times \Omega \times \mathbb{R}^3$ with $B(x',\Delta_2) \subset \Omega$ such that
    \begin{equation}\label{CORO2.4 LOW BOUND}
        f(t,x,v) \geq A,\ \forall (t,x,v) \in [\tau, \tau+\Delta_1+\Delta_1'] \times B(x',\Delta_2) \times B(v', \Delta_2).
    \end{equation}
Then, given $y \in \Omega$ such that $\overline{x'y} \subset \Omega$, we have $B(y,\frac{\Delta_2}{2^{m+1}}) \subset \Omega$. Furthermore, for any $x \in \Omega, v \in \mathbb{R}^3$
\begin{align*}
    f(t,x,v) \geq \mathbb{B}(m,\tau,\Delta_1,\Delta_2,A,|v'|) \mathbf{1}_{B(y,\frac{\Delta_2}{2^{m+1}}) \times B\left(\mathbb{V}(\Delta_1,\Delta_2,\tau,m,y,x'), \frac{\Delta_2}{2^{m+1}}\right)}(x,v)
\end{align*}
for any $m \geq \mathbb{M}(\Delta_1,\Delta_2,|v'|,d(\overline{x'y},\partial \Omega))$ and $\mathbb{T}(\Delta_1,\Delta_2,\tau,m) \leq t \leq \min \{\tau+\Delta_1+\Delta_1' ,\tau+ \Delta_1+\mathbb{D}(\Delta_1,\Delta_2,m)\}$.

Here, the functions $\mathbb{M}(\Delta_1,\Delta_2,\mathfrak{y}_1,\mathfrak{y}_2)$, $\mathbb{B}(m,\tau,\tau',\Delta_2,A,\mathfrak{y}_3)$, $\mathbb{D}(\Delta_1,\Delta_2,m)$, $\mathbb{T}(\Delta_1,\Delta_2,\tau,m)$, and $\mathbb{V}(\Delta_1,\Delta_2,\tau,m,y,x')$ are defined as

\begin{align*}
    &\mathbb{M}(\Delta_1,\Delta_2,\mathfrak{y}_1,\mathfrak{y}_2):=\max\left\{1,\left\lceil \log_{\frac{3\sqrt{2}}{4}}\left(\frac{\frac{4d_{\Omega}}{\Delta_1}+\mathfrak{y}_1+1}{\Delta_2}\right)\right\rceil 
,\left\lceil \log_{2}{\frac{\Delta_2+1}{\mathfrak{y}_2}} \right\rceil  \right\},\\
    &\mathbb{B}(m,\tau,\Delta_1,\Delta_2,A,\mathfrak{y}_3):= \alpha_m\left(\tau,\tau+\frac{\Delta_1}{4},\Delta_2, A,\mathfrak{y}_3\right) e^{-C_L\left \langle \frac{4d_{\Omega}}{\Delta_1}+1\right \rangle ^{\gamma^+}},\\
&\mathbb{D}(\Delta_1,\Delta_2,m):=\frac{\Delta_1\Delta_2}{2^{m+2}(4d_{\Omega}+\Delta_1)},\\
&\mathbb{T}(\Delta_1,\Delta_2,\tau,m):=\max\left\{\tau+\frac{\Delta_1}{2},\Delta_1+\tau-\frac{\Delta_1\Delta_2}{2^{m+2}(4d_{\Omega}+\Delta_1)} \right\},\\
&\mathbb{V}(\Delta_1,\Delta_2,\tau,m,y,x'):=\frac{2(y-x')}{\mathbb{T}(\Delta_1,\Delta_2,\tau,m)-\tau}=\frac{2(y-x')}{\max\left\{\frac{\Delta_1}{2},\Delta_1-\frac{\Delta_1\Delta_2}{2^{m+2}(4d_{\Omega}+\Delta_1)} \right\}}.
\end{align*}

\end{corollary}

\begin{remark}
    By \textbf{Remark \ref{alpha monotone thing}}, we notice that 
    \begin{enumerate}
    \item $\mathbb{B}(m,\tau,\Delta_1,\Delta_2,A,\mathfrak{y}_3)$ is strictly decreasing in $\mathfrak{y}_3$. 
    \item  $\mathbb{M}(\Delta_1,\Delta_2,\mathfrak{y}_1,\mathfrak{y}_2)$ is increasing in $\mathfrak{y}_1$ and decreasing in $\mathfrak{y}_2$.
    \item  $\mathbb{T}(\Delta_1,\Delta_2,\tau,m)$ is increasing in $m$.
    \item  $|\mathbb{V}(\Delta_1,\Delta_2,\tau,m,y,x')|$ is decreasing in $m$ with the upper bound $\frac{4d_{\Omega}}{\Delta_1}$. 
    \end{enumerate}
\end{remark}

Now, we will use the multiple lower bounds around the initial point $x_I$ from \textbf{Proposition \ref{initial point lower bound on a single point}} to derive lower bounds around the points $\{ y^0_i \}$. 

\begin{proposition} \label{initial point lower bound near boundary}
 Let $\Omega \subset \mathbb{R}^3$ satisfy \textbf{Assumption A} and let the collision kernel $B$ satisfy \textbf{Assumption B} with $\nu<0$. Fix $d$ with $0< d <\min\{1,\delta\}$ as in \textbf{Lemma \ref{initial cover over boundary}} and a collection of points $\{ y_i^0 \}_{i=1}^{m_1+m_2} \subset \Omega$ defined in \textbf{Lemmas \ref{initial cover over boundary} and \ref{initial cover general}}.
 Let $x_I$ be as defined in \textbf{Proposition \ref{initial point lower bound on a single point}}. 

Suppose that, for fixed $\alpha \in [0,1]$, there exists a continuous mild solution to \eqref{Boltzmann equation}--\eqref{boundary condition}.
 Then, for any $\tau'' \in (0,\Delta^0)$, with $\Delta^0 \in (0,\min\{\frac{d}{10}, \frac{1}{2}d(x_I,\partial\Omega)\}]$ as in \textbf{Proposition \ref{initial point lower bound on a single point}}, there exist positive numbers $\mathfrak{R}(\tau'')$, $\mathfrak{B}(\tau'')$, and $\mathfrak{d}(\tau'')$, depending on $\tau''$, $\Omega$, $M$, and $E_f$ (and $L_{f,p}$ if $\gamma <0$), and a family of velocity vectors $$\{ v_i(\tau'') \}_{i=1}^{m_1+m_2} \subset B(0,\mathfrak{R}(\tau''))$$ 
such that $$\bigcup_{i=1}^{m_1+m_2}B\left(y_i^0,\mathfrak{d}(\tau'') \right) \subset \Omega.$$ Moreover, for each $1\leq i \leq m_1+m_2$, we have

\begin{equation}\label{Prop 2.2 eq}
    f(t,x,v) \geq \ \mathfrak{B}(\tau'')\mathbf{1}_{ B\left(y_i^0 ,\mathfrak{d}(\tau'') \right) \times B\left(v_i(\tau''),\mathfrak{d}(\tau'')\right)}(x,v), 
\end{equation}
for all $$\tau''-\mathfrak{d}(\tau'') \leq t \leq \tau''+\mathfrak{d}(\tau'').$$
Here, we use the same radius $\mathfrak{d}(\tau'')$ for the space and velocity balls and for the time window.
\end{proposition}

\begin{proof}

First, by \textbf{Proposition \ref{initial point lower bound on a single point}}, we have
\begin{equation}
    \forall v \in \mathbb{R}^3,\ f(t,x,v) \geq \alpha'_0 \mathbf{1}_{ B(v_I,\Delta^0)}(v),
\end{equation}
 for all $t \in [0,\Delta^0]$, $x \in B(x_I,\Delta^0)$. Here, the quantities $\alpha'_0$, $x_I$, $v_I$, $\Delta^0$ are defined in \textbf{Proposition \ref{initial point lower bound on a single point}}. Note that $B(x_I,\Delta^0) \subset \Omega$.
Consider a good zigzag $\{x_I,y_1,y_2,...,y_{\operatorname{Conn}_d(x_I,y_i^0)}=y_i^0\}$ in $\Omega$ (cf. \textbf{Lemma \ref{zigzag lemma}}).

Given $0<\tau''< \Delta^0$, for any $j \in \mathbb{N}$, and a $j$-tuple $$\mathfrak{M}_j:=(\mathfrak{m}_1,\mathfrak{m}_2,...,\mathfrak{m}_j) \in \mathbb{N}^j,$$ we define $ \mathbb{B}_{j}(\mathfrak{M}_j,\tau'',\Delta^0,\alpha'_0,|v_I|)$, $\mathbb{T}_{j}(\mathfrak{M}_j,\tau'',\Delta^0)$, $\mathbb{D}_{j}(\mathfrak{M}_j,\tau'',\Delta^0)$, $\mathbb{V}_{j}(\mathfrak{M}_j,\tau'',\Delta^0)$, $\mathbb{U}_{j}(\mathfrak{M}_j,\tau'',\Delta^0)$,
$\mathfrak{n}_{j}$ as
\begin{align*}
    &\mathbb{T}_{1}(\mathfrak{M}_1,\tau'',\Delta^0):=\mathbb{T}(\tau'',\Delta^0,0,\mathfrak{m}_1),\\ 
&\mathbb{D}_1(\mathfrak{M}_1,\tau'',\Delta^0):=\mathbb{D}(\tau'',\Delta^0,\mathfrak{m}_1),\\
&\mathbb{V}_1(\mathfrak{M}_1,\tau'',\Delta^0):=\mathbb{V}(\tau'',\Delta^0,0,\mathfrak{m}_1,y_1,x_I),\\
    &\mathbb{B}_{1}(\mathfrak{M}_1,\tau'',\Delta^0,\alpha'_0,|v_I|):=\mathbb{B}(\mathfrak{m}_1,0,\tau'',\Delta^0,\alpha'_0,|v_I|),\\
   & \mathbb{U}_{1}:=\Delta^0.
\end{align*}
\begin{align*}   
&\mathbb{T}_{j+1}(\mathfrak{M}_{j+1},\tau'',\Delta^0):=\mathbb{T}\left(\tau''-\mathbb{T}_{j}(\mathfrak{M}_j,\tau'',\Delta^0),\frac{\Delta^0}{2^{\sum_{k=1}^{j}(\mathfrak{m}_k+1)}},\mathbb{T}_{j}(\mathfrak{M}_j,\tau'',\Delta^0),\mathfrak{m}_{j+1}\right),\\
&
\mathbb{D}_{j+1}(\mathfrak{M}_{j+1},\tau'',\Delta^0)=\mathbb{D}\left(\tau''-\mathbb{T}_{j}(\mathfrak{M}_j,\tau'',\Delta^0),\frac{\Delta^0}{2^{\sum_{k=1}^{j}(\mathfrak{m}_k+1)}},\mathfrak{m}_{j+1}\right),\\
    &\mathbb{V}_{j+1}(\mathfrak{M}_{j+1},\tau'',\Delta^0):=\mathbb{V}\left(\tau''-\mathbb{T}_j(\mathfrak{M}_j,\tau'',\Delta^0),\frac{\Delta^0}{2^{\sum_{k=1}^{j}(\mathfrak{m}_k+1)}},\mathbb{T}_{j}(\mathfrak{M}_j,\tau'',\Delta^0),\mathfrak{m}_{j+1},y_{j+1},y_j\right)\\
    &\hspace{3.3cm}=\frac{2(y_{j+1}-y_j)}{\mathbb{T}_{j+1}(\mathfrak{M}_{j+1},\tau'',\Delta^0)-\mathbb{T}_{j}(\mathfrak{M}_j,\tau'',\Delta^0)},\\
   &\mathbb{B}_{j+1}(\mathfrak{M}_{j+1},\tau'',\Delta^0,\alpha'_0,|v_I|)\\
   &\hspace{1cm}:=\mathbb{B}\Bigg(\mathfrak{m}_{j+1},\mathbb{T}_j(\mathfrak{M}_j,\tau'',\Delta^0),\tau''-\mathbb{T}_j(\mathfrak{M}_j,\tau'',\Delta^0),\frac{\Delta^0}{2^{\sum_{k=1}^{j}(\mathfrak{m}_k+1)}},\\
&\hspace{2.2cm}\mathbb{B}_j(\mathfrak{M}_j,\tau'',\Delta^0,\alpha'_0,|v_I|),\frac{2d_{\Omega}}{\mathbb{T}_{j}(\mathfrak{M}_{j},\tau'',\Delta^0)-\mathbb{T}_{j-1}(\mathfrak{M}_{j-1},\tau'',\Delta^0)}\Bigg),\\
&\mathbb{U}_{j+1}(\mathfrak{M}_{j+1},\tau'',\Delta^0):=\min\{\mathbb{U}_{j}(\mathfrak{M}_j,\tau'',\Delta^0),\tau'' +\mathbb{D}_{j}(\mathfrak{M}_j,\tau'',\Delta^0)\}.
\end{align*}
Here, the notation $\mathbb{B}$, $\mathbb{T}$, $\mathbb{D}$, $\mathbb{V}$ are defined in
\textbf{Corollary \ref{redoable translation proposition}}. We also define $\mathbb{T}_{0}:=0$ by convention. By \textbf{Remark \ref{ineq TDU}}, we have $\mathbb{T}_j(\mathfrak{M}_j,\tau'',\Delta^0)< \tau''$ and that all denominators $\mathbb{T}_j(\mathfrak{M}_j,\tau'',\Delta^0)-\mathbb{T}_{j-1}(\mathfrak{M}_j-1,\tau'',\Delta^0)$ are positive.
We also define 
\begin{align*}
 &\mathfrak{n}_{1} := \max\left\{1,\left\lceil \log_{\frac{3\sqrt{2}}{4}}\left(\frac{\frac{4d_{\Omega}}{\tau''}+|v_I|+1}{\Delta^0}\right)\right\rceil 
,\left\lceil \log_{2}{\frac{\Delta^0+1}{\min\{d(x_I,\partial\Omega),\frac{d}{10}\}}} \right\rceil  \right\}  ,\\
&\mathfrak{N}_j(m):=(\mathfrak{n}_1,\mathfrak{n}_2,...,\mathfrak{n}_{j-1},m), \,\mathfrak{N}_j:=\mathfrak{N}_j(\mathfrak{n}_j),\\
&\mathfrak{n}_{j+1} := \max\Bigg\{1,\left\lceil \log_{\frac{3\sqrt{2}}{4}}\left(\frac{\frac{4d_{\Omega}}{\tau''-\mathbb{T}_{j}(\mathfrak{N}_j,\tau'',\Delta^0)}+\frac{2d_{\Omega}}{\mathbb{T}_{j}(\mathfrak{N}_j,\tau'',\Delta^0)-\mathbb{T}_{j-1}(\mathfrak{N}_{j-1},\tau'',\Delta^0)}+1}{\frac{\Delta^0}{2^{\sum_{k=1}^{j}(\mathfrak{n}_k+1)}}}\right)\right\rceil 
,\\
&\hspace{2.4cm}\left\lceil \log_{2}{\frac{\left(\frac{\Delta^0}{2^{\sum_{k=1}^{j}(\mathfrak{n}_k+1)}}+1\right)}{\frac{d}{10}}} \right\rceil  \Bigg\}.
\end{align*}
  
Now, we claim that for $1\leq j\leq \operatorname{Conn}_d(x_I,y_i^0)$, we have
\begin{equation}
\begin{split}
    &f(t,x,v) \geq \ \mathbb{B}_{j}(\mathfrak{N}_j(m),\tau'',\Delta^0,\alpha'_0,|v_I|)\\
    & \hspace{2cm} \mathbf{1}_{B\left(y_{j},\frac{\Delta^0}{2^{m+1+\sum_{k=1}^{j-1}(\mathfrak{n}_k+1)}}\right) \times B\left(\mathbb{V}_{j}(\mathfrak{N}_j(m),\tau'',\Delta^0), \frac{\Delta^0}{2^{m+1+\sum_{k=1}^{j-1}(\mathfrak{n}_k+1)}}\right)}(x,v),
\end{split}
\end{equation}
for any 
\begin{equation*} 
\begin{split}
    &\mathbb{T}_{j}(\mathfrak{N}_j(m),\tau'',\Delta^0)\leq t \\ \leq
    &\min\{\mathbb{U}_{j}(\mathfrak{N}_j(m),\tau'',\Delta^0),\tau'' +\mathbb{D}_{j}(\mathfrak{N}_j(m),\tau'',\Delta^0)\}
\end{split}
\end{equation*}
and $m \geq \mathfrak{n}_j$. Here, $\alpha'_0$ is the constant defined in \textbf{Proposition \ref{initial point lower bound on a single point}}. 
\medskip

This will be proved by induction on $j$.

\medskip\paragraph{\textbf{Step 1: base case: propagation of the lower bound from $x_I$ to $y_1$.}}

\

For the case $j=1$, we  consider the \textbf{Corollary \ref{redoable translation proposition}} with $A=\alpha'_0$, $\tau=0$,  $\Delta_1=\tau''$, $\Delta_1'=\Delta^0-\tau''$, $x'=x_I$, $y=y_1$, $v'=v_I$, and $\Delta_2=\Delta^0$. Notice that all the assumptions of \textbf{Corollary \ref{redoable translation proposition}} are satisfied since 
 $B(x_I,\Delta^0) \subset B(x_I, \frac{1}{2}d(x_I,\partial \Omega)) \subset \Omega$ and $\overline{x_Iy_1} \in \Omega$ and the lower bound \eqref{CORO2.4 LOW BOUND} holds. Thus we derive a lower bound at $y_1$ from $x_I$ by using \textbf{Corollary \ref{redoable translation proposition}}:
\begin{equation*}
      f(t,x,v) \geq \mathbb{B}(m,0,\tau'',\Delta^0,\alpha'_0,|v_I|) \mathbf{1}_{B\left(y_1,\frac{\Delta^0}{2^{m+1}}\right) \times B\left(\mathbb{V}(\tau'',\Delta^0,0,m,y_1,x_I), \frac{\Delta^0}{2^{m+1}}\right)}(x,v)
\end{equation*}
for any $\mathbb{T}(\tau'',\Delta^0,0,m) \leq t \leq \min\big\{ \Delta^0,\tau'' +\mathbb{D}(\tau'',\Delta^0,m)\big\}$ 

\noindent and $m  \geq  \mathbb{M}(\tau'',\Delta^0,|v_I|,d(\overline{x_Iy_1},\partial \Omega))$.
Notice that 
\begin{align*}
   &\mathbb{M}(\tau'',\Delta^0,|v_I|,d(\overline{x_Iy_1},\partial \Omega))\\
   =& \max\left\{1,\left\lceil \log_{\frac{3\sqrt{2}}{4}}\left(\frac{\frac{4d_{\Omega}}{\tau''}+|v_I|+1}{\Delta^0}\right)\right\rceil 
,\left\lceil \log_{2}{\frac{\Delta^0+1}{d(\overline{x_Iy_1},\partial \Omega)}} \right\rceil  \right\}\\
\leq& \max\left\{1,\left\lceil \log_{\frac{3\sqrt{2}}{4}}\left(\frac{\frac{4d_{\Omega}}{\tau''}+|v_I|+1}{\Delta^0}\right)\right\rceil 
,\left\lceil \log_{2}{\frac{\Delta^0+1}{\min\{d(x_I,\partial\Omega),\frac{d}{10}\}}} \right\rceil  \right\}\\
=&\mathfrak{n}_1,
\end{align*}
which implies that as long as $m \geq \mathfrak{n}_1$, we have $m \geq \mathbb{M}(\tau'',\Delta^0,|v_I|,d(\overline{x_Iy_1},\partial \Omega))$, and this finishes the case for $j=1$.

\medskip\paragraph{\textbf{Step 2: Induction: propagation of the lower bound from $y_{k}$ to $y_{k+1}$.}}

\

Now, suppose that the case $j=k$ holds, i.e.,
\begin{equation}
\begin{split}
    &f(t,x,v) \geq \ \mathbb{B}_{k}(\mathfrak{N}_k(\tilde{m}),\tau'',\Delta^0,\alpha'_0,|v_I|)\\
    & \hspace{2cm} \mathbf{1}_{B\left(y_{k},\frac{\Delta^0}{2^{\tilde{m}+1+\sum_{l=1}^{k-1}(\mathfrak{n}_l+1)}}\right) \times B\left(\mathbb{V}_{k}(\mathfrak{N}_k(\tilde{m}),\tau'',\Delta^0), \frac{\Delta^0}{2^{\tilde{m}+1+\sum_{l=1}^{k-1}(\mathfrak{n}_l+1)}}\right)}(x,v),
\end{split}
\end{equation}
for any 
\begin{equation*} 
\begin{split}
    &\mathbb{T}_{k}(\mathfrak{N}_k(\tilde{m}),\tau'',\Delta^0)\leq t \\ \leq
    &\min\{\mathbb{U}_{k}(\mathfrak{N}_k(\tilde{m}),\tau'',\Delta^0),\tau'' +\mathbb{D}_{k}(\mathfrak{N}_k(\tilde{m}),\tau'',\Delta^0)\}
\end{split}
\end{equation*}
and $\tilde{m} \geq \mathfrak{n}_k$.

By taking $\tilde{m}=\mathfrak{n}_k$, we use \textbf{Corollary \ref{redoable translation proposition}} again by considering 
\begin{align*}
    A&=\mathbb{B}_{k}(\mathfrak{N}_k,\tau'',\Delta^0,\alpha'_0,|v_I|),\\
    \tau&=\mathbb{T}_{k}(\mathfrak{N}_k,\tau'',\Delta^0),\\
    \Delta_1&=\tau''-\mathbb{T}_{k}(\mathfrak{N}_k,\tau'',\Delta^0),\\
    \Delta_1'&=\min\{\mathbb{U}_{k}(\mathfrak{N}_k,\tau'',\Delta^0),\tau'' +\mathbb{D}_{k}(\mathfrak{N}_k,\tau'',\Delta^0)\}-\tau'',\\
    x'&=y_k,\,y=y_{k+1},\\
    v'&=\mathbb{V}_{k}(\mathfrak{N}_k,\tau'',\Delta^0),\\
    \Delta_2&= \frac{\Delta^0}{2^{\sum_{l=1}^{k}(\mathfrak{n}_l+1)}}.
\end{align*}
We notice that 
by \textbf{Remark \ref{ineq TDU}},
we have $\Delta_1>0$ and $\Delta_1'>0$.

Also, all the assumptions of \textbf{Corollary \ref{redoable translation proposition}} are satisfied again since 
\begin{align*}
    B\left(y_k,\frac{\Delta^0}{2^{\sum_{l=1}^{k}(\mathfrak{n}_l+1)}}\right)\subset B\left(y_k, \frac{d}{10}\right) \subset \Omega,
\end{align*}
 $\overline{y_{k}y_{k+1}} \in \Omega$, $\Delta_1\leq \Delta^0\leq \frac{1}{2}$, and 
\begin{align*}
    \Delta_1'&=\min\{\mathbb{U}_{k}(\mathfrak{N}_k,\tau'',\Delta^0),\tau'' +\mathbb{D}_{k}(\mathfrak{N}_k,\tau'',\Delta^0)\}-\tau''\\
    &\leq \mathbb{U}_{k}(\mathfrak{N}_k,\tau'',\Delta^0)-\tau''\\
    &\leq \mathbb{U}_{1}-\tau''\leq \Delta^0-\tau'' \leq \Delta^0 \leq \frac{1}{2}.
\end{align*}
Here, we recall that $\mathbb{U}_{1}= \Delta^0$.
Hence, by noticing that 
\begin{small}
\begin{align*}
&|\mathbb{V}_k(\mathfrak{N}_k,\tau'',\Delta^0)|\\
=&\frac{2|y_{k}-y_{k-1}|}{\mathbb{T}_{k}(\mathfrak{N}_k,\tau'',\Delta^0)-\mathbb{T}_{k-1}(\mathfrak{N}_{k-1},\tau'',\Delta^0)}\\
\leq &\frac{2d_{\Omega}}{\mathbb{T}_{k}(\mathfrak{N}_k,\tau'',\Delta^0)-\mathbb{T}_{k-1}(\mathfrak{N}_{k-1},\tau'',\Delta^0)},
\end{align*}
\end{small}

we have
\begin{equation}\label{enough eq}
\begin{split}
    &f(t,x,v) \\
    \geq &\ \mathbb{B}\Bigg(m,\mathbb{T}_k(\mathfrak{N}_k,\tau'',\Delta^0),\tau''-\mathbb{T}_k(\mathfrak{N}_k,\tau'',\Delta^0),\\
    &\frac{\Delta^0}{2^{\sum_{l=1}^{k}(\mathfrak{n}_l+1)}}, \mathbb{B}_k(\mathfrak{N}_k,\tau'',\Delta^0,\alpha'_0,|v_I|),|\mathbb{V}_k(\mathfrak{N}_k,\tau'',\Delta^0)|\Bigg)\\
    &\mathbf{1}_{B\left(y_{k+1},\frac{\Delta^0}{2^{m+1+\sum_{l=1}^{k}(\mathfrak{n}_l+1)}}\right) \times B\left(\mathbb{V}_{k+1}(\mathfrak{N}_{k+1}(m),\tau'',\Delta^0), \frac{\Delta^0}{2^{m+1+\sum_{l=1}^{k}(\mathfrak{n}_l+1)}}\right)}(x,v)\\
    \geq&\mathbb{B}_{k+1}(\mathfrak{N}_{k+1}(m),\tau'',\Delta^0,\alpha'_0,|v_I|)\\
    &\mathbf{1}_{B\left(y_{k+1},\frac{\Delta^0}{2^{m+1+\sum_{l=1}^{k}(\mathfrak{n}_l+1)}}\right) \times B\left(\mathbb{V}_{k+1}(\mathfrak{N}_{k+1}(m),\tau'',\Delta^0), \frac{\Delta^0}{2^{m+1+\sum_{l=1}^{k}(\mathfrak{n}_l+1)}}\right)}(x,v),
\end{split}
\end{equation}
for any 
\begin{align*}
    &m\geq \mathbb{M}(\tau''-\mathbb{T}_{k}(\mathfrak{N}_k,\tau'',\Delta^0),\\
    &\hspace{0.5cm}\frac{\Delta^0}{2^{\sum_{l=1}^{k}(\mathfrak{n}_l+1)}},|\mathbb{V}_{k}(\mathfrak{N}_k,\tau'',\Delta^0)|,d(\overline{y_ky_{k+1}},\partial \Omega)).
\end{align*}

We also notice that 
\begin{align*}
     &\mathbb{M}\Bigg(\tau''-\mathbb{T}_{k}(\mathfrak{N}_k,\tau'',\Delta^0),\\
    &\hspace{0.6cm}\frac{\Delta^0}{2^{\sum_{l=1}^{k}(\mathfrak{n}_l+1)}},|\mathbb{V}_{k}(\mathfrak{N}_k,\tau'',\Delta^0)|,d(\overline{y_ky_{k+1}},\partial \Omega)\Bigg)\\
    \leq &\max\Bigg\{1,\left\lceil \log_{\frac{3\sqrt{2}}{4}}\left(\frac{\frac{4d_{\Omega}}{\tau''-\mathbb{T}_{k}(\mathfrak{N}_k,\tau'',\Delta^0)}+\frac{2d_{\Omega}}{\mathbb{T}_{k}(\mathfrak{N}_k,\tau'',\Delta^0)-\mathbb{T}_{k-1}(\mathfrak{N}_{k-1},\tau'',\Delta^0)}+1}{\frac{\Delta^0}{2^{\sum_{l=1}^{k}(\mathfrak{n}_l+1)}}}\right)\right\rceil
    \\
    &\hspace{1.2cm}\left\lceil \log_{2}{\frac{\frac{\Delta^0}{2^{\sum_{l=1}^{k}(\mathfrak{n}_l+1)}}+1}{\frac{d}{10}}} \right\rceil  \Bigg\}\\
    =& \mathfrak{n}_{k+1}.
\end{align*}
Hence,  $m \geq \mathfrak{n}_{k+1}$ is sufficient for \eqref{enough eq} to hold.

\

\paragraph{\textbf{Step 3: Conclusion.}}

\

Finally, by induction, the estimate holds for $j=1,...,\operatorname{Conn}_d(x_I,y^0_i)$ and we conclude the proof of \textbf{Proposition \ref{initial point lower bound near boundary}} by selecting $j=\operatorname{Conn}_d(x_I,y^0_i)$ and defining 
\begin{align*}
&\mathfrak{R}(\tau''):=\max_{1\leq j\leq \operatorname{Conn}_d(\Omega)}\left\{\frac{2d_{\Omega}}{\mathbb{T}_{j}(\mathfrak{N}_{j},\tau'',\Delta^0)-\mathbb{T}_{j-1}(\mathfrak{N}_{j-1},\tau'',\Delta^0)}\right\},\\
&\mathfrak{B}(\tau''):=\min_{1\leq j\leq \operatorname{Conn}_d(\Omega)}\{\mathbb{B}_{j}(\mathfrak{N}_j,\tau'',\Delta^0,\alpha'_0,|v_I|)\},\\
&\mathfrak{d}_X(\tau''):=\frac{\Delta^0}{2^{\sum_{i=1}^{\operatorname{Conn}_d(\Omega)}(\mathfrak{n}_i+1)}}\leq \frac{d}{10},\\
&\mathfrak{d}_T(\tau''):=\min\bigg\{ \frac{\Delta^0-\tau''}{2}, \tau''-\mathbb{T}_{\operatorname{Conn}_d(\Omega)}(\mathfrak{N}_{\operatorname{Conn}_d(\Omega)},\tau'',\Delta^0) , \\
&\hspace{2.5cm}\mathbb{U}_{\operatorname{Conn}_d(\Omega)}(\mathfrak{N}_{\operatorname{Conn}_d(\Omega)},\tau'',\Delta^0)-\tau'', \mathbb{D}_{\operatorname{Conn}_d(\Omega)}(\mathfrak{N}_{\operatorname{Conn}_d(\Omega)},\tau'',\Delta^0) \bigg\},\\
&\mathfrak{d}(\tau''):=\min\{\mathfrak{d}_X(\tau'')  ,\mathfrak{d}_T(\tau'')\}\leq \frac{d}{10},\\
&v_i(\tau''):=\mathbb{V}_{\operatorname{Conn}_d(x_I,y^0_i)}(\mathfrak{N}_{\operatorname{Conn}_d(x_I,y^0_i)},\tau'',\Delta^0).
\end{align*}

\end{proof}

\begin{remark}\label{ineq TDU}
    Notice that for any $j \geq 1$, we have
\begin{equation}\label{small est T}
   \tau'' > \mathbb{T}_{j+1}(\mathfrak{M}_{j+1},\tau'',\Delta^0) >  \mathbb{T}_{j}(\mathfrak{M}_j,\tau'',\Delta^0)\geq\frac{\tau''}{2}
\end{equation}
(with possibly equality when $j=1$ in the lower bound $\mathbb{T}_j \geq \frac{\tau''}{2}$)
\begin{equation}\label{small est D}
    0<\mathbb{D}_{j+1}(\mathfrak{M}_{j+1},\tau'',\Delta^0) <  \mathbb{D}_{j}(\mathfrak{M}_j,\tau'',\Delta^0),
\end{equation}
and 
\begin{equation}\label{small est U}
    \tau''<\mathbb{U}_{j+1}(\mathfrak{M}_{j+1},\tau'',\Delta^0) \leq  \mathbb{U}_{j}(\mathfrak{M}_j,\tau'',\Delta^0),
\end{equation}
for $\mathfrak{m}_1 \geq \mathfrak{n}_1$, $\mathfrak{m}_i \geq 0$ for $i \geq 2$.
\end{remark}

\begin{proof}

We prove it by induction on $j$. For $j=1$, assume $\mathfrak{m}_1\geq \mathfrak{n}_1$. We notice that

\begin{equation}\label{tau dominate}
    \frac{\tau''}{2} \leq\mathbb{T}_1(\mathfrak{m}_1,\tau'',\Delta^0)=\mathbb{T}(\tau'',\Delta^0,0,\mathfrak{m}_1)=\max\left\{ \frac{\tau''}{2},\tau''-\frac{\tau''\Delta^0}{2^{\mathfrak{m}_1+2}(4d_{\Omega}+\tau'')} \right\} < \tau''
\end{equation}
and that 

\begin{equation}
\begin{split}
    &\mathbb{T}_2((\mathfrak{m}_1,\mathfrak{m}_2),\tau'',\Delta^0)=\mathbb{T}(\tau''-\mathbb{T}_{1}(\mathfrak{m}_1,\tau'',\Delta^0),\frac{\Delta^0}{2^{\mathfrak{m}_1+1}},\mathbb{T}_{1}(\mathfrak{m}_1,\tau'',\Delta^0),\mathfrak{m}_2)\\
        =&\max \left\{ \frac{\tau''}{2}+\frac{\mathbb{T}_{1}(\mathfrak{m}_1,\tau'',\Delta^0)}{2}, \tau''-\frac{(\tau''-\mathbb{T}_{1}(\mathfrak{m}_1,\tau'',\Delta^0))\frac{\Delta^0}{2^{\mathfrak{m}_1+1}}}{2^{\mathfrak{m}_2+2}(4d_{\Omega}+(\tau''-\mathbb{T}_{1}(\mathfrak{m}_1,\tau'',\Delta^0)))} \right\}.
\end{split}
\end{equation}

By \eqref{tau dominate}, we have 
\begin{equation*}
    \frac{\tau''}{2}+\frac{\mathbb{T}_{1}(\mathfrak{m}_1,\tau'',\Delta^0)}{2} > \mathbb{T}_{1}(\mathfrak{m}_1,\tau'',\Delta^0).
\end{equation*}
%with equality if and only if $\mathbb{T}_{1}(\mathfrak{m}_1,\tau'',\Delta^0)=\tau''$.
Next, we notice that the following inequality
\begin{equation}
    \tau''-\frac{(\tau''-\mathbb{T}_{1}(\mathfrak{m}_1,\tau'',\Delta^0))\frac{\Delta^0}{2^{\mathfrak{m}_1+1}}}{2^{\mathfrak{m}_2+2}(4d_{\Omega}+(\tau''-\mathbb{T}_{1}(\mathfrak{m}_1,\tau'',\Delta^0)))}> \mathbb{T}_{1}(\mathfrak{m}_1,\tau'',\Delta^0)
\end{equation}
is equivalent to the inequality below:
\begin{equation}
    \Delta^0<2^{\mathfrak{m}_1+\mathfrak{m}_2+3}(4d_{\Omega}+(\tau''-\mathbb{T}_{1}(\mathfrak{m}_1,\tau'',\Delta^0))).
\end{equation}

We notice that by \textbf{Lemma \ref{initial cover over boundary}}, we have $d_{\Omega}\geq d$. Indeed, we have
\begin{equation*}
 x_i^0-2dn(x_i^0),\,x_i^0-\frac{1}{2}dn(x_i^0) \in \Omega,\,  \left|x_i^0-2dn(x_i^0)-(x_i^0-\frac{1}{2}dn(x_i^0))\right|=\frac{3d}{2} . 
\end{equation*}
Notice that 
\begin{equation}
    2^{\mathfrak{m}_1} \geq \frac{\Delta^0+1}{\min\{d(x_I,\partial\Omega),\frac{d}{10}\}}.
\end{equation}
Here we use the fact that $$\mathfrak{m}_1 \geq \mathfrak{n}_1=\max\left\{1,\left\lceil \log_{\frac{3\sqrt{2}}{4}}\left(\frac{\frac{4d_{\Omega}}{\tau''}+|v_I|+1}{\Delta^0}\right)\right\rceil 
,\left\lceil \log_{2}{\frac{\Delta^0+1}{\min\{d(x_I,\partial\Omega),\frac{d}{10}\}}} \right\rceil  \right\}.$$
As a result, we deduce that 

\begin{equation}
    \begin{split}
        &2^{\mathfrak{m}_1+\mathfrak{m}_2+3}(4d_{\Omega}+(\tau''-\mathbb{T}_{1}(\mathfrak{m}_1,\tau'',\Delta^0)))\\
         \geq&\frac{\Delta^0+1}{\min\{d(x_I,\partial\Omega),\frac{d}{10}\}}2^{\mathfrak{m}_2+3}(4d_{\Omega}+(\tau''-\mathbb{T}_{1}(\mathfrak{m}_1,\tau'',\Delta^0))) \\
        \geq&\frac{32(\Delta^0+1)d}{\min\{d(x_I,\partial\Omega),\frac{d}{10}\}}\\
        >& \frac{d(\Delta^0+1)}{\frac{d}{10}}\\
         >& \Delta^0.
    \end{split}
\end{equation}
Hence, we conclude that $\mathbb{T}_2((\mathfrak{m}_1,\mathfrak{m}_2),\tau'',\Delta^0)>\mathbb{T}_1(\mathfrak{m}_1,\tau'',\Delta^0)$.

For the number $\mathbb{D}_j(\mathfrak{M}_j,\tau'',\Delta^0)$, we have
\begin{align*}   
\mathbb{D}_1(\mathfrak{m}_1,\tau'',\Delta^0)=\mathbb{D}(\tau'',\Delta^0,\mathfrak{m}_1)=\frac{\tau''\Delta^0}{2^{\mathfrak{m}_1+2}(4d_{\Omega}+\tau'')}>0.
\end{align*}

We also notice that
\begin{align*}   
&\mathbb{D}_{2}((\mathfrak{m}_1,\mathfrak{m}_2),\tau'',\Delta^0)=\mathbb{D}(\tau''-\mathbb{T}_{1}(\mathfrak{m}_1,\tau'',\Delta^0),\frac{\Delta^0}{2^{\mathfrak{m}_1+1}},\mathfrak{m}_2)\\
=&\frac{(\tau''-\mathbb{T}_{1}(\mathfrak{m}_1,\tau'',\Delta^0))\frac{\Delta^0}{2^{\mathfrak{m}_1+1}}}{2^{\mathfrak{m}_2+2}(4d_{\Omega}+\tau''-\mathbb{T}_{1}(\mathfrak{m}_1,\tau'',\Delta^0))}\\
=&\frac{(\tau''-\mathbb{T}_{1}(\mathfrak{m}_1,\tau'',\Delta^0))\Delta^0}{2^{\mathfrak{m}_1+\mathfrak{m}_2+3}(4d_{\Omega}+\tau''-\mathbb{T}_{1}(\mathfrak{m}_1,\tau'',\Delta^0))}\\
 <&\frac{\tau''\Delta^0}{2^{\mathfrak{m}_1+2}(4d_{\Omega}+\tau'')}=\mathbb{D}_1(\mathfrak{m}_1,\tau'',\Delta^0).
\end{align*} 

For $\mathbb{U}_{j}(\mathfrak{M}_j,\tau'',\Delta^0)$, we have 
\begin{align*}
    \mathbb{U}_{1}:=\Delta^0>\tau'',
\end{align*}
\begin{align*}
    &\mathbb{U}_{2}((\mathfrak{m}_1,\mathfrak{m}_2),\tau'',\Delta^0)=\min\{\mathbb{U}_{1},\tau''+ \mathbb{D}_{1}(\mathfrak{m}_{1},\tau'',\Delta^0)\}\leq \mathbb{U}_{1}.
\end{align*}
So we finished the case $j=1$.

Suppose that \eqref{small est T}, \eqref{small est D}, \eqref{small est U} hold for $j=k $, we have 

   \begin{equation*}
    \begin{split}
        & (\tau''-\mathbb{T}_{k}(\mathfrak{M}_k,\tau'',\Delta^0))(4d_{\Omega}+(\tau''-\mathbb{T}_{k-1}(\mathfrak{M}_{k-1},\tau'',\Delta^0)))\\
        <& 2^{\mathfrak{m}_{k+1}+1}(\tau''-\mathbb{T}_{k-1}(\mathfrak{M}_{k-1},\tau'',\Delta^0))(4d_{\Omega}+(\tau''-\mathbb{T}_{k}(\mathfrak{M}_k,\tau'',\Delta^0))),
    \end{split}
    \end{equation*}
from which we deduce that 
     \begin{equation*}
    \begin{split}
        & \frac{(\tau''-\mathbb{T}_{k}(\mathfrak{M}_k,\tau'',\Delta^0))\Delta^0}{2^{1+\sum_{l=1}^{k+1}(\mathfrak{m}_l+1)}(4d_{\Omega}+(\tau''-\mathbb{T}_{k}(\mathfrak{M}_k,\tau'',\Delta^0)))}\\
        <& \frac{(\tau''-\mathbb{T}_{k-1}(\mathfrak{M}_{k-1},\tau'',\Delta^0))\Delta^0}{2^{1+\sum_{l=1}^{k}(\mathfrak{m}_l+1)}(4d_{\Omega}+(\tau''-\mathbb{T}_{k-1}(\mathfrak{M}_{k-1},\tau'',\Delta^0)))}.
    \end{split}
    \end{equation*}

Hence, we have 
    \begin{equation}
    \begin{split}
        &\mathbb{T}_{k+1}(\mathfrak{M}_{k+1},\tau'',\Delta^0)\\
        =& \mathbb{T}(\tau''-\mathbb{T}_{k}(\mathfrak{M}_k,\tau'',\Delta^0),\frac{\Delta^0}{2^{\sum_{l=1}^{k}(\mathfrak{m}_l+1)}},\mathbb{T}_{k}(\mathfrak{M}_k,\tau'',\Delta^0),\mathfrak{m}_{k+1})\\
        =& 
\max\begin{Bmatrix}
\frac{\tau''}{2}+\frac{\mathbb{T}_{k}(\mathfrak{M}_k,\tau'',\Delta^0)}{2},\\
\tau''-\frac{(\tau''-\mathbb{T}_{k}(\mathfrak{M}_k,\tau'',\Delta^0))\frac{\Delta^0}{2^{\sum_{l=1}^{k}(\mathfrak{m}_l+1)}}}{2^{\mathfrak{m}_{k+1}+2}(4d_{\Omega}+(\tau''-\mathbb{T}_{k}(\mathfrak{M}_k,\tau'',\Delta^0)))}
\end{Bmatrix}
\\
        =&\max \Bigg\{ \frac{\tau''}{2}+\frac{\mathbb{T}_{k}(\mathfrak{M}_k,\tau'',\Delta^0)}{2},\\
        &\hspace{1.2cm}\tau''-\frac{(\tau''-\mathbb{T}_{k}(\mathfrak{M}_k,\tau'',\Delta^0))\frac{\Delta^0}{2^{\sum_{l=1}^{k}(\mathfrak{m}_l+1)}}}{2^{\mathfrak{m}_{k+1}+2}(4d_{\Omega}+(\tau''-\mathbb{T}_{k}(\mathfrak{M}_k,\tau'',\Delta^0)))} \Bigg\}\\
        =& \max \Bigg\{ \frac{\tau''}{2}+\frac{\mathbb{T}_{k}(\mathfrak{M}_k,\tau'',\Delta^0)}{2},\\
        &\hspace{1.2cm} \tau''-\frac{(\tau''-\mathbb{T}_{k}(\mathfrak{M}_k,\tau'',\Delta^0))\Delta^0}{2^{1+\sum_{l=1}^{k+1}(\mathfrak{m}_l+1)}(4d_{\Omega}+(\tau''-\mathbb{T}_{k}(\mathfrak{M}_k,\tau'',\Delta^0)))} \Bigg\}\\
        >&  \max \Bigg\{ \frac{\tau''}{2}+\frac{\mathbb{T}_{k-1}(\mathfrak{M}_{k-1},\tau'',\Delta^0)}{2},\\
        & \hspace{1.2cm} \tau''-\frac{(\tau''-\mathbb{T}_{k-1}(\mathfrak{M}_{k-1},\tau'',\Delta^0))\Delta^0}{2^{1+\sum_{l=1}^{k}(\mathfrak{m}_l+1)}(4d_{\Omega}+(\tau''-\mathbb{T}_{k-1}(\mathfrak{M}_{k-1},\tau'',\Delta^0)))} \Bigg\}\\
        =&\mathbb{T}_{k}(\mathfrak{M}_k,\tau'',\Delta^0).
    \end{split}
    \end{equation}
Also, we notice that by induction hypothesis
 \begin{equation}
    \begin{split}
&\mathbb{T}_{k+1}(\mathfrak{M}_{k+1},\tau'',\Delta^0)\\
=&\max \Bigg\{ \frac{\tau''}{2}+\frac{\mathbb{T}_{k}(\mathfrak{M}_k,\tau'',\Delta^0)}{2},\\
& \hspace{1.2cm} \tau''-\frac{(\tau''-\mathbb{T}_{k}(\mathfrak{M}_k,\tau'',\Delta^0))\Delta^0}{2^{1+\sum_{l=1}^{k+1}(\mathfrak{m}_l+1)}(4d_{\Omega}+(\tau''-\mathbb{T}_{k}(\mathfrak{M}_k,\tau'',\Delta^0)))} \Bigg\}\\
<&\tau''.
    \end{split}
    \end{equation}
    
Next, we notice that 

\begin{align*}   
0&<\mathbb{D}_{k+1}(\mathfrak{M}_{k+1},\tau'',\Delta^0)\\
&=\mathbb{D}(\tau''-\mathbb{T}_{k}(\mathfrak{M}_k,\tau'',\Delta^0),\frac{\Delta^0}{2^{\sum_{l=1}^{k}(\mathfrak{m}_l+1)}},\mathfrak{m}_{k+1})\\
&=\frac{(\tau''-\mathbb{T}_{k}(\mathfrak{M}_k,\tau'',\Delta^0))\frac{\Delta^0}{2^{\sum_{l=1}^{k}(\mathfrak{m}_l+1)}}}{2^{\mathfrak{m}_{k+1}+2}(4d_{\Omega}+(\tau''-\mathbb{T}_{k}(\mathfrak{M}_k,\tau'',\Delta^0)))}\\
&<\frac{(\tau''-\mathbb{T}_{k-1}(\mathfrak{M}_{k-1},\tau'',\Delta^0))\frac{\Delta^0}{2^{\sum_{l=1}^{k-1}(\mathfrak{m}_l+1)}}}{2^{\mathfrak{m}_{k}+2}(4d_{\Omega}+(\tau''-\mathbb{T}_{k-1}(\mathfrak{M}_{k-1},\tau'',\Delta^0)))}\\
&=\mathbb{D}_{k}(\mathfrak{M}_k,\tau'',\Delta^0),
\end{align*} 
and that
\begin{align*}
    \tau''<&\mathbb{U}_{k+1}(\mathfrak{M}_{k+1},\tau'',\Delta^0)\\
    =&\min\{\mathbb{U}_{k}(\mathfrak{M}_k,\tau'',\Delta^0),\tau'' +\mathbb{D}_{k}(\mathfrak{M}_k,\tau'',\Delta^0)\}\\
    <& \min\{\mathbb{U}_{k-1}(\mathfrak{M}_{k-1},\tau'',\Delta^0),\tau'' +\mathbb{D}_{k-1}(\mathfrak{M}_{k-1},\tau'',\Delta^0)\}\\
    =&\mathbb{U}_{k}(\mathfrak{M}_k,\tau'',\Delta^0).
\end{align*}
Hence, we show that \eqref{small est T}, \eqref{small est D}, \eqref{small est U} holds for $j=k+1$, and we conclude the proof of the statement of \textbf{Remark \ref{ineq TDU}}.

\end{proof}
\begin{remark}
    The numbers  $\{\mathfrak{n}_i\}_{i=1}^{\operatorname{Conn}_d(\Omega)}$ and $\mathfrak{R}(\tau'')$ do not depend on the choice of zigzag. 
\end{remark}

\begin{proposition} \label{initial point lowerbound general sense}
   Suppose that $\Omega \subset \mathbb{R}^3$ satisfies \textbf{Assumption A} and that the collision kernel $B$ satisfies \textbf{Assumption B} with $\nu<0$, $\alpha \in [0,1]$. Let $f(t,x,v)$ be a continuous mild solution to \eqref{Boltzmann equation}--\eqref{boundary condition}. We consider $0< \tau'' < \Delta^0$, where $\Delta^0$ is given in \textbf{Proposition \ref{initial point lower bound on a single point}}. Then, there exist $0<\delta_{X}(\tau'')= \delta_{V}(\tau'')\leq \frac{d}{56}$, $\delta_{T}(\tau'')$, $a_0(\tau'')>0$, which depend on $\Omega$, $M$, $E_f$, and a modulus of continuity of $f_0$ (and $L_{f,p}$ if $\gamma <0$); and $\mathfrak{R}_{\min}(\tau'')\geq 2$, which depends on $\tau''$, $\Omega$, $M$, and $E_f$ (and $L_{f,p}$ if $\gamma <0$), such that $\delta_{X}(\tau'') \geq \delta_T(\tau'')$ and that for all $N \in \mathbb{N}$, there exist $\mathcal{N}(N,\tau'')$ in $\mathbb{N}$, $\{ x_i \}_{i=1}^{\mathcal{N}(N,\tau'')}$ in $\Omega$ and $\{ \overline{v}_{i}(\tau'') \}_{i=1}^{\mathcal{N}(N,\tau'')}$ in $B(0,\mathfrak{R}_{\min}(\tau''))$, such that

 \begin{equation}
\overline{\Omega} \subset \underset{1 \leq i \leq \mathcal{N}(N,\tau'')}{\bigcup}B\left(x_i,\frac{\delta_X(\tau'')}{2^N}\right),
     \end{equation}

      \begin{equation}
 f(t,x,v) \geq a_0(\tau'')\mathbf{1}_{B(\overline{v}_{i}(\tau''),\delta_V(\tau''))}(v)
     \end{equation}
for any $1\leq i \leq \mathcal{N}(N,\tau'')$, and any $(t,x,v) \in [\tau'', \tau''+\delta_{T}(\tau'')]\times [B(x_i,\delta_X(\tau''))\cap \overline{\Omega}]\times \mathbb{R}^3$.  

\end{proposition}

\begin{proof}

Given $0<\tau''<\Delta^0$, we choose $0<d< \min\{1, \delta\}$ and notice that by \textbf{Proposition \ref{initial point lower bound near boundary}} there exist $\mathfrak{B}(\tau'')>0$, $\frac{d}{10}\geq \mathfrak{d}(\tau'')>0$, $\mathfrak{R}(\tau'')$, $\{v_i(\tau'')\}_{i=1}^{m_1+m_2} \in B(0,\mathfrak{R}(\tau''))$ such that
\begin{equation}\label{Prop 2.2 eq later}
    f(t,x,v) \geq \ \mathfrak{B}(\tau'')\mathbf{1}_{B(y^0_i,\mathfrak{d}(\tau'')) \times B(v_i(\tau''),\mathfrak{d}(\tau''))}(v) ,\ \forall t \in [\tau''-\mathfrak{d}(\tau''), \tau''+\mathfrak{d}(\tau'')] ,
\end{equation}
for any $1\leq i \leq m_1+m_2$.

Notice that, by the definition and \textbf{Remark \ref{ineq TDU}}, we have 
\begin{equation}
\begin{split}
    \tau'' &> \tau''-\mathfrak{d}(\tau'') \geq \tau'' - (\tau''-\mathbb{T}_{\operatorname{Conn}_d(\Omega)}(\mathfrak{N}_{\operatorname{Conn}_d(\Omega)},\tau'',\Delta^0))\\
    &=\mathbb{T}_{\operatorname{Conn}_d(\Omega)}(\mathfrak{N}_{\operatorname{Conn}_d(\Omega)},\tau'',\Delta^0)>\frac{\tau''}{2},
\end{split}
\end{equation}
\begin{equation}
    \tau''+\mathfrak{d}(\tau'') \leq \tau'' + \frac{\Delta^0-\tau''}{2}=\frac{\Delta^0+\tau''}{2}<\Delta^0.
\end{equation}
Now, we define 
\begin{equation}\label{2025 09/08 02:19}
    \mathfrak{n}'':= \max\left\{\mathbb{M}\left(\mathfrak{d}(\tau''),\mathfrak{d}(\tau''),\mathfrak{R}(\tau''),\frac{d}{16}\right),  \left\lceil \log_{\frac{3\sqrt{2}}{4}}\left(\frac{\frac{4d_{\Omega}}{\mathfrak{d}(\tau'')}+\mathfrak{R}(\tau'')+1}{\mathfrak{d}(\tau'')}\right)\right\rceil \right\}.
\end{equation}
Notice that $$r_{\mathfrak{n}''}(\mathfrak{d}(\tau''))=\left(\frac{3\sqrt{2}}{4}\right)^{\mathfrak{n}''}\mathfrak{d}(\tau'') \geq \frac{4d_{\Omega}}{\mathfrak{d}(\tau'')}+\mathfrak{R}(\tau'')+1.$$ We also define
\begin{equation}\label{painful definition}
\begin{split}
    \delta_X(\tau'')=\delta_V(\tau''):= \min\left\{  \frac{1}{4}\min\limits_{1\leq i \leq m_1+m_2}d(y_i^0,\partial \Omega), \frac{d}{56} , \frac{\mathfrak{d}(\tau'')}{2^{\mathfrak{n}''+2}},\frac{2d_{\Omega}}{\mathfrak{d}(\tau'')}\right\}.
\end{split}
\end{equation}

Given $N \in \mathbb{N}$, we notice that by compactness of $\overline{\Omega}$, there exists $\{ x_i \}_{i=1}^{\mathcal{N}(N,\tau'')} \in \Omega$ such that $\overline{\Omega} \subset \underset{1\leq i \leq \mathcal{N}(N,\tau'')}{\bigcup} B(x_i,\frac{\delta_X(\tau'')}{2^N})$. 
Given $1\leq i \leq \mathcal{N}(N,\tau'')$, we consider two cases. 

\textbf{1. Point near the boundary}: If $d(x_i,\partial \Omega) \leq \frac{d}{7}$, then there exists $x'_i \in \partial \Omega$ such that $|x'_i-x_i|\leq \frac{d}{6}$. Also, by \eqref{initial cover}, there exists $k(i)\in \{ 
1,2,...,m_1 \}$ such that $|x_{k(i)}^0-x'_i|\leq \frac{d}{8}$, so that $|x_i-x_{k(i)}^0| \leq |x_i-x_i'|+|x_i'-x_{k(i)}^0| \leq \frac{7}{24}d$.

 We consider
\begin{align}
    T'':=\max\left\{\tau''-\frac{\mathfrak{d}(\tau'')}{2},\tau''-\frac{\mathfrak{d}(\tau'')^2}{2^{\mathfrak{n}''+2}(4d_{\Omega}+\mathfrak{d}(\tau''))} \right\},
\end{align}

\begin{align*}
&\overline{v}_{i}(\tau''):=\frac{2(x_i-y^0_{k(i)})}{T''-(\tau''-\mathfrak{d}(\tau''))}=\frac{2(x_i-y^0_{k(i)})}{\max\left\{\frac{\mathfrak{d}(\tau'')}{2},\mathfrak{d}(\tau'')-\frac{\mathfrak{d}(\tau'')^2}{2^{\mathfrak{n}''+2}(4d_{\Omega}+\mathfrak{d}(\tau''))} \right\}}.
\end{align*}

\begin{equation*}
    \mathfrak{R}^1_{\min}(\tau''):=\frac{2d_{\Omega}}{\max\left\{\frac{\mathfrak{d}(\tau'')}{2},\mathfrak{d}(\tau'')-\frac{\mathfrak{d}(\tau'')^2}{2^{\mathfrak{n}''+2}(4d_{\Omega}+\mathfrak{d}(\tau''))} \right\}}
\end{equation*}

Here, we recall the definition of $y^0_{j}$ in \textbf{Lemma \ref{initial cover over boundary}}.
As long as $x \in B(x_i,\min\{\frac{\mathfrak{d}(\tau'')}{2^{\mathfrak{n}''+1}},\frac{d}{56}\})\cap \Omega$, $v \in B\left(\overline{v}_{i}(\tau''), \min \left\{ \frac{\mathfrak{d}(\tau'')}{2^{\mathfrak{n}''+1}(t-(\tau''-\mathfrak{d}(\tau'')))},1 \right\}\right)$, $T'' \leq t \leq T''+   \frac{\mathfrak{d}(\tau'')}{2^{\mathfrak{n}''+1}\left(\frac{2d_{\Omega}}{T''-(\tau''-\mathfrak{d}(\tau''))}+1\right)}$, we have
\begin{equation}\label{2025 4/18 12:04}
    x-\frac{t-(\tau''-\mathfrak{d}(\tau''))}{2}v \in B\left(y^0_{k(i)},\frac{\mathfrak{d}(\tau'')}{2^{\mathfrak{n}''}}\right).
\end{equation}

Furthermore, by definition of $\mathfrak{n}''$ in \eqref{2025 09/08 02:19}, we have
\begin{equation}
r_{\mathfrak{n}''}(\mathfrak{d}(\tau'')) \geq \mathfrak{R}(\tau'')+ \frac{4d_{\Omega}}{\mathfrak{d}(\tau'')}+1 \geq |v_{k(i)}(\tau'')|+\left|\frac{2(x_i-y^0_{k(i)})}{T''-(\tau''-\mathfrak{d}(\tau''))}\right|+1 \geq |v_{k(i)}(\tau'')|+|v|,
\end{equation}
which implies 
\begin{equation}\label{2025/04/18 12:09}
    v \in B(v_{k(i)}(\tau''), r_{\mathfrak{n}''}(\mathfrak{d}(\tau''))).
\end{equation}
Here, $\mathfrak{R}(\tau'')$ is defined in \textbf{Proposition \ref{initial point lower bound near boundary}}.
Next, notice that 
\begin{equation} \label{3/9 07:51}
    |x-x^0_{k(i)}| \leq |x-x_i|+|x_i-x^0_{k(i)}| \leq \frac{d}{56}+\frac{7d}{24}< \frac{d}{3},
\end{equation}
 so we have $x \in B(x^0_{k(i)},\frac{d}{3})$.
 Now, we parametrize $x$ and $x-\frac{t-(\tau''-\mathfrak{d}(\tau''))}{2}v$ as below:
\begin{equation}
    x=x^0_{k(i)}+\tilde{x}_1e^1_{k(i)}+\tilde{x}_2e^2_{k(i)}-\tilde{x}_3n(x^0_{k(i)}),
\end{equation}
\begin{equation} \label{3/9 07:52}
    x-\frac{t-(\tau''-\mathfrak{d}(\tau''))}{2}v=x^0_{k(i)}+\overline{x}_1e^1_{k(i)}+\overline{x}_2e^2_{k(i)}-\overline{x}_3n(x^0_{k(i)}).
\end{equation}
We notice by \eqref{3/9 07:51} that $\tilde{x}_3 < \frac{d}{3}$.
 
Since we have (notice that $\mathfrak{d}(\tau'') \leq \frac{d}{8}$)
\begin{equation*}
    x-\frac{t-(\tau''-\mathfrak{d}(\tau''))}{2}v \in B(y^0_{k(i)},\frac{\mathfrak{d}(\tau'')}{2^{\mathfrak{n}''}}) \subset B(y^0_{k(i)},\frac{d}{8}),
\end{equation*}
we deduce that $\overline{x}_3 > \frac{3d}{4}-\frac{d}{8}$.

As a result, we have 
\begin{align*}
 &|x-[x-\frac{t-(\tau''-\mathfrak{d}(\tau''))}{2}v]| \\
 \leq &|x-x^0_{k(i)}|+|x^0_{k(i)}-y^0_{k(i)}|+|y^0_{k(i)}-[x-\frac{t-(\tau''-\mathfrak{d}(\tau''))}{2}v]| \\
 \leq &\frac{d}{3}+\frac{3d}{4}+\frac{d}{8}=\frac{29d}{24}, 
\end{align*}
which implies that
\begin{equation} \label{170 eq}
\begin{split}
    &n(x^0_{k(i)})\cdot \frac{v}{|v|}\\
    =&n(x^0_{k(i)})\cdot \frac{x-[x-\frac{t-(\tau''-\mathfrak{d}(\tau''))}{2}v]}{|x-[x-\frac{t-(\tau''-\mathfrak{d}(\tau''))}{2}v]|}\\
    \geq& n(x^0_{k(i)}) \cdot \frac {[(\tilde{x}_1-\overline{x}_1)e^1_{k(i)}+(\tilde{x}_2-\overline{x}_2)e^2_{k(i)}-(\tilde{x}_3-\overline{x}_3)n(x^0_{k(i)})]}{\frac{29d}{24}}\\
    =&\frac{\overline{x}_3-\tilde{x}_3}{\frac{29d}{24}}\\
    \geq&  \frac{\frac{3d}{4}-\frac{d}{8}-\frac{d}{3}}{\frac{29d}{24}}= \frac{7}{29}.
\end{split}
\end{equation} 
Hence, by using the method as before, we have
\begin{equation}\label{a line in Omega 3.6}
    \left\{ \,x-sv \middle|\ 0 \leq s \leq \frac{t-(\tau''-\mathfrak{d}(\tau''))}{2} \,\right\} \subset \Omega.
\end{equation}
We also notice that 
\begin{equation}\label{delta_T smal est}
\begin{split}
      &T''+   \frac{\mathfrak{d}(\tau'')}{2^{\mathfrak{n}''+1}\left(\frac{2d_{\Omega}}{T''-(\tau''-\mathfrak{d}(\tau''))}+1\right)}\\
      \geq& 
    \max\left\{\tau''-\frac{\mathfrak{d}(\tau'')}{2},\tau''-\frac{\mathfrak{d}(\tau'')^2}{2^{\mathfrak{n}''+2}(4d_{\Omega}+\mathfrak{d}(\tau''))} \right\}+  \frac{\mathfrak{d}(\tau'')^2}{2^{\mathfrak{n}''+1}\left(4d_{\Omega}+\mathfrak{d}(\tau'')\right)}\\
    \geq& \tau''+\frac{\mathfrak{d}(\tau'')^2}{2^{\mathfrak{n}''+2}\left(4d_{\Omega}+\mathfrak{d}(\tau'')\right)}.
\end{split}
\end{equation}

Since $B(y^0_{k(i)},\mathfrak{d}(\tau'')) \subset B(y^0_{k(i)},\frac{d}{8}) \subset \Omega$ by \textbf{Lemma \ref{initial cover general}}, from \eqref{Prop 2.2 eq later}, we can apply \textbf{Proposition \ref{ini to multi}} with $\tau=\tau''-\mathfrak{d}(\tau'') $, $\Delta_1=2\mathfrak{d}(\tau''),\ \Delta_2=\mathfrak{d}(\tau'')$, $x'=y^0_{k(i)} $, and $v'=v_{k(i)}(\tau'')$.

Hence, we have for $\tau'' \leq t \leq T''+   \frac{\mathfrak{d}(\tau'')}{2^{\mathfrak{n}''+1}\left(\frac{2d_{\Omega}}{T''-(\tau''-\mathfrak{d}(\tau''))}+1\right)}$ (see \textbf{Remark \ref{alpha monotone thing}} and the fact that $T''+   \frac{\mathfrak{d}(\tau'')}{2^{\mathfrak{n}''+1}\left(\frac{2d_{\Omega}}{T''-(\tau''-\mathfrak{d}(\tau''))}+1\right)} \leq \tau'' + \mathfrak{d}(\tau'')$)

\begin{equation}
\begin{split}
    &f\left(\frac{t+(\tau''-\mathfrak{d}(\tau''))}{2},x-\frac{t-(\tau''-\mathfrak{d}(\tau''))}{2}v,v\right)\\
    \geq &\alpha_{\mathfrak{n}''}\left(\tau''-\mathfrak{d}(\tau''),\tau''-\frac{\mathfrak{d}(\tau'')}{2},\mathfrak{d}(\tau''),\mathfrak{B}(\tau''),\mathfrak{R}(\tau'')\right)\\
    &\mathbf{1}_{ B\left(y^0_{k(i)},\frac{\mathfrak{d}(\tau'')}{2^{\mathfrak{n}''}}\right) \times B(v_{k(i)}(\tau''), r_{\mathfrak{n}''}(\mathfrak{d}(\tau'')))}\left(x-\frac{t-(\tau''-\mathfrak{d}(\tau''))}{2}v,v\right),
\end{split}
\end{equation}
which implies by the Duhamel formula that for any $x \in B(x_i,\min\{\frac{\mathfrak{d}(\tau'')}{2^{\mathfrak{n}''+1}},\frac{d}{56}\})\cap \Omega$, $v \in B(\overline{v}_{i}(\tau''), \min \{ \frac{\mathfrak{d}(\tau'')}{2^{\mathfrak{n}''+1}(t-(\tau''-\mathfrak{d}(\tau'')))},1 \})$, $\tau'' \leq t \leq T''+   \frac{\mathfrak{d}(\tau'')}{2^{\mathfrak{n}''+1}\left(\frac{2d_{\Omega}}{T''-(\tau''-\mathfrak{d}(\tau''))}+1\right)}$, we have
\begin{equation}\label{one eq to another eq }
\begin{split}
&f(t,x,v)\\
\geq &f\left(\frac{t+(\tau''-\mathfrak{d}(\tau''))}{2},x-\frac{t-(\tau''-\mathfrak{d}(\tau''))}{2}v,v\right)e^{-\frac{t-(\tau''-\mathfrak{d}(\tau''))}{2}C_L\left \langle \mathfrak{R}^1_{\min}(\tau'')+1\right \rangle ^{\gamma^+}}\\
\geq &\alpha_{\mathfrak{n}''}\left(\tau''-\mathfrak{d}(\tau''),\tau''-\frac{\mathfrak{d}(\tau'')}{2},\mathfrak{d}(\tau''),\mathfrak{B}(\tau''),\mathfrak{R}(\tau'')\right)\\
&\mathbf{1}_{ B\left(y^0_{k(i)},\frac{\mathfrak{d}(\tau'')}{2^{\mathfrak{n}''}}\right) \times B(v_{k(i)}(\tau''), r_{\mathfrak{n}''}(\mathfrak{d}(\tau'')))}\left(x-\frac{t-(\tau''-\mathfrak{d}(\tau''))}{2}v,v\right)e^{-C_L\left \langle \mathfrak{R}^1_{\min}(\tau'')+1\right \rangle ^{\gamma^+}}\\
\geq &\alpha_{\mathfrak{n}''}\left(\tau''-\mathfrak{d}(\tau''),\tau''-\frac{\mathfrak{d}(\tau'')}{2},\mathfrak{d}(\tau''),\mathfrak{B}(\tau''),\mathfrak{R}(\tau'')\right) e^{-C_L\left \langle  
\mathfrak{R}^1_{\min}(\tau'')+1 \right \rangle ^{\gamma^+}}.
\end{split}
\end{equation}
Here, the last line is from \eqref{2025 4/18 12:04} and \eqref{2025/04/18 12:09}.
Now, we define $$\delta_{T,1}(\tau''):=\frac{\mathfrak{d}(\tau'')^2}{2^{\mathfrak{n}''+2}(4d_{\Omega}+\mathfrak{d}(\tau''))}.$$ Notice that $\delta_{T,1}(\tau'') \leq \mathfrak{d}(\tau'')$, and we see by \eqref{delta_T smal est} that
\begin{equation*}
  \tau''+\delta_{T,1}(\tau'')  \leq  T''+   \frac{\mathfrak{d}(\tau'')}{2^{\mathfrak{n}''+1}\left(\frac{2d_{\Omega}}{T''-(\tau''-\mathfrak{d}(\tau''))}+1\right)} .
\end{equation*}

As a result, we deduce from \eqref{one eq to another eq } that for all 
$(x,v) \in \Omega \times \mathbb{R}^3$ and $\tau'' \leq t \leq \tau''+\delta_{T,1}(\tau'')$, we have 
\begin{align*}
 &   f(t,x,v) \geq \alpha_{\mathfrak{n}''}\left(\tau''-\mathfrak{d}(\tau''),\tau''-\frac{\mathfrak{d}(\tau'')}{2},\mathfrak{d}(\tau''),\mathfrak{B}(\tau''),\mathfrak{R}(\tau'')\right) e^{-C_L\left \langle  
\mathfrak{R}^1_{\min}(\tau'')+1 \right \rangle ^{\gamma^+}}\\
&\mathbf{1}_{B(x_i,\min\{\frac{\mathfrak{d}(\tau'')}{2^{\mathfrak{n}''+1}},\frac{d}{56}\}) \times B(\overline{v}_{i}(\tau''), \min \{ \frac{\mathfrak{d}(\tau'')}{2^{\mathfrak{n}''+1}(t-(\tau''-\mathfrak{d}(\tau'')))},1 \})}(x,v)\\
&\geq \alpha_{\mathfrak{n}''}\left(\tau''-\mathfrak{d}(\tau''),\tau''-\frac{\mathfrak{d}(\tau'')}{2},\mathfrak{d}(\tau''),\mathfrak{B}(\tau''),\mathfrak{R}(\tau'')\right) e^{-C_L\left \langle  
\mathfrak{R}^1_{\min}(\tau'')+1 \right \rangle ^{\gamma^+}} \\
&\hspace{0.5cm}\mathbf{1}_{B(x_i,\delta_X) \times B(\overline{v}_{i}(\tau''), \delta_V(\tau''))}(x,v),
\end{align*}
 where we used \eqref{painful definition} to derive that $\delta_X \leq  \max\{\frac{\mathfrak{d}(\tau'')}{2^{\mathfrak{n}''+1}},\frac{d}{56}\}$.
In the final line we used the fact that 
\begin{align*}
    &\min \left\{ \frac{\mathfrak{d}(\tau'')}{2^{\mathfrak{n}''+1}(t-(\tau''-\mathfrak{d}(\tau'')))},1 \right\} \\
    \geq & \min \left\{ \frac{\mathfrak{d}(\tau'')}{2^{\mathfrak{n}''+1}((\tau''+\mathfrak{d}(\tau''))-(\tau''-\mathfrak{d}(\tau'')))},1 \right\} \geq \min \left\{\frac{1}{2^{\mathfrak{n}''+2}},1 \right\}.
\end{align*}

\textbf{2. Point far from the boundary.} 

If $d(x_i,\partial \Omega) > \frac{d}{7}$, by \textbf{Lemma \ref{initial cover general}}, we can find $m_1+1 \leq k(i)\leq m_1+m_2$ such that $x_i \in B(y^0_{k(i)}, \frac{d}{16})$, which implies $\overline{y^0_{k(i)}x_i} \subset B(y^0_{k(i)},\frac{d}{16})  \subset \Omega-\Omega_{\frac{d}{16}}$.

We can use \textbf{Corollary \ref{redoable translation proposition}}, since by \eqref{Prop 2.2 eq later}, \eqref{CORO2.4 LOW BOUND} is satisfied with $\tau=\tau''-\mathfrak{d}(\tau'')$,  $\Delta_1=\Delta_1'=\Delta_2=\mathfrak{d}(\tau'')$, $x'=y^0_{k(i)}$, $y=x_i$, $v'=v_{k(i)}(\tau'')$. Notice as before that all the assumptions of \textbf{Corollary \ref{redoable translation proposition}} are satisfied by the fact that $B(y^0_{k(i)},\mathfrak{d}(\tau'')) \subset B(y^0_{k(i)}, \frac{d}{8}) \subset \Omega$ and that $\overline{y^0_{k(i)}x_i} \subset B\left(y^0_{k(i)},\frac{d}{16}\right) \subset  \Omega$.

Thus, we have
\begin{equation}\label{2025 3 26 02:36}
\begin{split}
   & f(t,x,v)\\
   &\geq \mathbb{B}(\tilde{m},\tau''-\mathfrak{d}(\tau''),\mathfrak{d}(\tau''),\mathfrak{d}(\tau''),\mathfrak{B}(\tau''),|v_{k(i)}(\tau'')|)\\ &\hspace{0.4cm}\times\mathbf{1}_{B(x_i,\frac{\mathfrak{d}(\tau'')}{2^{\tilde{m}+1}}) \times B\left(\mathbb{V}(\mathfrak{d}(\tau''),\mathfrak{d}(\tau''),\tau''-\mathfrak{d}(\tau''),\tilde{m},x_i,y^0_{k(i)}), \frac{\mathfrak{d}(\tau'')}{2^{\tilde{m}+1}}\right)}(x,v).
\end{split}
\end{equation}
for $\tilde{m} \geq \mathbb{M}(\mathfrak{d}(\tau''),\mathfrak{d}(\tau''),|v_{k(i)}(\tau'')|,d(\overline{y^0_{k(i)}x_i},\partial \Omega))$, $(x,v) \in \Omega \times \mathbb{R}^3$ and $\mathbb{T}(\mathfrak{d}(\tau''),\mathfrak{d}(\tau''),\tau''-\mathfrak{d}(\tau''),\tilde{m}) \leq t \leq \min\{\tau''+\mathfrak{d}(\tau''),\tau''+\mathbb{D}(\mathfrak{d}(\tau''),\mathfrak{d}(\tau''),\tilde{m}))\}$. 
Next, we define $\delta_{T,2}(\tau''):=\min\{\mathfrak{d}(\tau''),\mathbb{D}(\mathfrak{d}(\tau''),\mathfrak{d}(\tau''),\mathfrak{n}'')\}$.
We notice that 
\begin{equation}
\mathfrak{n}''=\mathbb{M}\left(\mathfrak{d}(\tau''),\mathfrak{d}(\tau''),\mathfrak{R}(\tau''),\frac{d}{16}\right) \geq \mathbb{M}\left(\mathfrak{d}(\tau''),\mathfrak{d}(\tau''),|v_{k(i)}(\tau'')|,d\left(\overline{y^0_{k(i)}x_i},\partial \Omega\right)\right).
\end{equation}
As a result, by \eqref{2025 3 26 02:36}, we have for $(x,v) \in \Omega \times \mathbb{R}^3$ and $\mathbb{T}(\mathfrak{d}(\tau''),\mathfrak{d}(\tau''),\tau''-\mathfrak{d}(\tau''),\mathfrak{n}'') \leq t \leq \min\{\tau''+\mathfrak{d}(\tau''),\tau''+\mathbb{D}(\mathfrak{d}(\tau''),\mathfrak{d}(\tau''),\mathfrak{n}''))\}=\tau''+\delta_{T,2}(\tau'')$
\begin{align*}
   & f(t,x,v)\\
   \geq &\mathbb{B}(\mathfrak{n}'',\tau''-\mathfrak{d}(\tau''),\mathfrak{d}(\tau''),\mathfrak{d}(\tau''),\mathfrak{B}(\tau''),|v_{k(i)}(\tau'')|)\\
   &\mathbf{1}_{B(x_i,\frac{\mathfrak{d}(\tau'')}{2^{\mathfrak{n}''+1}}) \times B\left(\mathbb{V}(\mathfrak{d}(\tau''),\mathfrak{d}(\tau''),\tau''-\mathfrak{d}(\tau''),\mathfrak{n}'',x_i,y^0_{k(i)}), \frac{\mathfrak{d}(\tau'')}{2^{\mathfrak{n}''+1}}\right)}(x,v)\\
     \geq &\mathbb{B}(\mathfrak{n}'',\tau''-\mathfrak{d}(\tau''),\mathfrak{d}(\tau''),\mathfrak{d}(\tau''),\mathfrak{B}(\tau''),|v_{k(i)}(\tau'')|)\\
     &
    \mathbf{1}_{B(x_i,\delta_X) \times B\left(\mathbb{V}(\mathfrak{d}(\tau''),\mathfrak{d}(\tau''),\tau''-\mathfrak{d}(\tau''),\mathfrak{n}'',x_i,y^0_{k(i)}), \delta_V(\tau'') \right)}(x,v)\\
     \geq &\mathbb{B}(\mathfrak{n}'',\tau''-\mathfrak{d}(\tau''),\mathfrak{d}(\tau''),\mathfrak{d}(\tau''),\mathfrak{B}(\tau''),\mathfrak{R}(\tau''))\\
     &\mathbf{1}_{B(x_i,\delta_X) \times B\left(\mathbb{V}(\mathfrak{d}(\tau''),\mathfrak{d}(\tau''),\tau''-\mathfrak{d}(\tau''),\mathfrak{n}'',x_i,y^0_{k(i)}), \delta_V(\tau'') \right)}(x,v),
\end{align*}
where we used \eqref{painful definition}.
We now define
\begin{equation}
\overline{v}_{i}(\tau''):=\mathbb{V}(\mathfrak{d}(\tau''),\mathfrak{d}(\tau''),\tau''-\mathfrak{d}(\tau''),\mathfrak{n}'',x_i,y^0_{k(i)})=\frac{2(x_i-y^0_{k(i)})}{\max\left\{\frac{\mathfrak{d}(\tau'')}{2},\mathfrak{d}(\tau'')-\frac{\mathfrak{d}(\tau'')^2}{2^{\mathfrak{n}''+2}(4d_{\Omega}+\mathfrak{d}(\tau''))} \right\}},
\end{equation}
\begin{align*}
    &\mathfrak{R}^2_{\min}(\tau''):=\frac{2d_{\Omega}}{\max\left\{\frac{\mathfrak{d}(\tau'')}{2},\mathfrak{d}(\tau'')-\frac{\mathfrak{d}(\tau'')^2}{2^{\mathfrak{n}''+2}(4d_{\Omega}+\mathfrak{d}(\tau''))} \right\}}
\end{align*}
and finish the case by noticing that 

\begin{equation*}
    \mathbb{T}(\mathfrak{d}(\tau''),\mathfrak{d}(\tau''),\tau''-\mathfrak{d}(\tau''),\mathfrak{n}'') < \tau'' 
\end{equation*}
and that
\begin{equation*}
    |\overline{v}_{i}(\tau'')| \leq \mathfrak{R}^2_{\min}(\tau'').
\end{equation*}

We conclude the proof by defining 
\begin{equation}\label{2025 11/12 16:21}
    \delta_{T}(\tau''):=\min\{\delta_{T,1}(\tau''),\delta_{T,2}(\tau''),\delta_{X}(\tau'')\},
\end{equation}

\begin{equation}\label{R_min defi}
    \mathfrak{R}_{\min}(\tau''):=\max\{\mathfrak{R}^1_{\min}(\tau'') , \mathfrak{R}^2_{\min}(\tau''),2\},
\end{equation}
and 
\begin{equation*}
\begin{split}
   &a_0(\tau''):=  \min \bigg\{ \alpha_{\mathfrak{n}''}\left(\tau''-\mathfrak{d}(\tau''),\tau''-\frac{\mathfrak{d}(\tau'')}{2},\mathfrak{d}(\tau''),\mathfrak{B}(\tau''),\mathfrak{R}(\tau'')\right) e^{-C_L\left \langle  
\mathfrak{R}^1_{\min}(\tau'')+1 \right \rangle ^{\gamma^+}},\\
   &\hspace{2.5cm} \mathbb{B}(\mathfrak{n}'',\tau''-\mathfrak{d}(\tau''),\mathfrak{d}(\tau''),\mathfrak{d}(\tau''),\mathfrak{B}(\tau''),\mathfrak{R}(\tau'')) \bigg\}. 
\end{split}
\end{equation*}
\end{proof}

\section{Maxwellian bound for the cut-off case for non-fully specular reflection condition}

In this section, we consider the lower bound problem in the cut-off case with a non-fully specular boundary condition. 

\begin{proposition} \label{positve integral estimate}
    Suppose that $\Omega \subset \mathbb{R}^3$ satisfies \textbf{Assumption A} and that the collision kernel $B$ satisfies \textbf{Assumption B} with $\alpha \in [0,1]$. We consider a continuous mild solution $f(t,x,v)$ of \eqref{Boltzmann equation}--\eqref{boundary condition} . Given $\tau'' \in (0,\Delta^0)$ with $\Delta^0>0$  given in \textbf{Proposition \ref{initial point lower bound on a single point}}, there exist $b(\tau'')$, $\delta_{T,I}(\tau'')>0$, which depends on $\Omega, M, E_f$, and the modulus of continuity of $f_0$, such that for $t \in [\tau'', \tau''+\delta_{T,I}(\tau'')]$, 

    \begin{equation}
        \forall x \in \partial \Omega, \ \int_{v_*\cdot n(x)>0}f(t,x,v_*)(v_*\cdot n(x))\,dv_*>b(\tau'').
    \end{equation}
\end{proposition}

\begin{remark}
    Although the estimate remains valid for $\alpha=1$, in our analysis, we only employ it for $\alpha \in [0,1)$.
\end{remark}

\begin{proof}

Given $0<\tau''<\Delta^0$, we choose $0<d< \min\{1, \delta\}$ and notice that by \textbf{Proposition \ref{initial point lower bound near boundary}} there exist $\mathfrak{B}(\tau'')>0$, $\frac{d}{10}\geq \mathfrak{d}(\tau'')>0$, $\mathfrak{R}(\tau'')$, $\{v_i(\tau'')\}_{i=1}^{m_1+m_2} \in B(0,\mathfrak{R}(\tau''))$ such that
\begin{equation}\label{Prop 2.2 eq later}
    f(t,x,v) \geq \ \mathfrak{B}(\tau'')\mathbf{1}_{B(y^0_i,\mathfrak{d}(\tau'')) \times B(v_i(\tau''),\mathfrak{d}(\tau''))}(x,v) ,\ \forall t \in [\tau''-\mathfrak{d}(\tau''), \tau''+\mathfrak{d}(\tau'')] ,
\end{equation}
for any $1\leq i \leq m_1+m_2$.
   Given $x_{\partial} \in \partial \Omega$, by \textbf{Lemma \ref{initial cover over boundary}}, there exists $1\leq i \leq m_1$ such that $x_{\partial} \in B(x^0_i,\frac{d}{8})$.

  Then, we proceed with the same method as in \textbf{Proposition \ref{initial point lowerbound general sense}}, this time we propagate the lower bound from $x^0_i$ to $x_{\partial}$ 

   We define
 \begin{equation}\label{m' new defi}
     \mathfrak{m}_{\partial}:= \max\left\{1,\left\lceil \log_{\frac{3\sqrt{2}}{4}}\left(\frac{\frac{4d_{\Omega}}{\mathfrak{d}(\tau'')}+\mathfrak{R}(\tau'')+1}{\mathfrak{d}(\tau'')}\right)\right\rceil 
 \right\},
 \end{equation} and define

\begin{align}
    T_{\partial}:=\max\left\{\tau''-\frac{\mathfrak{d}(\tau'')}{2},\tau''-\frac{\mathfrak{d}(\tau'')^2}{2^{\mathfrak{m}_{\partial}+2}(4d_{\Omega}+\mathfrak{d}(\tau''))} \right\},
\end{align}

\begin{align}\label{v partial defi}
&v_{x_{\partial}}(\tau''):=\frac{2(x_{\partial}-y^0_{i})}{T_{\partial}-(\tau''-\mathfrak{d}(\tau''))}=\frac{2(x_{\partial}-y^0_{i})}{\max\left\{\frac{\mathfrak{d}(\tau'')}{2},\mathfrak{d}(\tau'')-\frac{\mathfrak{d}(\tau'')^2}{2^{\mathfrak{m}_{\partial}+2}(4d_{\Omega}+\mathfrak{d}(\tau''))} \right\}},
\end{align}
\begin{align}
&R^{\partial}_{\min}(\tau''):=\frac{2d_{\Omega}}{\max\left\{\frac{\mathfrak{d}(\tau'')}{2},\mathfrak{d}(\tau'')-\frac{\mathfrak{d}(\tau'')^2}{2^{\mathfrak{m}_{\partial}+2}(4d_{\Omega}+\mathfrak{d}(\tau''))} \right\}}.
\end{align}

Notice that, by the definition of notation of $R^{\partial}_{\min}(\tau'')$, we have
\begin{equation}\label{2025 3 26 02:55}
    |v_{x_{\partial}}(\tau'')| \leq R^{\partial}_{\min}(\tau'').
\end{equation}

As long as $v \in B\left(v_{x_{\partial}}(\tau''), \min \left\{ \frac{\mathfrak{d}(\tau'')}{2^{\mathfrak{m}_{\partial}+1}(t-(\tau''-\mathfrak{d}(\tau'')))},1 \right\}\right)$, $T_{\partial} \leq t \leq T_{\partial}+   \frac{\mathfrak{d}(\tau'')}{2^{\mathfrak{m}_{\partial}+1}\left(\frac{2d_{\Omega}}{T_{\partial}-(\tau''-\mathfrak{d}(\tau''))}+1\right)}$, we have

\begin{equation}\label{2025 3 26 03:12}
\begin{split}
       & \left|x_{\partial}-\frac{t-(\tau''-\mathfrak{d}(\tau''))}{2}v-y^0_{i}\right| < \frac{\mathfrak{d}(\tau'')}{2^{\mathfrak{m}_{\partial}+1}}.
\end{split}
\end{equation}

  Furthermore, by the definition of $\mathfrak{m}_{\partial}$ in \eqref{m' new defi}, we have
\begin{equation}\label{2025 3 26 03:13}
r_{\mathfrak{m}_{\partial}}(\mathfrak{d}(\tau'')) \geq \mathfrak{R}(\tau'')+ \frac{4d_{\Omega}}{\mathfrak{d}(\tau'')}+1 \geq |v_{i}(\tau'')|+\left|\frac{2(x_{\partial}-y^0_{i})}{T_{\partial}-(\tau''-\mathfrak{d}(\tau''))}\right|+1 \geq |v_{i}(\tau'')|+|v|,
\end{equation}
which implies $v \in B(v_{i}(\tau''), r_{\mathfrak{m}_{\partial}}(\mathfrak{d}(\tau'')))$. Here, we recall $\mathfrak{R}(\tau'')$ again in \textbf{Proposition \ref{initial point lower bound near boundary}}.
Next, notice that 
\begin{equation}\label{3/9 07:49}
    |x_{\partial}-x^0_{i}| \leq \frac{d}{8},
\end{equation}
and that 
\begin{equation}\label{3/9 07:50}
    x_{\partial}-\frac{t-(\tau''-\mathfrak{d}(\tau''))}{2}v \in B(y^0_{i},\frac{\mathfrak{d}(\tau'')}{2^{\mathfrak{m}_{\partial}+1}})\subset B(y^0_{i},\mathfrak{d}(\tau'')) \subset B(y^0_{i},\frac{d}{8})
\end{equation}
by the fact that $\mathfrak{d}(\tau'') \leq \frac{d}{8}$.
As before, we note that the segment $\left\{ x_{\partial}-sv|\ 0 \leq s \leq \frac{t-(\tau''-\mathfrak{d}(\tau''))}{2} \right\}$ between $x_{\partial}$ and $x_{\partial}-\frac{t-(\tau''-\mathfrak{d}(\tau''))}{2}v \in B\left(y^0_{i},\frac{\mathfrak{d}(\tau'')}{2^{\mathfrak{m}_{\partial}}}\right)$ also lies within $\Omega$. 
We also notice that 
\begin{equation}\label{painful saturday 1}
\begin{split}
      &T_{\partial}+   \frac{\mathfrak{d}(\tau'')}{2^{\mathfrak{m}_{\partial}+1}\left(\frac{2d_{\Omega}}{T_{\partial}-(\tau''-\mathfrak{d}(\tau''))}+1\right)}\\
      \geq& 
    \max\left\{\tau''-\frac{\mathfrak{d}(\tau'')}{2},\tau''-\frac{\mathfrak{d}(\tau'')^2}{2^{\mathfrak{m}_{\partial}+2}(4d_{\Omega}+\mathfrak{d}(\tau''))} \right\}+  \frac{\mathfrak{d}(\tau'')^2}{2^{\mathfrak{m}_{\partial}+1}\left(4d_{\Omega}+\mathfrak{d}(\tau'')\right)}\\
    \geq& \tau''+\frac{\mathfrak{d}(\tau'')^2}{2^{\mathfrak{m}_{\partial}+2}\left(4d_{\Omega}+\mathfrak{d}(\tau'')\right)}.
\end{split}
\end{equation}

Then, using \textbf{Proposition \ref{ini to multi}} (note that as before by \eqref{Prop 2.2 eq later}, the assumption \eqref{ini to multi assume eq} holds with $\tau=\tau''-\mathfrak{d}(\tau'') $, $\Delta_1=2\mathfrak{d}(\tau''),\ \Delta_2=\mathfrak{d}(\tau'')$, $x'=y^0_{i} $, and $v'=v_{i}(\tau'')$, we also have $B(y^0_{i},\mathfrak{d}(\tau'')) \subset B(y^0_{i},\frac{d}{8}) \subset \Omega$ by \textbf{Lemma \ref{initial cover general}}), we have for $\tau'' \leq t \leq T_{\partial}+   \frac{\mathfrak{d}(\tau'')}{2^{\mathfrak{m}_{\partial}+1}\left(\frac{2d_{\Omega}}{T_{\partial}-(\tau''-\mathfrak{d}(\tau''))}+1\right)}$ (see \textbf{Remark \ref{alpha monotone thing}} and the fact that $T_{\partial}+   \frac{\mathfrak{d}(\tau'')}{2^{\mathfrak{m}_{\partial}+1}\left(\frac{2d_{\Omega}}{T_{\partial}-(\tau''-\mathfrak{d}(\tau''))}+1\right)} \leq \tau'' + \mathfrak{d}(\tau'')$)

\begin{equation}
\begin{split}
    &f\left(\frac{t+(\tau''-\mathfrak{d}(\tau''))}{2},x_{\partial}-\frac{t-(\tau''-\mathfrak{d}(\tau''))}{2}v,v\right)\\
    \geq& \alpha_{\mathfrak{m}_{\partial}}\left(\tau''-\mathfrak{d}(\tau''),\tau''-\frac{\mathfrak{d}(\tau'')}{2},\mathfrak{d}(\tau''),\mathfrak{B}(\tau''),\mathfrak{R}(\tau'')\right)\\
    &\mathbf{1}_{ B\left(y^0_{i},\frac{\mathfrak{d}(\tau'')}{2^{\mathfrak{m}_{\partial}}}\right) \times B\left(v_{i}(\tau''), r_{\mathfrak{m}_{\partial}}(\mathfrak{d}(\tau''))\right)}\left(x_{\partial}-\frac{t-(\tau''-\mathfrak{d}(\tau''))}{2}v,v\right),
\end{split}
\end{equation}
which implies by the Duhamel formula that for any $v \in B(v_{x_{\partial}}(\tau''), \min \{ \frac{\mathfrak{d}(\tau'')}{2^{\mathfrak{m}_{\partial}+1}(t-(\tau''-\mathfrak{d}(\tau'')))},1 \})$, $\tau'' \leq t \leq T_{\partial}+   \frac{\mathfrak{d}(\tau'')}{2^{\mathfrak{m}_{\partial}+1}\left(\frac{2d_{\Omega}}{T_{\partial}-(\tau''-\mathfrak{d}(\tau''))}+1\right)}$, we have (notice that by \eqref{2025 3 26 02:55}, we have $|v| \leq |v_{x_{\partial}}(\tau'')|+1 \leq R^{\partial}_{\min}(\tau'')+1$)
\begin{equation}\label{painful saturday 2}
\begin{split}
&f(t,x_{\partial},v)\\
\geq& f\left(\frac{t+(\tau''-\mathfrak{d}(\tau''))}{2},x_{\partial}-\frac{t-(\tau''-\mathfrak{d}(\tau''))}{2}v,v\right)e^{-\frac{t-(\tau''-\mathfrak{d}(\tau''))}{2}C_L\left \langle R^{\partial}_{\min}(\tau'')+1\right \rangle ^{\gamma^+}}\\
\geq& \alpha_{\mathfrak{m}_{\partial}}\left(\tau''-\mathfrak{d}(\tau''),\tau''-\frac{\mathfrak{d}(\tau'')}{2},\mathfrak{d}(\tau''),\mathfrak{B}(\tau''),\mathfrak{R}(\tau'')\right)\\
&\mathbf{1}_{ B\left(y^0_{i},\frac{\mathfrak{d}(\tau'')}{2^{\mathfrak{m}_{\partial}}}\right) \times B(v_{i}(\tau''), r_{\mathfrak{m}_{\partial}}(\mathfrak{d}(\tau'')))}\left(x_{\partial}-\frac{t-(\tau''-\mathfrak{d}(\tau''))}{2}v,v\right)e^{-C_L\left \langle R^{\partial}_{\min}(\tau'')+1\right \rangle ^{\gamma^+}}.
\end{split}
\end{equation}
Next, we recall \eqref{2025 3 26 03:12} and \eqref{2025 3 26 03:13}: 
\begin{equation}
    x_{\partial}-\frac{t-(\tau''-\mathfrak{d}(\tau''))}{2}v \in B\left(y^0_{i},\frac{\mathfrak{d}(\tau'')}{2^{\mathfrak{m}_{\partial}}}\right), v \in B(v_{i}(\tau''), r_{\mathfrak{m}_{\partial}}(\mathfrak{d}(\tau''))).
\end{equation}

So, we have
\begin{equation}
\begin{split}
&\alpha_{\mathfrak{m}_{\partial}}(\tau''-\mathfrak{d}(\tau''),\frac{2\tau''-\mathfrak{d}(\tau'')}{2},\mathfrak{d}(\tau''),\mathfrak{B}(\tau''),\mathfrak{R}(\tau''))\\
&\mathbf{1}_{ B\left(y^0_{i},\frac{\mathfrak{d}(\tau'')}{2^{\mathfrak{m}_{\partial}}}\right) \times B(v_{i}(\tau''), r_{\mathfrak{m}_{\partial}}(\mathfrak{d}(\tau'')))}(x_{\partial}-\frac{t-(\tau''-\mathfrak{d}(\tau''))}{2}v,v)e^{-C_L\left \langle R^{\partial}_{\min}(\tau'')+1\right \rangle ^{\gamma^+}}\\
\geq &\alpha_{\mathfrak{m}_{\partial}}(\tau''-\mathfrak{d}(\tau''),\frac{2\tau''-\mathfrak{d}(\tau'')}{2},\mathfrak{d}(\tau''),\mathfrak{B}(\tau''),\mathfrak{R}(\tau'')) e^{-C_L\left \langle  
R^{\partial}_{\min}(\tau'')+1 \right \rangle ^{\gamma^+}}.
\end{split}
\end{equation}

Now, we define $\delta_{T,I}(\tau''):=\frac{\mathfrak{d}(\tau'')^2}{2^{\mathfrak{m}_{\partial}+2}(4d_{\Omega}+\mathfrak{d}(\tau''))}$(notice as before that $\delta_{T,I}(\tau'') \leq \mathfrak{d}(\tau'')$), and notice by \eqref{painful saturday 1} that
\begin{equation}
    T_{\partial}+   \frac{\mathfrak{d}(\tau'')}{2^{\mathfrak{m}_{\partial}+1}\left(\frac{2d_{\Omega}}{T_{\partial}-(\tau''-\mathfrak{d}(\tau''))}+1\right)}
    \geq \tau''+\delta_{T,I}(\tau'').
\end{equation}

As a result, we have by \eqref{painful saturday 2} that for all 
$(x,v) \in \Omega \times \mathbb{R}^3$ and $\tau'' \leq t \leq \tau''+\delta_{T,I}(\tau'')$ ,
\begin{align*}
 &   f(t,x_{\partial},v)\\
 \geq& \alpha_{\mathfrak{m}_{\partial}}(\tau''-\mathfrak{d}(\tau''),\frac{2\tau''-\mathfrak{d}(\tau'')}{2},\mathfrak{d}(\tau''),\mathfrak{B}(\tau''),\mathfrak{R}(\tau'')) e^{-C_L\left \langle  
R^{\partial}_{\min}(\tau'')+1 \right \rangle ^{\gamma^+}}\\
&\mathbf{1}_{ B(v_{x_{\partial}}(\tau''), \min \{ \frac{\mathfrak{d}(\tau'')}{2^{\mathfrak{m}_{\partial}+1}(t-(\tau''-\mathfrak{d}(\tau'')))},1 \})}(v)\\
\geq &\alpha_{\mathfrak{m}_{\partial}}(\tau''-\mathfrak{d}(\tau''),\frac{2\tau''-\mathfrak{d}(\tau'')}{2},\mathfrak{d}(\tau''),\mathfrak{B}(\tau''),\mathfrak{R}(\tau'')) e^{-C_L\left \langle  
R^{\partial}_{\min}(\tau'')+1 \right \rangle ^{\gamma^+}} \mathbf{1}_{B(v_{x_{\partial}}(\tau''),\frac{1}{2^{\mathfrak{m}_{\partial}+2}})}(v).
\end{align*}
In the final line, we use the fact again that 
\begin{equation*}
    \min \left\{ \frac{\mathfrak{d}(\tau'')}{2^{\mathfrak{m}_{\partial}+1}(t-(\tau''-\mathfrak{d}(\tau'')))},1 \right\} \geq \min \left\{ \frac{\mathfrak{d}(\tau'')}{2^{\mathfrak{m}_{\partial}+1}((\tau''+\mathfrak{d}(\tau''))-(\tau''-\mathfrak{d}(\tau'')))},1 \right\} \geq \frac{1}{2^{\mathfrak{m}_{\partial}+2}}.
\end{equation*}
Next, we define 
\begin{align*}
    &A':=\alpha_{\mathfrak{m}_{\partial}}(\tau''-\mathfrak{d}(\tau''),\frac{2\tau''-\mathfrak{d}(\tau'')}{2},\mathfrak{d}(\tau''),\mathfrak{B}(\tau''),\mathfrak{R}(\tau'')) e^{-C_L\left \langle  
R^{\partial}_{\min}(\tau'')+1 \right \rangle ^{\gamma^+}},\\
&\delta_{V,I}(\tau''):=\frac{1}{2^{\mathfrak{m}_{\partial}+2}}.
\end{align*}
  We conclude that 
\begin{equation}
\begin{split}
&f(t,x_{\partial},v)\geq A'\mathbf{1}_{ B(v_{x_{\partial}}(\tau''),\delta_{V,I}(\tau''))}(v),\\
\end{split}
\end{equation}
for $\tau'' \leq t \leq \tau''+\delta_{T,I}(\tau'')$. Here we recall as in \eqref{v partial defi} that $v_{x_{\partial}}(\tau'')=\frac{2(x_{\partial}-y^0_i)}{\max\left\{\frac{\mathfrak{d}(\tau'')}{2},\mathfrak{d}(\tau'')-\frac{\mathfrak{d}(\tau'')^2}{2^{\mathfrak{m}_{\partial}+2}(4d_{\Omega}+\mathfrak{d}(\tau''))} \right\}}$.

Next, by parameterizing $$ x_{\partial}= x_i^{0}+x_{\partial,1}e^1_{i}+x_{\partial,2}e^2_{i}-\phi_i(x_{\partial,1},x_{\partial,2})n(x_i^{0}), $$ we find that 

    \begin{equation}
\begin{split}
  &v_{x_{\partial}}(\tau'')\cdot n(x_{\partial})\\
   =&\frac{2}{\max\left\{\frac{\mathfrak{d}(\tau'')}{2},\mathfrak{d}(\tau'')-\frac{\mathfrak{d}(\tau'')^2}{2^{\mathfrak{m}_{\partial}+2}(4d_{\Omega}+\mathfrak{d}(\tau''))} \right\}}[(x_i^0-y_i^0)+(x_{\partial}-x_i^0)]\cdot n(x_{\partial})\\
   \geq& \frac{2}{\max\left\{\frac{\mathfrak{d}(\tau'')}{2},\mathfrak{d}(\tau'')-\frac{\mathfrak{d}(\tau'')^2}{2^{\mathfrak{m}_{\partial}+2}(4d_{\Omega}+\mathfrak{d}(\tau''))} \right\}\sqrt{|\nabla\phi_i(x_{\partial,1},x_{\partial,2})|^2+1}}\\
  &\times\left[\left(0,0,-\frac{3d}{4}\right)\cdot (\partial_x\phi_i(x_{\partial,1},x_{\partial,2}),\partial_y\phi_i(x_{\partial,1},x_{\partial,2}),-1)-\left|(x_{\partial}-x_i^0)\right|\right]\\
   \geq &\frac{5d}{8\mathfrak{d}(\tau'')}.
\end{split}  
\end{equation}

Then, we have
\begin{equation*}
 v\cdot n(x_{\partial}) \geq \frac{5d}{16\mathfrak{d}(\tau'')}, \ \forall v \in B\left(v_{x_{\partial}}(\tau''), \frac{5d}{16\mathfrak{d}(\tau'')} \right).      
\end{equation*}

Hence, for $t \in [\tau'',\tau''+\delta_{T,I}(\tau'')]$, $x_{\partial} \in \partial \Omega$,

    \begin{equation}
\begin{split}
  &  \int_{v_*\cdot n(x_{\partial})>0}f(t,x_{\partial},v_*)(v_*\cdot n(x_{\partial}))\,dv_*\\
   \geq& \int_{B\left(v_{x_{\partial}}(\tau''),\min\{ \delta_{V,I}(\tau''),\frac{5d}{16\mathfrak{d}(\tau'')} \} \right)} A' \frac{5d}{16\mathfrak{d}(\tau'')}\,dv_*\\
 \geq& A' \frac{5d}{16\mathfrak{d}(\tau'')}\frac{4\pi(\min\{ \delta_{V,I}(\tau''),\frac{5d}{16\mathfrak{d}(\tau'')} \})^3}{3}\\
    =:&b(\tau''),
\end{split}  
\end{equation}
and we finish the proof.

\end{proof}

The rest of the proof of existence of lower bound is exactly the same as in Briant's paper \cite{Bri1}(from Prop. 3.9, pp. 21–25). Before we introduce the next proposition, we define some constants here:

\begin{definition}\label{2025 09/23 11:51}
  Given $\tau'',\tau'>0$, $N \in \mathbb{N}$, $0\leq\alpha<1$, $\delta_X(\tau'')$, $\delta_V(\tau'')$, $\mathfrak{R}_{\min}(\tau'')$ as in \textbf{Prop. \ref{initial point lowerbound general sense}} and a sequence $\{ \xi_n \}_{n=1}^{\infty}$ satisfying $0\leq\xi_n \leq\frac{1}{4}$ for $n \in \mathbb{N}$,  we define the following numbers:
  \begin{equation}\label{r'' defi}
r''_{n}:=
\begin{cases}
\delta_V(\tau''),\ n=0\\
\sqrt{2}(1-\xi_{n})r''_{n-1},\ n \geq 1,
\end{cases}
  \end{equation}

  \begin{equation}\label{3/9 05:07}
b_{n,\tau''}(\tau'):=
\begin{cases}
\min\{b(\tau''),a_{0}(\tau'')\}&,\ n=0\\
\min \begin{Bmatrix}
    (1-\alpha)b(\tau'')e^{-C_L(\tau'-\tau'')\langle \mathfrak{R}_{\min}(\tau'')+r''_{1} \rangle^{\gamma^+}}\\
    \times\frac{1}{2\pi T_B^2}e^{-\frac{(\mathfrak{R}_{\min}(\tau'')+r''_{1} )^2}{2T_B}},\\
(\min(a_{0}(\tau''),b(\tau')))^2C_Qr_{0}^{''3+\gamma}\xi^{\frac{1}{2}}_{1} \\\times \frac{\tau'-\tau''}{2^{2}(\mathfrak{R}_{\min}(\tau'')+r''_{1})}e^{-C_L\frac{(\tau'-\tau'')\langle \mathfrak{R}_{\min}(\tau'')+r''_{1} \rangle^{\gamma^+}}{2(\mathfrak{R}_{\min}(\tau'')+r''_{1})}} 
\end{Bmatrix} & ,\ n=1\\
(1-\alpha)b(\tau'')e^{-C_L(\tau'-\tau'')\langle \mathfrak{R}_{\min}(\tau'')+r''_{n} \rangle^{\gamma^+}} \frac{1}{2\pi T_B^2}e^{-\frac{(\mathfrak{R}_{\min}(\tau'')+r''_{n} )^2}{2T_B}}&,\ n \geq 2,
\end{cases}
  \end{equation}
where $b(\tau'')$ is introduced in \textbf{Prop. \ref{positve integral estimate}}, and

  \begin{equation}\label{3/9 05:08}
a_{n,\tau''}(\tau'):=
\begin{cases}
a_0(\tau'')&,\ n=0\\
(\min(a_{n-1,\tau''}(\tau'),b_{n-1,\tau''}(\tau')))^2C_Qr_{n-1}^{''3+\gamma}\xi^{\frac{1}{2}}_{n}\\
\times \frac{\tau'-\tau''}{2^{n+1}(\mathfrak{R}_{\min}(\tau'')+r''_{n})}e^{-C_L\frac{(\tau'-\tau'')\langle \mathfrak{R}_{\min}(\tau'')+r''_{n} \rangle^{\gamma^+}}{2^{n}(\mathfrak{R}_{\min}(\tau'')+r''_{n})}}&,\ n \geq 1,
\end{cases}
  \end{equation}
where $a_0(\tau'')$ is defined in \textbf{Proposition \ref{initial point lowerbound general sense}}.    
\end{definition}

We now introduce the following proposition.
\begin{proposition}\label{final lower bound extension}
    Suppose that $\Omega \subset \mathbb{R}^3$ satisfies \textbf{Assumption A} and that the collision kernel $B$ satisfies \textbf{Assumption B} with $\nu<0$. Let $f$ be the continuous mild solution of \eqref{Boltzmann equation}--\eqref{boundary condition} with $0 \leq\alpha<1$. We fix $\tau'' \in (0,\Delta^0)$, where $\Delta^0>0$ is given in \textbf{Proposition \ref{initial point lower bound on a single point}}.
   Then, there exist $\delta_{T,D}(\tau'')>0$, which depends on $\Omega, M, E_f$, and the modulus of continuity of $f_0$, such that for every $\tau' \in (\tau'',\tau''+\delta_{T,D}(\tau'')]$ and every $N \in \mathbb{N}$, for each $0\leq n \leq N$ and each $1 \leq i \leq \mathcal{N}(N,\tau'')$, we have 

   \begin{equation}\label{3/9 12:43}
   \begin{split}
       &  \ \ \ \ \   f(t,x,v) \geq \min\left(a_{n,\tau''}(\tau'),b_{n,\tau''}(\tau')\right)\mathbf{1}_{B(\overline{v}_{i}(\tau''),r''_{n})}(v),\\
       &\forall (t,x) \in \left[\tau''+(\tau'-\tau'')\left(1-\frac{1}{2^{n+1}(\mathfrak{R}_{\min}(\tau'')+r''_{n})}\right),\tau'\right]\times \left[B\left(x_i,\frac{\delta_X(\tau'')}{2^n}\right)\cap \overline{\Omega}\right],
   \end{split}
   \end{equation}
   where $x_i$ and $\overline{v}_{i}(\tau'')$ are defined in \textbf{Proposition \ref{initial point lowerbound general sense}} and depend on $N$.
   Here, we recall that the notation $\delta_X(\tau'')$, $\mathfrak{R}_{\min}(\tau'')$, $\mathcal{N}(N,\tau'')$ are defined in \textbf{Proposition \ref{initial point lowerbound general sense}} and that $a_{n,\tau''}(\tau')>0,\,b_{n,\tau''}(\tau')>0$ and $r''_{n}$ are defined in \textbf{Definition \ref{2025 09/23 11:51}}.
\end{proposition}

\begin{proof}
    The proof follows the same induction-on-$n$ scheme as Proposition 3.9 of \cite{Bri1}. %We provide only a sketch of the proof here. 
    Notice that we have $\mathfrak{R}_{\min}(\tau'')\geq 1$ in \textbf{Proposition \ref{initial point lowerbound general sense}}.  Then, we define (we recall that the notion $\delta_{T,I}(\tau'')$ provided in \textbf{Proposition \ref{positve integral estimate}})
    \begin{equation}\label{delta _TD defi}
      \delta_{T,D}(\tau''):=\min\{\delta_X(\tau''),\delta_T(\tau''),\delta_{T,I}(\tau'')\}
    \end{equation} and pick $\tau' \in (\tau'',\tau''+\delta_{T,D}(\tau'')]$. This choice ensures $\delta_{T,D}(\tau'') \leq \delta_X(\tau'')$
 (for the interior displacement), $\delta_{T,D}(\tau'') \leq \delta_{T}(\tau'')$ (for the boundary integral) and $\delta_{T,D}(\tau'') \leq \delta_{T,I}(\tau'')$(for the time slab).
    The proof uses an induction on $n$. The case $n=0$ is a direct consequence of \textbf{Proposition \ref{initial point lowerbound general sense}}.
    \medskip 
    
    We now assume that the case $n=k$ holds. Given $1 \leq i \leq \mathcal{N}(N,\tau'')$, consider
    \begin{equation}\label{Prop 4.2 txv assumption}
    \begin{split}
                &(t,x,v) \in \left[\tau''+(\tau'-\tau'')\left(1-\frac{1}{2^{k+2}(r''_{k+1}+\mathfrak{R}_{\min}(\tau''))}\right),\tau'\right] \\
        &\hspace{2cm}\times \left[B\left(x_i,\frac{\delta_X(\tau'')}{2^{k+1}}\right) \cap \overline{\Omega} \right] \times B\left(\overline{v}_{i}(\tau''),r''_{k+1}\right).
    \end{split}
    \end{equation}
    We generate the announced lower bound depending on whether the characteristic line $\{X_{s,t}(x,v)\}_{\tau'' \leq s \leq t}$ touches the boundary or not.

    \medskip \textbf{First case: $X_{s,t}(x,v) \in \overline{\Omega}$ for $\tau'' \leq s \leq t$.}
We use the Duhamel formula \eqref{Duhamel formula} and apply \textbf{Lemma \ref{up bound of L lemma}} to deduce that
    \begin{equation}\label{3/9 12:31}
    \begin{split}
        &f(t,x,v)\\
        \geq &\int_{\tau''+(\tau'-\tau'')\left(1-\frac{1}{2^{k+1}\left(r''_{k+1}+\mathfrak{R}_{\min}(\tau'')\right)}\right)}^{\tau''+(\tau'-\tau'')\left(1-\frac{1}{2^{k+2}\left(r''_{k+1}+\mathfrak{R}_{\min}(\tau'')\right)}\right)} e^{-C_L(t-s)\left\langle r''_{k+1}+\mathfrak{R}_{\min}(\tau'')\right\rangle^{\gamma^+}} \\ 
        &Q^+[f(s,X_{s,t}(x,v),\cdot),f(s,X_{s,t}(x,v),\cdot)](v)\,ds.
    \end{split}
    \end{equation}
Next, we notice that in the assumption that 
 \begin{equation}\label{2025 04/06 08:26}
      \tau' \in (\tau'',\tau''+\delta_{T,D}(\tau'')],       
      \end{equation}

for any $(t,x,v)$ which satisfies \eqref{Prop 4.2 txv assumption}, for all 
\begin{equation}\label{Prop 4.2 s range}
\begin{split}
    &\tau''+(\tau'-\tau'')\left(1-\frac{1}{2^{k+1}(r''_{k+1}+\mathfrak{R}_{\min}(\tau''))}\right)\leq s \\ 
    \leq& \tau''+(\tau'-\tau'')\left(1-\frac{1}{2^{k+2}(r''_{k+1}+\mathfrak{R}_{\min}(\tau''))}\right),
\end{split}
  \end{equation}
      we have
\begin{equation}
    \begin{split}
   |X_{s,t}(x,v)-x_i|\leq \frac{\delta_X(\tau'')}{2^{k}}.
    \end{split}
    \end{equation}
Indeed, we have
\begin{equation}
    |X_{s,t}(x,v)-x_i|\leq |X_{s,t}(x,v)-x|+|x-x_i| \leq |v|(t-s)+\frac{\delta_X(\tau'')}{2^{k+1}}.
\end{equation}
    Plus, we have  
    \begin{equation}
        |v| \leq r''_{k+1}+|\overline{v}_{i}(\tau'')| \leq r''_{k+1}+\mathfrak{R}_{\min}(\tau''),
    \end{equation}
    \begin{equation*}
      t-s \leq \tau'-\left[\tau''+(\tau'-\tau'')\left(1-\frac{1}{2^{k+1}(r''_{k+1}+\mathfrak{R}_{\min}(\tau''))}\right)\right]=\frac{\tau'-\tau''}{2^{k+1}(r''_{k+1}+\mathfrak{R}_{\min}(\tau''))}.
    \end{equation*}
    Hence, we have (we also use \eqref{2025 04/06 08:26} and \eqref{delta _TD defi})
    \begin{equation}
    \begin{split}
   &|X_{s,t}(x,v)-x_i| \leq \frac{\tau'-\tau''}{2^{k+1}}+\frac{\delta_X(\tau'')}{2^{k+1}} \leq \frac{\delta_{T,D}(\tau'')}{2^{k+1}}+\frac{\delta_X(\tau'')}{2^{k+1}}  \leq \frac{\delta_X(\tau'')}{2^{k}}.
    \end{split}
    \end{equation}
Since $r''_{k+1}(\tau'') > r''_{k}(\tau'')$, we have $$\tau''+(\tau'-\tau'')\left(1-\frac{1}{2^{k+1}\left(r''_{k+1}+\mathfrak{R}_{\min}(\tau'')\right)}\right) > \tau''+(\tau'-\tau'')\left(1-\frac{1}{2^{k+1}\left(r''_{k}+\mathfrak{R}_{\min}(\tau'')\right)}\right).$$ 
    Thus we can apply the induction assumption and deduce that
    \begin{equation}\label{Prop4.2 very smal esti}
        f(s,X_{s,t}(x,v),w) \geq \min\left(a_{k,\tau''}(\tau'),b_{k,\tau''}(\tau')\right)\mathbf{1}_{B(\overline{v}_{i}(\tau''),r''_{k})}(w),
    \end{equation}
for any $w \in \mathbb{R}^3$, $s$ satisfying \eqref{Prop 4.2 s range}, $i=1,...,\mathcal{N}(N,\tau'')$, and $(t,x,v)$ satisfying \eqref{Prop 4.2 txv assumption}.
    Thus, we deduce from \eqref{Prop4.2 very smal esti} and  \textbf{Lemma \ref{spreading property lemma}} that
      \begin{equation}
      \begin{split}
    &f(t,x,v)\\ 
    \geq & \int_{\tau''+(\tau'-\tau'')\left(1-\frac{1}{2^{k+1}\left(r''_{k+1}+\mathfrak{R}_{\min}(\tau'')\right)}\right)}^{\tau''+(\tau'-\tau'')\left(1-\frac{1}{2^{k+2}\left(r''_{k+1}+\mathfrak{R}_{\min}(\tau'')\right)}\right)} \min\left(a_{k,\tau''}(\tau'),b_{k,\tau''}(\tau')\right)^2\\
    &\times e^{-C_L(t-s)\langle r''_{k+1}+\mathfrak{R}_{\min}(\tau'')\rangle^{\gamma^+}} C_Q(r''_{k})^{3+\gamma}\xi_{k+1}^{\frac{1}{2}}\mathbf{1}_{B\left(\overline{v}_{i}(\tau''),r''_{k+1}\right)}(v)\,ds\\
    \geq &\int_{\tau''+(\tau'-\tau'')\left(1-\frac{1}{2^{k+1}\left(r''_{k+1}+\mathfrak{R}_{\min}(\tau'')\right)}\right)}^{\tau''+(\tau'-\tau'')\left(1-\frac{1}{2^{k+2}\left(r''_{k+1}+\mathfrak{R}_{\min}(\tau'')\right)}\right)} \min\left(a_{k,\tau''}(\tau'),b_{k,\tau''}(\tau')\right)^2\\
    &\times e^{-C_L(t-s)\langle r''_{k+1}+\mathfrak{R}_{\min}(\tau'')\rangle^{\gamma^+}} C_Q(r''_{k})^{3+\gamma}\xi_{k+1}^{\frac{1}{2}}\,ds\\
  \geq & \min\left(a_{k,\tau''}(\tau'),b_{k,\tau''}(\tau')\right)^2 C_Q (r''_{k})^{3+\gamma}\xi_{k+1}^{\frac{1}{2}}\frac{\tau'-\tau''}{2^{k+2}(r''_{k+1}+\mathfrak{R}_{\min}(\tau''))}\\
  &\times e^{\frac{-(\tau'-\tau'')C_L\langle r''_{k+1}+\mathfrak{R}_{\min}(\tau'') \rangle^{\gamma^+}}{2^{k+1}(r''_{k+1}+\mathfrak{R}_{\min}(\tau''))}}\\
  =&a_{k+1,\tau''}(\tau') \geq \min\left(a_{k+1,\tau''}(\tau'),b_{k+1,\tau''}(\tau')\right).
    \end{split}
    \end{equation}
    Hence, we showed that \eqref{3/9 12:43} holds for $n=k+1$ when $X_{s,t}(x,v) \in \Omega$ for $\tau'' \leq s \leq t$.

\medskip

\textbf{Second case: $X_{s,t}(x,v) \notin \overline{\Omega}$ for some $s \in  [\tau'', t]$.} 
%If $\alpha =1$ then by \eqref{Duhamel formula specular}, we can do exactly as in the first case to derive the same bound. Hence, we only need to consider the case that $\alpha < 1$.

First, we observe that by definition of $t_{\partial}(x,v)$ in \textbf{Definition \ref{3/9 01:05}}, we have $t_{\partial}(x,v)\leq t-s \leq t- \tau'' \leq \tau'-\tau''$. Next, we notice by \eqref{delta _TD defi} that $\tau' \in (\tau'',\tau''+\delta_{T,I}(\tau'')]$. Then, since $\alpha<1$, the diffuse part yields a strictly positive contribution, thus we use \eqref{Duhamel formula B}, \textbf{Proposition \ref{positve integral estimate}}, and \textbf{Lemma \ref{up bound of L lemma}} to derive that for any $(t,x,v)$ satisfying \eqref{Prop 4.2 txv assumption}, we have 
\begin{equation}
\begin{split}
&f(t,x,v)\\
\geq& (1-\alpha)\left( \int_{w\cdot n(X_{t-t_{\partial}(x,v),t}(x,v))>0} f(t,X_{t-t_{\partial}(x,v),t}(x,v),w)(w\cdot n(X_{t-t_{\partial}(x,v),t}(x,v))) dw \right)\\
&\frac{1}{2\pi T_B^2}e^{-\frac{|v|^2}{2T_B}} \exp\left[-\int_{t-t_{\partial}(x,v)}^tL[f(s,X_{s,t}(x,v),\cdot)](v)\,ds\right]\\
\geq& (1-\alpha)b(\tau'')\frac{1}{2\pi T_B^2}e^{-\frac{|v|^2}{2T_B}} e^{-t_{\partial}(x,v)C_L\left\langle v\right\rangle ^{\gamma^+}}\\
\geq& (1-\alpha)b(\tau'')\frac{1}{2\pi T_B^2}e^{-\frac{(\mathfrak{R}_{\min}(\tau'')+r''_{k+1})^2}{2T_B}} e^{-(\tau'-\tau'') C_L{\left \langle \mathfrak{R}_{\min}(\tau'')+r''_{k+1}\right\rangle} ^{\gamma^+}}\\
\geq& b_{k+1,\tau''}(\tau') \geq \min\left(a_{k+1,\tau''}(\tau'),b_{k+1,\tau''}(\tau')\right),
\end{split}
\end{equation}
which shows that \eqref{3/9 12:43} holds for $n=k+1$ in the case when the backward characteristic $s \mapsto X_{s,t}(x,v)$ touches the boundary.

In conclusion, we showed that \eqref{3/9 12:43} holds for $n=k+1$ and by induction that \eqref{3/9 12:43} holds for $0 \leq n \leq N$.
\end{proof}

Finally, we end the proof of \textbf{Theorem \ref{Main theorem}} for case $\alpha \in [0,1)$. We pick $0<\xi<\frac{1}{4}$ and setting $\xi_n:=\xi^n$, by which we can deduce that (we recall $r_n''$ from \eqref{r'' defi})

\begin{equation}
  c_r(\xi)2^{\frac{n}{2}} \leq  r''_{n}=\delta_V(\tau'') 2^{\frac{n}{2}}\Pi_{i=1}^{n}(1-\xi^i) \leq \delta_V(\tau'') 2^{\frac{n}{2}},
\end{equation}
where $c_r(\xi):=\delta_V(\tau'')\Pi_{i=1}^{\infty}(1-\xi^i)>0$.

%Thanks to Lemma 3.3 in \cite{Mou 1}%
We first show that given $ \tau_2 \in \left[\tau''+\min\left\{\frac{\delta_{T,D}(\tau'')-\tau''}{2},\tau''\right\},\tau''+\delta_{T,D}(\tau'')\right]$ there exists $\mathcal{A}>0$ such that
\begin{equation}
    f(\tau_2,x,v) \geq \mathcal{A}^{2^n},\ \forall n \in \mathbb{N},\ (x,v) \in \overline{\Omega} \times B(0,c_r(\xi)2^{\frac{n}{2}}),
\end{equation}
to get a lower maxwellian lower bound of a continuous mild solution of \eqref{Boltzmann equation}--\eqref{boundary condition}.

We will use the proof from \cite{Bri1, Mou 1} to show the existence of $\mathcal{A}_1>0$ and $\mathcal{A}_2>0$ such that
\begin{equation}
    b_{n,\tau''}(\tau_2) \geq \mathcal{A}_1^{2^n},\ a_{n,\tau''}(\tau_2)\geq \mathcal{A}_2^{2^n},
\end{equation}
for any $0<\tau''<\tau_2\leq \tau''+\delta_{T,D}(\tau'')$.

In this article, we only show the sketch of the proof. For details, see \cite{Bri1, Mou 1}.
First, we observe that for $n \geq 2$ (we recall that $\mathfrak{R}_{\min}(\tau'') \geq 1$)
\begin{align*}
&b_{n,\tau''}(\tau_2)\\
=&(1-\alpha)b(\tau'')e^{-C_L(\tau_2-\tau'')\langle \mathfrak{R}_{\min}(\tau'')+r''_{n} \rangle^{\gamma^+}} \frac{1}{2\pi T_B^2}e^{-\frac{(\mathfrak{R}_{\min}(\tau'')+r''_{n} )^2}{2T_B}}     \\
=&(1-\alpha)b(\tau'')e^{-C_L(\tau_2-\tau'')(1+ \mathfrak{R}_{\min}(\tau'')+r''_{n})^{\gamma^+}} \frac{1}{2\pi T_B^2}e^{-\frac{(\mathfrak{R}_{\min}(\tau'')+r''_{n} )^2}{2T_B}}     \\
 \geq & (1-\alpha)b(\tau'') \frac{1}{2\pi T_B^2}e^{-C_L\tau_2\mathfrak{R}_{\min}(\tau'')^{\gamma^+}(2+r''_{n})^2}e^{-\frac{(\mathfrak{R}_{\min}(\tau'')+r''_{n} )^2}{2T_B}}\\
 \geq &(1-\alpha)b(\tau'') \frac{1}{2\pi T_B^2}e^{-8C_L\tau_2\mathfrak{R}_{\min}(\tau'')^{\gamma^+}}e^{-\frac{\mathfrak{R}_{\min}(\tau'')^2}{T_B}}e^{-\left((2C_L\tau_2\mathfrak{R}_{\min}(\tau'')^{\gamma^+}+\frac{1}{T_B})(\delta_V(\tau''))^2 2^n\right) }\\
\geq & C_1e^{-\left((2C_L\tau_2 \mathfrak{R}_{\min}(\tau'')^{\gamma^+}+\frac{1}{T_B})(\delta_V(\tau''))^2 2^n\right) },
\end{align*}
where we defined $C_1:=(1-\alpha)b(\tau'') \frac{1}{2\pi T_B^2}e^{-8C_L\tau_2\mathfrak{R}_{\min}(\tau'')^{\gamma^+}}e^{-\frac{\mathfrak{R}_{\min}(\tau'')^2}{T_B}}$.

By defining 
\begin{equation*}
    \mathcal{A}_1(\tau_2):=
    \begin{cases}
        \min \left\{b_{0,\tau''}(\tau_2), \sqrt{b_{1,\tau''}(\tau_2)}, e^{-\left((2C_L\tau_2\mathfrak{R}_{\min}(\tau'')^{\gamma^+}+\frac{1}{T_B})(\delta_V(\tau''))^2\right) } \right\},\ C_1 \geq 1\\
        \min \left\{b_{0,\tau''}(\tau_2), \sqrt{b_{1,\tau''}(\tau_2)}, C_1e^{-\left((2C_L\tau_2\mathfrak{R}_{\min}(\tau'')^{\gamma^+}+\frac{1}{T_B})(\delta_V(\tau''))^2\right) } \right\},\ C_1 < 1,
    \end{cases}
\end{equation*}
we conclude that for $n \in \mathbb{N}$
\begin{equation}\label{3/9 05:15}
    b_{n,\tau''}(\tau_2)\geq \mathcal{A}_1^{2^n}(\tau_2).
\end{equation}

For $a_{n,\tau''}(\tau_2)$, we also recall the definition (for $n \geq 1$):

\begin{align*}
    &a_{n,\tau''}(\tau_2)=(\min(a_{n-1,\tau''}(\tau_2),b_{n-1,\tau''}(\tau_2)))^2C_Qr_{n-1}^{''3+\gamma}\xi^{\frac{n}{2}}\frac{\tau_2-\tau''}{2^{n+1}(\mathfrak{R}_{\min}(\tau'')+r''_{n})}\\
    &\hspace{1.9cm}\times e^{-C_L\frac{(\tau_2-\tau'')\langle \mathfrak{R}_{\min}(\tau'')+r''_{n} \rangle^{\gamma^+}}{2^{n}(\mathfrak{R}_{\min}(\tau'')+r''_{n})}}.
\end{align*}

We notice that for $n \geq 1$, we have

\begin{align*}
    &\frac{(\tau_2-\tau'')\langle \mathfrak{R}_{\min}(\tau'')+r''_{n} \rangle^{\gamma^+}}{2^{n+1}(\mathfrak{R}_{\min}(\tau'')+r''_{n})}\leq \frac{\tau_2(1+ \mathfrak{R}_{\min}(\tau''))^{\gamma^+}}{\mathfrak{R}_{\min}(\tau'')}+\frac{\tau_2(\delta_V(\tau''))^{\gamma^+}}{\mathfrak{R}_{\min}(\tau'')}.  
\end{align*}

By defining 

\begin{equation*}
    C_2:= C_Q \min\left\{\frac{\delta_{T,D}(\tau'')-\tau''}{2},\tau''\right\} e^{ -C_L\left(\frac{\tau_2(1+ \mathfrak{R}_{\min}(\tau''))^{\gamma^+}}{\mathfrak{R}_{\min}(\tau'')}+\frac{\tau_2(\delta_V(\tau''))^{\gamma^+}}{\mathfrak{R}_{\min}(\tau'')} \right) },
\end{equation*}
we have the following inequality:

\begin{align*}
      &a_{n,\tau''}(\tau_2)\\
      \geq& C_2(\min(a_{n-1,\tau''}(\tau_2),b_{n-1,\tau''}(\tau_2)))^2 r_{n-1}^{''3+\gamma}\xi^{\frac{n}{2}}\frac{1}{2^{n+1}(\mathfrak{R}_{\min}(\tau'')+r''_{n})}\\
       \geq& C_2(\min(a_{n-1,\tau''}(\tau_2),b_{n-1,\tau''}(\tau_2)))^2 (c_r(\xi)2^{\frac{n-1}{2}})^{3+\gamma}\xi^{\frac{n}{2}}\frac{1}{2^{n+1}(\mathfrak{R}_{\min}(\tau'')+\delta_V(\tau'')2^{\frac{n}{2}})}\\   
       \geq & C_2(\min(a_{n-1,\tau''}(\tau_2),b_{n-1,\tau''}(\tau_2)))^2 (\frac{c_r(\xi)}{2})^{3+\gamma} (2^{\frac{3+\gamma}{2}})^{n}\xi^{\frac{n}{2}}\frac{1}{2^{n+1}(\mathfrak{R}_{\min}(\tau'')+\delta_V(\tau'')2^{\frac{n}{2}})}\\
       \geq &C_2(\min(a_{n-1,\tau''}(\tau_2),b_{n-1,\tau''}(\tau_2)))^2 (\frac{c_r(\xi)}{2})^{3+\gamma} (2^{\frac{3+\gamma}{2}})^{n}\xi^{\frac{n}{2}}\frac{1}{4}\frac{1}{2^n}\min\{ \frac{1}{\mathfrak{R}_{\min}(\tau'')}, \frac{1}{\delta_V(\tau'') 2^{\frac{n}{2}}} \}\\
        \geq& C_2(\min(a_{n-1,\tau''}(\tau_2),b_{n-1,\tau''}(\tau_2)))^2 (\frac{c_r(\xi)}{2})^{3+\gamma} \frac{1}{4}(2^{\frac{1+\gamma}{2}}\xi^{\frac{1}{2}})^{n}\min\{ \frac{1}{\mathfrak{R}_{\min}(\tau'')}, \frac{1}{\delta_V(\tau'') 2^{\frac{n}{2}}} \}\\
        \geq & C_2(\min(a_{n-1,\tau''}(\tau_2),b_{n-1,\tau''}(\tau_2)))^2 (\frac{c_r(\xi)}{2})^{3+\gamma} \frac{1}{4}(2^{\frac{1+\gamma}{2}}\xi^{\frac{1}{2}})^{n}\\
       &\times \min\left\{ \frac{1}{\max\{ \mathfrak{R}_{\min}(\tau''), \delta_V(\tau'') \} }, \frac{1}{\max\{ \mathfrak{R}_{\min}(\tau''), \delta_V(\tau'') \} 2^{\frac{n}{2}}} \right\}\\
        \geq& C_2(\min(a_{n-1,\tau''}(\tau_2),b_{n-1,\tau''}(\tau_2)))^2 (\frac{c_r(\xi)}{2})^{3+\gamma} \frac{1}{4(\mathfrak{R}_{\min}(\tau'')+\delta_V(\tau''))}(2^{\frac{1+\gamma}{2}}\xi^{\frac{1}{2}})^{n} \frac{1}{ 2^{\frac{n}{2}}} \\
        \geq& C_2(\min(a_{n-1,\tau''}(\tau_2),b_{n-1,\tau''}(\tau_2)))^2 \left(\frac{c_r(\xi)}{2}\right)^{3+\gamma} \frac{1}{4(\mathfrak{R}_{\min}(\tau'')+\delta_V(\tau''))}(2^{\frac{\gamma}{2}}\xi^{\frac{1}{2}})^{n}.
\end{align*}
By defining 
\begin{equation}
    C_3:=C_2 (\frac{c_r(\xi)}{2})^{3+\gamma} \frac{1}{4(\mathfrak{R}_{\min}(\tau'')+\delta_V(\tau''))},\ \lambda:=\min\{ 1, 2^{\frac{\gamma}{2}}\xi^{\frac{1}{2}}, C_3 2^{\frac{\gamma}{2}}\xi^{\frac{1}{2}}\},
\end{equation}
we have
\begin{equation}\label{final iterate}
    a_{n,\tau''}(\tau_2)\geq (\min(a_{n-1,\tau''}(\tau_2),b_{n-1,\tau''}(\tau_2)))^2 \lambda^{n}. 
\end{equation}

Next, note that by \eqref{3/9 05:07} and \eqref{3/9 05:08}, we have
\begin{equation*}
    b_{0,\tau''}(\tau_2) \leq a_{0,\tau''}(\tau_2),\ b_{1,\tau''}(\tau_2) \leq a_{1,\tau''}(\tau_2).
\end{equation*}
Then, we define the following numbers:
\begin{equation}
    l_n:=\min\{ 1\leq l \leq n-1 \mid a_{n-l,\tau''}(\tau_2) \geq b_{n-l,\tau''}(\tau_2) \}.
\end{equation}
In doing so, we can iterate \eqref{final iterate} and use \eqref{3/9 05:15} to derive:

\begin{equation}\label{3/9 06:22}
\begin{split}
     a_{n,\tau''}(\tau_2) &\geq \lambda^{n+2(n-1)+...+2^{l_n-1}(n-l_n+1)}  (b_{n-l_n,\tau''}(\tau_2))^{2^{l_n}}\\
     &\geq \lambda^{n+2(n-1)+...+2^{l_n-1}(n-l_n+1)}  \mathcal{A}_1^{2^{n-l_n}2^{l_n}}\\
     &\geq \lambda^{\sum_{i=0}^{l_n-1}2^i(n-i)}  \mathcal{A}_1^{2^{n}}.
\end{split}
\end{equation}
Next, we notice that
\begin{equation}\label{3/9 06:23}
\begin{split}
     &\sum_{i=0}^{l_n-1}2^i(n-i)\leq 2^{l_n}(n-l_n+2).
\end{split}
\end{equation}
Then, we notice that for any $0 \leq m \leq n-1 < \infty$
\begin{equation}\label{3/9 05:55}
    \begin{split}
     &   2^{n-1-m}(m+3)\leq 2^{n+1}.
    \end{split}
\end{equation}
By replacing $m$ in \eqref{3/9 05:55} by $n-l_n-1$, we have 
\begin{equation*}
    2^{l_n}(n-l_n+2) \leq 2^{n+1}.
\end{equation*}
As a result, we conclude from \eqref{3/9 06:22}, \eqref{3/9 06:23}, and \eqref{3/9 05:55} that
\begin{align*}
     &a_{n,\tau''}(\tau_2) \geq \lambda^{2^{l_n}(n-l_n+2)}  \mathcal{A}_1^{2^{n}}\geq \lambda^{2^{n+1}}  \mathcal{A}_1^{2^{n}}.
\end{align*}
Notice that we used the fact that $\lambda\leq 1$. Hence, by defining $\mathcal{A}_2(\tau_2) :=\lambda^2 \mathcal{A}_1(\tau_2)$, we conclude the proof of \textbf{Theorem \ref{Main theorem}} in the case $0\leq \alpha <1$ where $$t\in \left[\tau''+\min\left\{\frac{\delta_{T,D}(\tau'')-\tau''}{2},\tau''\right\},\tau''+\delta_{T,D}(\tau'')\right].$$
To deduce the \textbf{Theorem \ref{Main theorem 2}} from \textbf{Theorem \ref{Main theorem}}. 
We observe that 
$$$$

We consider the first term of \eqref{Duhamel formula} and \eqref{Duhamel formula B}. Then we use \textbf{Lemma \ref{up bound of L lemma}} to control the damping effect along the characteristic line $X_{s,t}(x,v)$. Thus, the solution is a super solution of the damp transport equation. That is, we either have
\begin{equation}\label{Duhamel formula later}
\begin{split}
f(t,x,v)&\geq f_{0}(X_{0,t}(x,v),v)e^{-tC_L(1+|v|^2)},
\end{split}
\end{equation}
when $t \leq t_{\partial}(x,v)$, or
\begin{equation}\label{Duhamel formula B later}
\begin{split}
f(t,x,v)&= \alpha f(t-t_{\partial}(x,v),X_{t-t_{\partial}(x,v),t}(x,v),R(X_{t-t_{\partial}(x,v),t}(x,v),v))\\
&\hspace{0.5cm}e^{-t_{\partial}(x,v)C_L(1+|v|^2)}\\
&+(1-\alpha)\left( \int_{w\cdot n(X_{t-t_{\partial}(x,v),t}(x,v))>0} f(t,X_{t-t_{\partial}(x,v),t}(x,v),w)(w\cdot n(X_{t-t_{\partial}(x,v),t}(x,v))) dw \right)\\
&\hspace{0.5cm} \frac{1}{2\pi T_B^2}e^{-\frac{|v|^2}{2T_B}} e^{-t_{\partial}(x,v)C_L(1+|v|^2)},
\end{split}
\end{equation}
when $t \geq t_{\partial}(x,v)$.
Then \textbf{Theorem \ref{Main theorem 2}} can be deduced from the comparison principle.

%To show that 
%Notice that as in \cite{Bri1, Bri2}, it suffices to show that $$f(\tau',x,v) \geq \frac{\rho}{(2\pi\theta)^{\frac{3}{2}}}e^{-\frac{|v|^2}{2\theta}},\ \forall \, (x,v) \in \overline{\Omega}\times \mathbb{R}^3$$, where $\tau'$ is given in \textbf{Proposition \ref{final lower bound extension}}.

\section{Maxwellian bound in the cut-off case for fully specular reflection condition}

In this chapter, we consider the fully specular reflection condition $\alpha=1$. %Throughout the chapter, some of the proof is omitted because of the similarity as in Chapters 3 and 4. For a detailed proof, see \cite{}. 

We first introduce with omitted proof the following estimate of the lower bound of away from the boundary, which is similar to \textbf{Proposition 4.1, \cite{Bri2}}. We note that the argument does not use the convexity of $\Omega$.
\begin{proposition}\label{Brian inner estimate}
     Suppose that $\Omega \subset \mathbb{R}^3$ satisfies \textbf{Assumption A} and that the collision kernel $B$ satisfies \textbf{Assumption B} with $\nu<0$. We consider a continuous mild solution $f(t,x,v)$ of \eqref{Boltzmann equation}--\eqref{boundary condition} with $\alpha=1$. Fix $\tau''$ as in \textbf{Proposition \ref{initial point lowerbound general sense}}. Then, for any  $\tau' \in \left(\tau'',\tau''+\delta_T(\tau'')\right]$ and $N \in \mathbb{N}$, the following holds: for any $1 \leq n \leq N$, $0<l \leq \delta_X(\tau'')$, and $$\mathfrak{R}>\max\left\{\frac{l}{\tau'-\tau''},\mathfrak{R}_{\min}(\tau'')+(\frac{3\sqrt{2}}{4})^N\delta_V(\tau'')\right\},$$ if $(t,x,v) \in \left[\tau''+\frac{l}{2^n \mathfrak{R}}, \tau'\right] \times \overline{\Omega} \times B(0,\mathfrak{R})$, and there exists $t_1 \in \left[\tau'',t-\frac{l}{2^n\mathfrak{R}}\right]$, we have $X_{t_1,t}(x,v) \in (\Omega-\Omega_l) \cap B\left(x_i,\frac{\delta_X(\tau'')}{2^n}\right)$ for some $1 \leq i \leq \mathcal{N}(N,\tau'')$, then we have

   \begin{align}\label{2025 04/07 11:25}
       &  \ \ \ \ \   f(t,x,v) \geq a^S_n(l,\tau'',\tau',\mathfrak{R})\mathbf{1}_{B(\overline{v}_{i}(\tau''),r^S_n)}(V_{t_1,t}(x,v)).
   \end{align}
Here, the notation $\delta_X(\tau'')$, $\delta_V(\tau'')$, $\delta_T(\tau'')$, and $\overline{v}_{i}(\tau'')$ are as in \textbf{Proposition \ref{initial point lowerbound general sense}} and may depend on $N$. The sequence $r^S_n$ and $a^S_n(l,\tau')$ are defined by

\begin{equation}
    \begin{cases}
        r^S_0(\tau'')=\delta_V(\tau''),\\
        r^S_{n+1}(\tau'')=\frac{3\sqrt{2}}{4}r^S_n(\tau''),
    \end{cases}
\end{equation}
   
\begin{equation}
    \begin{cases}
        a^S_0(l,\tau'',\tau',\mathfrak{R})=a_0(\tau'')e^{-C_L(\tau'-\tau'')\left\langle \mathfrak{R} \right\rangle^{\gamma^+}},\\
        a^S_{n+1}(l,\tau'',\tau',\mathfrak{R})=C_Q\frac{(r^S_n(\tau''))^{3+\gamma}l}{2^{n+4}\mathfrak{R}}e^{-(\tau'-\tau'')C_L\left\langle \mathfrak{R} \right\rangle^{\gamma^+} }a^S_n\left(\frac{l}{8},\tau'',\tau',\mathfrak{R} \right)^2,
    \end{cases}
\end{equation}
where we recall that $a_0(\tau'')$, $\mathfrak{R}_{\min}(\tau'')$ are defined in \textbf{Proposition \ref{initial point lowerbound general sense}}.  
\end{proposition}

\begin{proof}
    We prove \eqref{2025 04/07 11:25} by using an induction on n. To show that \eqref{2025 04/07 11:25} holds in the case $n=0$, we first consider the first term of the right hand side of \eqref{Duhamel formula specular}, which implies by \textbf{Lemma \ref{up bound of L lemma}} that
    \begin{equation}\label{2025 9/30 04:04}
\begin{split}
f(t,x,v)\geq f(t_1,X_{t_1,t}(x,v),V_{t_1,t}(x,v))e^{-(\tau'-\tau'')C_L\left\langle \mathfrak{R} \right\rangle^{\gamma^+}}.
\end{split}
    \end{equation}
Then we use \textbf{Proposition \ref{initial point lowerbound general sense}} to derive 
\begin{equation}\label{2025 04/07 11:56}
\begin{split}
    &f(t_1,X_{t_1,t}(x,v),V_{t_1,t}(x,v)) \\
    &\geq a_0(\tau'')\mathbf{1}_{(B(x_i,\delta_X(\tau''))\cap \overline{\Omega}) \times B(\overline{v}_{i}(\tau''),\delta_V(\tau''))}(X_{t_1,t}(x,v),V_{t_1,t}(x,v))\\
    &=a_0(\tau'')\mathbf{1}_{ B(\overline{v}_{i}(\tau''),\delta_V(\tau''))}(V_{t_1,t}(x,v)),
\end{split}
\end{equation}
where we used the assumption that $X_{t_1,t}(x,v) \in (\Omega-\Omega_l) \cap B\left(x_i,\frac{\delta_X(\tau'')}{2^n}\right)$.
    By combining \eqref{2025 9/30 04:04} and \eqref{2025 04/07 11:56}, we proved the case $n=0$.
    
    Now, we assume that the case $n=k \leq N-1$ holds. Given $0<l \leq \delta_X(\tau'')$, $(t,x,v) \in [\tau''+\frac{l}{2^{k+1} \mathfrak{R}}, \tau'] \times \overline{\Omega} \times B(0,\mathfrak{R})$ with $t_1 \in [\tau'',t-\frac{l}{2^{k+1}\mathfrak{R}}]$ such that $X_{t_1,t}(x,v) \in (\Omega-\Omega_l) \cap B(x_i,\frac{\delta_X(\tau'')}{2^{k+1}})$, we consider the second term of the right hand side of \eqref{Duhamel formula specular} and use \textbf{Lemma \ref{up bound of L lemma}} to derive that
\begin{equation}\label{2025 04/07 02:45}
    \begin{split}
&f(t,x,v)\\ 
\geq& e^{-(\tau'-\tau'')C_L\left\langle \mathfrak{R}\right\rangle^{\gamma^+}} \\
&\int_{t_1+\frac{l}{2^{k+3}\mathfrak{R}}}^{t_1+\frac{l}{2^{k+2}\mathfrak{R}}} Q^+[f(s,X_{s,t}(x,v),\cdot),f(s,X_{s,t}(x,v),\cdot)](V_{s,t}(x,v))\,ds,
\end{split}
\end{equation}
where we use the fact that $t \geq t_1+\frac{l}{2^{k+2}\mathfrak{R}}$. 
Next, we notice that
\begin{equation}
\begin{split}
    |X_{t_1,t}(x,v)-X_{s,t}(x,v)|
    \leq(s-t_1)|v| < \frac{l}{2^{k+2}} 
\end{split}
\end{equation}
for any $(t,x,v) \in [\tau''+\frac{l}{2^{k+1} \mathfrak{R}}, \tau'] \times \overline{\Omega} \times B(0,\mathfrak{R})$ with $t_1 \in [\tau'',t-\frac{l}{2^{k+1}\mathfrak{R}}]$, $s \in [t_1+\frac{l}{2^{k+3}\mathfrak{R}},t_1+\frac{l}{2^{k+2}\mathfrak{R}}]$, which implies that 
\begin{equation}
    X_{s,t}(x,v) \in \Omega-\Omega_{l-\frac{l}{2^{k+2}}}.
\end{equation}
We also deduce that
\begin{equation}\label{2025 04/07 03:07}
   V_{s,t}(x,v)=V_{t_1,t}(x,v)
\end{equation}
since under specular reflection, the backward characteristic velocity is piecewise constant and identical on $[t_1,s]$ without collisions.
Next, we notice that for $u \in B(0,\mathfrak{R})$,
\begin{equation}
\begin{split}
    &|X_{t_1,s}(X_{s,t}(x,v),u) -X_{t_1,t}(x,v)| \\
    \leq &|X_{t_1,s}(X_{s,t}(x,v),u) -X_{s,t}(x,v)|+|X_{s,t}(x,v) -X_{t_1,t}(x,v)| \\
    \leq &(s-t_1)|u|+(s-t_1)|v| \leq \frac{l}{2^{k+1}} \leq \frac{\delta_X(\tau'')}{2^{k+1}},
\end{split}
\end{equation}
which implies
\begin{equation}  X_{t_1,s}(X_{s,t}(x,v),u) \in (\Omega-\Omega_{\frac{l}{8}})\cap B\left(x_i,\frac{\delta_X(\tau'')}{2^k}\right),  
\end{equation}
and 
\begin{equation}\label{2025 09/30 04:57}
    V_{t_1,s}(X_{s,t}(x,v),u)=u,
\end{equation}
for $u \in B(0,\mathfrak{R})$. In addition, for any $s \in [t_1+\frac{l}{2^{k+3}\mathfrak{R}},t_1+\frac{l}{2^{k+2}\mathfrak{R}}]$, we have
\begin{equation}
  \tau''+  \frac{\frac{l}{8}}{2^k \mathfrak{R}} \leq s \leq \tau',
\end{equation}
\begin{equation}
    t_1 \in \left[\tau'',s-\frac{\frac{l}{8}}{2^k \mathfrak{R}}\right].
\end{equation}
As a result, the assumptions of induction are satisfied and we deduce that
\begin{equation}
\begin{split}
     f(s,X_{s,t}(x,v),u)\geq a^S_k\left(\frac{l}{8},\tau'',\tau',\mathfrak{R}\right)\mathbf{1}_{B(\overline{v}_{i}(\tau''),r^S_k)}(u).
\end{split}
\end{equation}
for $(t,x,v) \in \left[\tau''+\frac{l}{2^{k+1} \mathfrak{R}}, \tau'\right] \times \overline{\Omega} \times B(0,\mathfrak{R})$ and  $t_1 \in [\tau'',t-\frac{l}{2^{k+1}\mathfrak{R}}]$ with $s \in \left[t_1+\frac{l}{2^{k+3}\mathfrak{R}},t_1+\frac{l}{2^{k+2}\mathfrak{R}}\right]$, $u \in B(0,\mathfrak{R})$. Here, we used \eqref{2025 09/30 04:57}.

Hence, we have 
\begin{equation}\label{2025 04/07 02:48}
\begin{split}
    &Q^+[f(s,X_{s,t}(x,v),\cdot),f(s,X_{s,t}(x,v),\cdot)](V_{s,t}(x,v)) \\
    \geq& a^S_k\left(\frac{l}{8},\tau'',\tau',\mathfrak{R}\right)^2 Q^+[\mathbf{1}_{B(\overline{v}_{i}(\tau''),r^S_k)\cap B(0,\mathfrak{R})}(\cdot),\mathbf{1}_{B(\overline{v}_{i}(\tau''),r^S_k)\cap B(0,\mathfrak{R})}(\cdot)](V_{s,t}(x,v))\\
    =&a^S_k\left(\frac{l}{8},\tau'',\tau',\mathfrak{R}\right)^2 Q^+[\mathbf{1}_{B(\overline{v}_{i}(\tau''),r^S_k)}(\cdot),\mathbf{1}_{B(\overline{v}_{i}(\tau''),r^S_k)}(\cdot)](V_{s,t}(x,v)).
\end{split}
\end{equation}
 Here, we use the fact that $\mathfrak{R} > \mathfrak{R}_{\min}(\tau'')+(\frac{3\sqrt{2}}{4})^N\delta_V(\tau'') \geq |\overline{v}_{i}(\tau'')|+r^S_k$ for the derivation of the last line.
 Now, we use \textbf{Lemma \ref{spreading property lemma}} with $\xi=\frac{1}{2}$ to derive
 \begin{equation}\label{2025 04/07 02:46}
 \begin{split}
     Q^+[\mathbf{1}_{B(\overline{v}_{i}(\tau''),r^S_k)}(\cdot),\mathbf{1}_{B(\overline{v}_{i}(\tau''),r^S_k)}(\cdot)] (V_{s,t}(x,v)) \geq C_Q (r^S_k)^{3+\gamma}\frac{1}{2}\mathbf{1}_{B(\overline{v}_{i}(\tau''),r^S_{k+1})}(V_{s,t}(x,v)).
 \end{split}
 \end{equation}
Hence, we deduce by \eqref{2025 04/07 02:45},  \eqref{2025 04/07 03:07}, \eqref{2025 04/07 02:48}, and \eqref{2025 04/07 02:46} that
\begin{equation}
\begin{split}
&f(t,x,v)\\ 
\geq &e^{-(\tau'-\tau'')C_L\left\langle \mathfrak{R}\right\rangle^{\gamma^+}} \int_{t_1+\frac{l}{2^{k+3}\mathfrak{R}}}^{t_1+\frac{l}{2^{k+2}\mathfrak{R}}} Q^+[f(s,X_{s,t}(x,v),\cdot),f(s,X_{s,t}(x,v),\cdot)](V_{s,t}(x,v))\,ds\\
\geq &e^{-(\tau'-\tau'')C_L\left\langle \mathfrak{R}\right\rangle^{\gamma^+}} \int_{t_1+\frac{l}{2^{k+3}\mathfrak{R}}}^{t_1+\frac{l}{2^{k+2}\mathfrak{R}}} a^S_k\left(\frac{l}{8},\tau'',\tau',\mathfrak{R}\right)^2C_Q (r^S_k)^{3+\gamma}\frac{1}{2}\mathbf{1}_{B(\overline{v}_{i}(\tau''),r^S_{k+1})}(V_{s,t}(x,v))\,ds\\
=&e^{-(\tau'-\tau'')C_L\left\langle \mathfrak{R}\right\rangle^{\gamma^+}}  a^S_k\left(\frac{l}{8},\tau'',\tau',\mathfrak{R}\right)^2C_Q (r^S_k)^{3+\gamma}\frac{1}{2}\int_{t_1+\frac{l}{2^{k+3}\mathfrak{R}}}^{t_1+\frac{l}{2^{k+2}\mathfrak{R}}} \mathbf{1}_{B(\overline{v}_{i}(\tau''),r^S_{k+1})}(V_{s,t}(x,v))\,ds\\
\geq &e^{-(\tau'-\tau'')C_L\left\langle \mathfrak{R}\right\rangle^{\gamma^+}}  a^S_k\left(\frac{l}{8},\tau'',\tau',\mathfrak{R}\right)^2C_Q (r^S_k)^{3+\gamma}\frac{1}{2}\frac{l}{2^{k+3}\mathfrak{R}}\mathbf{1}_{B(\overline{v}_{i}(\tau''),r^S_{k+1})}(V_{t_1,t}(x,v))\\
\geq &a^S_{k+1}(l,\tau'',\tau',\mathfrak{R})\mathbf{1}_{B(\overline{v}_{i}(\tau''),r^S_{k+1})}(V_{t_1,t}(x,v)).
\end{split}
\end{equation}
As a result, the induction hypothesis for $n=k+1$ holds. By induction, we deduce that \eqref{2025 04/07 11:25} holds for $0 \leq n \leq N$ and we conclude the result.

\end{proof}

The following corollary is obtained by a slight modification of the previous proposition (\textbf{Corollary 4.2, \cite{Bri2}}):
\begin{corollary}\label{gazing coro}
     Suppose that $\Omega \subset \mathbb{R}^3$ satisfies \textbf{Assumption A} and that the collision kernel $B$ satisfies \textbf{Assumption B} with $\nu<0$. Let $f(t,x,v)$ be a continuous mild solution of \eqref{Boltzmann equation}--\eqref{boundary condition} with $\alpha=1$. Fix $\tau''>0$ as in \textbf{Proposition \ref{initial point lowerbound general sense}}. Then, given $\Delta^I_T \in (0, \delta_T(\tau'')]$ with $\tau^A \in (\tau'',\tau''+\Delta^I_T]$, for all $0<l \leq \delta_X(\tau'')$, there exist $a^S(l,\tau^A,\tau'',\Delta^I_T)>0$ and $\tilde{t}(l,\tau'',\tau^A,\Delta^I_T) \in \left(0,\frac{\tau^A-\tau''}{3}\right]$, depending on $\tau''$, $\Omega$, $M$, and $E_f$(and $L_{f,p}$ if $\gamma <0$) such that we have the following: for $(t,x,v) \in [\tau^A,\tau''+\Delta^I_T] \times \overline{\Omega} \times \mathbb{R}^3$, if for some $t_1 \in [\tau'',t-\tilde{t}(l,\tau'',\tau^A,\Delta^I_T)]$ we have $X_{t_1,t}(x,v) \in \Omega-\Omega_{l}$, then we have

   \begin{align} \label{2025/10/07 02:10}
       &  \ \ \ \ \   f(t,x,v) \geq a^S(l,\tau^A,\tau'',\Delta^I_T)\mathbf{1}_{B(0,3\mathfrak{R}_{\min}(\tau''))}(v).
   \end{align}
   Here, $\mathfrak{R}_{\min}, \delta_T(\tau'')$, and $\delta_X(\tau'')$ are given in \textbf{Proposition \ref{initial point lowerbound general sense}}.
\end{corollary}

\begin{proof}

    Given $0<l \leq \delta_X(\tau'')$ and $\tau^A \in (\tau'',\tau''+\Delta^I_T]$. We define

    \begin{equation}
        \tilde{N}_l(\tau'',\tau^A):=\left\lceil \max \left\{\log_{\frac{3\sqrt{2}}{4}}\left(\frac{4\mathfrak{R}_{\min}(\tau'')}{\delta_V(\tau'')}\right),\log_{2}\frac{3l}{\tau^A-\tau''} \right\} \right\rceil, 
    \end{equation}
    \begin{equation}
        \tilde{t}(l,\tau'',\tau^A,\Delta^I_T):=\frac{l}{2^{\tilde{N}_l(\tau'',\tau^A)}\left[\max\left\{\frac{l}{\Delta^I_T},\mathfrak{R}_{\min}(\tau'')+\left(\frac{3\sqrt{2}}{4}\right)^{\tilde{N}_l(\tau'',\tau^A)}\delta_V(\tau'')\right\}+1\right]},
    \end{equation}
    \begin{equation}
        \mathfrak{R}_N^S(l,\tau'',\tau^A):=1+\max\left\{\frac{l}{\tau^A-\tau''},\mathfrak{R}_{\min}(\tau'')+\left(\frac{3\sqrt{2}}{4}\right)^{\tilde{N}_l(\tau'',\tau^A)}\delta_V(\tau'')\right\}.
    \end{equation}
    By \textbf{Proposition \ref{initial point lowerbound general sense}}, we know that $\mathfrak{R}_{\min}(\tau'') \geq 2$ and $\delta_V(\tau'') \leq \frac{1}{56}$, which imply $\tilde{N}_l(\tau'',\tau^A) \geq 8$. We observe that
\begin{equation}
\begin{split}
        \tilde{t}(l,\tau'',\tau^A,\Delta^I_T) \leq \frac{l}{2^{\tilde{N}_l(\tau'',\tau^A)}} \leq \frac{\tau^A-\tau''}{3}.
\end{split}
\end{equation}
    The last inequality follows directly from the definition of $\tilde N_l$, which leads to $2^{\tilde{N}_l(\tau'',\tau^A)}\geq \frac{3l}{\tau^A-\tau''}$.
    
     Now, given $(t,x,v) \in [\tau^A,\tau''+\Delta^I_T] \times \overline{\Omega} \times \mathbb{R}^3$ with $t_1 \in [\tau'',t-\tilde{t}(l,\tau'',\tau^A,\Delta^I_T)]$ such that $X_{t_1,t}(x,v) \in \Omega-\Omega_{l}$. We consider $\{x_i\}_{i=1}^{\mathcal{N}(\tilde{N}_l(\tau'',\tau^A),\tau'')}$ from \textbf{Proposition \ref{initial point lowerbound general sense}} such that $$ \overline{\Omega} \subset \bigcup_{1\leq i \leq \mathcal{N}(\tilde{N}_l(\tau'',\tau^A),\tau'')} B\left(x_i, \frac{\delta_X(\tau'')}{2^{\tilde{N}_l(\tau'',\tau^A)}} \right)$$. There exists $1\leq i \leq \mathcal{N}(\tilde{N}_l(\tau'',\tau^A),\tau'')$ such that $X_{t_1,t}(x,v) \in B\left(x_i, \frac{\delta_X(\tau'')}{2^{\tilde{N}_l(\tau'',\tau^A)}}\right)$.
    Next, we find that
    \begin{equation}
        \tilde{t}(l,\tau'',\tau^A,\Delta^I_T) \geq \frac{l}{2^{\tilde{N}_l(\tau'',\tau^A)}\mathfrak{R}_N^S(l,\tau'',\tau^A)},
    \end{equation}
    which leads to
\begin{equation}
     t_1 \in \left[\tau'',t-\frac{l}{2^{\tilde{N}_l(\tau'',\tau^A)}\mathfrak{R}_N^S(l,\tau'',\tau^A)}\right].
\end{equation}
Moreover, we have
\begin{equation}
    \mathfrak{R}_N^S(l,\tau'',\tau^A)>\max\left\{\frac{l}{\Delta^I_T},\mathfrak{R}_{\min}(\tau'')+(\frac{3\sqrt{2}}{4})^{\tilde{N}_l(\tau'',\tau^A)}\delta_V(\tau'')\right\}.
\end{equation}

    Thus, we apply \textbf{Proposition \ref{Brian inner estimate}} with $\tau'=\tau''+\Delta^I_T$, $n=N = \tilde{N}_l(\tau'',\tau^A)$, and $\mathfrak{R}=\mathfrak{R}_N^S(l,\tau'',\tau^A)$ and deduce that for any $$(t,x,v) \in \left[\tau''+\frac{l}{2^{\tilde{N}_l(\tau'',\tau^A)} \mathfrak{R}_N^S(l,\tau'',\tau^A)}, \tau''+ \Delta^I_T\right] \times \overline{\Omega} \times B(0,\mathfrak{R}_N^S(l,\tau'',\tau^A)),$$
we have
   \begin{align*}
          f(t,x,v) &\geq a^S_{\tilde{N}_l(\tau'',\tau^A)}(l,\tau'',\tau''+\Delta^I_T,\mathfrak{R}_N^S(l,\tau'',\tau^A))\mathbf{1}_{B\left(\overline{v}_{i}(\tau''),r^S_{\tilde{N}_l(\tau'',\tau^A)}(\tau'')\right)}(V_{t_1,t}(x,v))\\
          &\geq a^S_{\tilde{N}_l(\tau'',\tau^A)}(l,\tau'',\tau''+\Delta^I_T,\mathfrak{R}_N^S(l,\tau'',\tau^A))\mathbf{1}_{B\left(0,3\mathfrak{R}_{\min}(\tau'')\right)}(V_{t_1,t}(x,v))\\
       &= a^S_{\tilde{N}_l(\tau'',\tau^A)}(l,\tau'',\tau''+\Delta^I_T,\mathfrak{R}_N^S(l,\tau'',\tau^A))\mathbf{1}_{B(0,3\mathfrak{R}_{\min}(\tau''))}(v).
   \end{align*}
Here, we used the fact that $r^S_{\tilde{N}_l(\tau'',\tau^A)}(\tau'') \geq 4\mathfrak{R}_{\min}(\tau'') \geq 3\mathfrak{R}_{\min}(\tau'')+|\overline{v}_{i}(\tau'')| $ for the second line.
   % So, we have
    %\begin{equation}
     %   t \in [\tau''+\tilde{t}(l,\tau'',\tau^A,\Delta^I_T), \tau''+\Delta_T^I],
    %\end{equation}

    Moreover, we have
    $$[\tau^A,\tau''+\Delta^I_T]\subset \left[\tau''+\frac{l}{2^{\tilde{N}_l(\tau'',\tau^A)} \mathfrak{R}_N^S(l,\tau'',\tau^A)}, \tau''+\Delta^I_T\right].$$
Finally, by the choice of $\tilde{N}_l(\tau'',\tau^A)$, we have 
$$\mathfrak{R}_N^S(l,\tau'',\tau^A) \geq 4\mathfrak{R}_{\min}(\tau'').$$ As a result, by defining 
\begin{equation}
a^S(l,\tau'',\tau^A,\Delta^I_T):=a^S_{\tilde{N}_l(\tau'',\tau^A)}(l,\tau'',\tau''+\Delta^I_T,\mathfrak{R}_N^S(l,\tau'',\tau^A)).
\end{equation}
we obtain the desired lower bound \eqref{2025/10/07 02:10}, which completes the proof.

\end{proof}

   Before stating the next proposition, we recall the quantities $\delta_X(\tau'')$, $\delta_V(\tau'')$, $\delta_T(\tau'')$, and $\mathfrak{R}_{\min}(\tau'')$ from \textbf{Proposition \ref{initial point lowerbound general sense}}, and notice that $\delta_T(\tau'')\leq\delta_X(\tau'')=\delta_V(\tau'') \leq \frac{d}{56}$ and $\mathfrak{R}_{\min}(\tau'')\geq 2$. We now introduce the following auxiliary sequences:
\begin{equation}
    \begin{cases}
        r^B_0(\tau'')=\delta_V(\tau''),\\
        r^B_{n+1}(\tau'')=\frac{3\sqrt{2}}{4}r^B_{n}(\tau'')-\frac{\delta_V(\tau'')}{40}.
    \end{cases}
\end{equation}
Observe that $\{r^B_{n}(\tau'')\}_{n=0}^{\infty}$ is strictly increasing and unbounded, that $$r^B_0(\tau'') <1 \leq \mathfrak{R}_{\min}(\tau''),$$ and that 
\begin{equation}
    \frac{2\mathfrak{R}_{\min}(\tau'')}{\mathfrak{R}_{\min}(\tau'')+1} \geq \frac{4}{3} > \frac{3\sqrt{2}}{4}.
\end{equation}
Hence, the following notation 
\begin{equation}\label{tiring small estimate}
   N_B(\tau''):=\min\{ n \mid2\mathfrak{R}_{\min}(\tau'') \geq r^B_{n}(\tau'') > \mathfrak{R}_{\min}(\tau'')+1 \}
\end{equation}
is well-defined.
We also define
\begin{equation}\label{2025 10 08 04:37}
    \vartheta_n^B(\tau''):= \frac{1}{2^{n+1}\left(r^B_{n}(\tau'')+\mathfrak{R}_{\min}(\tau'')\right)}.
\end{equation}
\begin{equation}\label{2025 11/12 04:14}
     \delta_T^B(\tau'') := \min\left\{ t_{\frac{\delta_V(\tau'')}{40}}\left(\mathfrak{R}_{\min}(\tau'')+r^B_{N_B(\tau'')}\right), \delta_T(\tau''),\frac{d}{4\left(\mathfrak{R}_{\min}(\tau'')+r_{N_B(\tau'')}^B(\tau'')\right)} \right\},
\end{equation}
 \begin{equation}
a^B_{n,\tau''}(\tau^B):=
\begin{cases}
a_0(\tau''),\ & n=0\\
\left(\min\{a^B_{n-1,\tau''}(\tau^B),b^B_{\tau''}(\tau^B)\}\right)^2C_Q(r^B_{n-1}(\tau''))^{3+\gamma}\\
\frac{\tau^B-\tau''}{2^{n+2}(\mathfrak{R}_{\min}(\tau'')+r^B_{n}(\tau''))}e^{-C_L\frac{(\tau^B-\tau'')\left\langle \mathfrak{R}_{\min}(\tau'')+r^B_{n}(\tau'') \right\rangle^{\gamma^+}}{2^{n}\left(\mathfrak{R}_{\min}(\tau'')+r^B_{n}(\tau'')\right)}},\  &n \geq 1,
\end{cases}
  \end{equation}

   \begin{equation}
b^B_{\tau''}(\tau^B):=a^S\left(l_{\frac{\delta_V(\tau'')}{40}}\left(\frac{\tau^B-\tau''}{6}\right),\tau^B,\tau'',\delta_T^B(\tau'')\right).
  \end{equation}
Here, the notation $a_0(\tau'')$ is from \textbf{Proposition \ref{initial point lowerbound general sense}} and the notation $a^S$ is from \textbf{Corollary \ref{gazing coro}}. %We emphasize that the parameters $\tau''$ and $\tau^B$ are fixed throughout this construction, while $n$ denotes the induction index.

\begin{proposition} \label{initial point lowerbound general sense fully specular}
   Suppose that $\Omega \subset \mathbb{R}^3$ satisfies \textbf{Assumption A} and that the collision kernel $B$ satisfies \textbf{Assumption B}. Fix $d$ with $0< d <\min\{1,\delta\}$ as in \textbf{Lemma \ref{initial cover over boundary}}. Let $f(t,x,v)$ be a continuous mild solution of \eqref{Boltzmann equation}--\eqref{boundary condition} with $\alpha=1$ and let $\tau''>0$ be as given in \textbf{Proposition \ref{initial point lowerbound general sense}}. The following holds: For any $\tau^B \in (\tau'',\tau''+\delta_T^B(\tau'')]$, there exist $a^B_0(\tau^B)>0$,  which depends on $\tau''$, $\Omega$, $M$, and $E_f$, such that 
      \begin{equation}\label{2025 10/08 04:02}
 f(t,x,v) \geq \min\left\{a^B_{N_B(\tau''),\tau''}(\tau^B),b^B_{\tau''}(\tau^B)\right\}\mathbf{1}_{B(0,1)}(v)
     \end{equation}
for any $$(t,x,v) \in \left[\tau''+(\tau^B-\tau'')\left(1-\vartheta_{N_B(\tau'')}^B(\tau'')\right),\tau^B\right]\times \Omega\times \mathbb{R}^3.$$

\noindent

Here, $N_B(\tau'')$ is defined in \eqref{tiring small estimate} and $\mathfrak{R}_{\min}(\tau'')$ is introduced in \textbf{Proposition \ref{initial point lowerbound general sense}} and $\vartheta_{N_B(\tau'')}$ is from \eqref{2025 10 08 04:37}.

\end{proposition}

\begin{proof}
We start in a similar way to the proof of \textbf{Proposition \ref{final lower bound extension}}. We apply \textbf{Proposition \ref{initial point lowerbound general sense}} to $\overline{\Omega}$ to deduce the existence of points $\{ x_j \}_{j=1}^{\mathcal{N}(N_B(\tau''),\tau'')} \subset \Omega$ such that $\overline{\Omega} \subset \underset{1\leq j\leq \mathcal{N}(N_B(\tau''),\tau'')}{\bigcup} B(x_j,\frac{\delta_X(\tau'')}{2^{N_B(\tau'')}})$ and 
\begin{equation}\label{2025 10/07 03:55}
 f(t,x,v) \geq a_0(\tau'')\mathbf{1}_{B(\overline{v}_{j}(\tau''),\delta_V(\tau''))}(v)
     \end{equation}
for any $1\leq j \leq \mathcal{N}(N_B(\tau''),\tau'')$ and any $(t,x,v) \in [\tau'', \tau''+\delta_{T}(\tau'')]\times [B(x_j,\delta_X(\tau''))\cap \overline{\Omega}]\times \mathbb{R}^3$.

Then, we recall the notation $t_{\epsilon}(v_M)$ in \eqref{t es and l es} and $\delta_T^B(\tau'')$ in \eqref{2025 11/12 04:14}
\begin{equation}
     \delta_T^B(\tau'') := \min\left\{ t_{\frac{\delta_V(\tau'')}{40}}\left(\mathfrak{R}_{\min}(\tau'')+r^B_{N_B(\tau'')}\right), \delta_T(\tau''),\frac{d}{4\left(\mathfrak{R}_{\min}(\tau'')+r_{N_B(\tau'')}^B(\tau'')\right)} \right\},
\end{equation}

Now, given $\tau^B \in (\tau'',\tau''+\delta_T^B(\tau'')]$ and $0\leq n \leq N_B(\tau'')$, we show that 
 \begin{equation}\label{boundary hard estimate}
 \begin{split}
     &  \ \ \ \ \   f(t,x,v) \geq \min\{a^B_{n,\tau''}(\tau^B),b^B_{\tau''}(\tau^B)\}\mathbf{1}_{B(\overline{v}_{i}(\tau''),r^B_{n}(\tau''))}(v),\\
       &\forall (t,x) \in \left[\tau''+(\tau^B-\tau'')\left(1-\vartheta_n^B(\tau'')\right),\tau^B\right]\times B\left(x_i,\frac{\delta_X(\tau'')}{2^n}\right)
 \end{split}    
       \end{equation}
   for $1\leq i \leq \mathcal{N}(N_B(\tau''),\tau'')$.

The proof is similar to the proof of \textbf{Proposition \ref{final lower bound extension}} and we use an induction on $n$. The base case $n=0$ follows directly from \eqref{2025 10/07 03:55}. Assume that the case $n=k$ holds. Given 
\begin{equation}\label{2025 4/2 03:59}
\begin{split}
    (t,x,v) \in & \left[\tau''+(\tau^B-\tau'')\left(1-\vartheta_{k+1}^B(\tau'')\right),\tau^B\right]\\
    &\times B\left(x_i,\frac{\delta_X(\tau'')}{2^{k+1}}\right) \times B(\overline{v}_{i}(\tau''),r^B_{k+1}(\tau'')),
\end{split}
\end{equation}
we use the Duhamel formula \eqref{Duhamel formula} and apply \textbf{Lemma \ref{up bound of L lemma}} to deduce that
    \begin{equation}\label{2025 10/7 12:31}
    \begin{split}
        &f(t,x,v)\\
        \geq &\int_{\tau''+(\tau^B-\tau'')\left(1-2\vartheta_{k+1}^B(\tau'')\right)}^{\tau''+(\tau^B-\tau'')\left(1-\vartheta_{k+1}^B(\tau'')\right)} e^{-C_L(t-s)\left\langle r^B_{k+1}(\tau'')+\mathfrak{R}_{\min}(\tau'')\right\rangle^{\gamma^+}} \\ 
        &Q^+[f(s,X_{s,t}(x,v),\cdot),f(s,X_{s,t}(x,v),\cdot)](V_{s,t}(x,v))\,ds.
    \end{split}
    \end{equation}
Since for any $s \in \left[\tau''+(\tau^B-\tau'')\left(1-2\vartheta_{k+1}^B(\tau'')\right), t\right] $, we have
\begin{equation}\label{2025 10/08/ 12:56}
    \begin{split}
        &|X_{s,t}(x,v)-x_i|
        \leq |X_{s,t}(x,v)-x|+|x-x_i| \leq |v|(t-s)+\frac{\delta_X(\tau'')}{2^{k+1}}\\
        \leq &\frac{\tau^B-\tau''}{2^{k+2}}+\frac{\delta_X(\tau'')}{2^{k+1}} \leq \frac{\delta_{T}^B(\tau'')}{2^{k+2}}+\frac{\delta_X(\tau'')}{2^{k+1}}  \leq \frac{\delta_X(\tau'')}{2^{k}} <\frac{d}{8}.
    \end{split}
\end{equation}
Here recall the definition of $\delta_T^B(\tau'')$ and $\delta_T(\tau'')$ from \eqref{2025 11/12 04:14} and \eqref{2025 11/12 16:21}.
  We next apply the induction assumption and \textbf{Lemma \ref{spreading property lemma}} to derive that

\begin{equation}
    \begin{split}
        &Q^+[f(s,X_{s,t}(x,v),\cdot),f(s,X_{s,t}(x,v),\cdot)](V_{s,t}(x,v)) \\
        \geq& C_Q  (r^B_{k}(\tau''))^{3+\gamma}\frac{1}{2}\min\left(a^B_{k,\tau''}(\tau^B),b^B_{\tau''}(\tau^B)\right)^2 \mathbf{1}_{B\left(\overline{v}_{i}(\tau''),r^B_{k+1}(\tau'')+\frac{\delta_V(\tau'')}{40}\right)}(V_{s,t}(x,v)).
    \end{split}
\end{equation}

We will generate the lower bound again depending on whether the characteristic line $\{X_s(x,v)\}_{\tau''+(\tau^B-\tau'')\left(1-2\vartheta_{k+1}^B(\tau'')\right) \leq s \leq t}$ touches the boundary or not.

If $|V_{s,t}(x,v)-v| < \frac{\delta_V(\tau'')}{40} $ for $\tau''+(\tau^B-\tau'')\left(1-2\vartheta_{k+1}^B(\tau'')\right) \leq s \leq t$, we have 
\begin{equation}
   \mathbf{1}_{B\left(\overline{v}_{i}(\tau''),r^B_{k+1}(\tau'')+\frac{\delta_V(\tau'')}{40}\right)}(V_{s,t}(x,v))=1,
\end{equation}
for $\tau''+(\tau^B-\tau'')\left(1-2\vartheta_{k+1}^B(\tau'')\right) \leq s \leq t$.
We deduce that
\begin{equation}
    \begin{split}
        &f(t,x,v)\\
        \geq& C_Q(r^B_{k}(\tau''))^{3+\gamma}\frac{1}{2}\min\left(a^B_{k,\tau''}(\tau^B),b^B_{\tau''}(\tau^B)\right)^2 \\
        &\int_{\tau''+(\tau^B-\tau'')\left(1-2\vartheta_{k+1}^B(\tau'')\right)}^{\tau''+(\tau^B-\tau'')\left(1-\vartheta_{k+1}^B(\tau'')\right)} e^{-C_L(t-s)\langle r^B_{k+1}(\tau'')+\mathfrak{R}_{\min}(\tau'')\rangle^{\gamma^+}} \,ds\\
        \geq& C_Q  (r^B_{k}(\tau''))^{3+\gamma}\frac{1}{2}\min\left(a^B_{k,\tau''}(\tau^B),b^B_{\tau''}(\tau^B)\right)^2 \frac{\tau^B-\tau''}{2^{k+2}(\mathfrak{R}_{\min}(\tau'')+r^B_{k+1}(\tau''))}\\
        &\times e^{-C_L\frac{(\tau^B-\tau'')\left\langle \mathfrak{R}_{\min}(\tau'')+r^B_{k+1}(\tau'') \right\rangle^{\gamma^+}}{2^{k+1}\left(\mathfrak{R}_{\min}(\tau'')+r^B_{k+1}(\tau'')\right)}}\\
        =&a^B_{k+1,\tau''}(\tau^B).
    \end{split}
    \end{equation}

On the other hand, if 
$|V_{s,t}(x,v)-v| \geq \frac{\delta_V(\tau'')}{40} $ for some $\tau''+(\tau^B-\tau'')\left(1-2\vartheta_{k+1}^B(\tau'')\right) \leq s \leq t$, we have $X_{s,t}(x,v) \in \partial\Omega$ for some $$s \in  \left[\tau''+(\tau^B-\tau'')\left(1-2\vartheta_{k+1}^B(\tau'')\right), t\right].$$ By \eqref{2025 10/08/ 12:56}, it follows that $x_i \in \Omega_{\frac{d}{8}}$, so by \textbf{Remark \ref{remark 2}}, there exists $1 \leq k(i)\leq m_1 $ such that $x_i  \in B(x^0_{k(i)}, \frac{d}{4})$.
%$$V_{s,t}(x,v)-v=V_{t-s}(X_{s,t}(x,v),V_{s,t}(x,v))-V_{s,t}(x,v)$$
Notice that  $|V_{s,t}(x,v)-v| \geq \frac{\delta_V(\tau'')}{40} $ for some $\tau'' \leq s \leq t$
Using \textbf{Lemma \ref{grazing geo lemma}}, we deduce that for any $0 \leq \tau_2 \leq t-\tau''$, we have $$X_{s',t}(x,v) \notin \Omega_{l_{\frac{\delta_V(\tau'')}{40}}(\tau_2)}\cap B(x_i^0,d)$$
for some $$s' \in [\tau'',\tau''+\tau_2],$$

which implies since $X_{s',t}(x,v) \in B(x^0_i,d)$:
$$X_{s',t}(x,v) \in \Omega- \Omega_{l_{\frac{\delta_V(\tau'')}{40}}(\tau_2)}.$$
Since we have
\begin{equation}
    \begin{split}
        t-\tau''
        \geq \tau''+(\tau^B-\tau'')\left(1-\vartheta_{k+1}^B(\tau'')\right)-\tau''
        =\frac{\tau^B-\tau''}{2} ,
    \end{split}
\end{equation}
we can take $\tau_2=\frac{\tau^B-\tau''}{6}$.
Next, we observe that
\begin{equation*}
    \begin{split}
        \tau''+\tau_2=\tau''+\frac{\tau^B-\tau''}{3} \leq t-\frac{\tau^B-\tau''}{3} \leq t-\tilde{t}\left(l_{\frac{\delta_V(\tau'')}{40}}\left(\frac{\tau^B-\tau''}{6}\right),\tau'',\tau^B,\delta_T^B(\tau'')\right).
    \end{split}
\end{equation*}
Consequently, we have
$$s' \in \left[\tau'',t-\tilde{t}\left(l_{\frac{\delta_V(\tau'')}{40}}\left(\frac{\tau^B-\tau''}{6}\right),\tau'',\tau^B,\delta_T^B(\tau'')\right)\right] $$ and 
$$X_{s',t}(x,v) \in \Omega- \Omega_{l_{\frac{\delta_V(\tau'')}{40}}(\tau_2)}.$$
We also notice that $$l_{\frac{\delta_V(\tau'')}{40}}\left(\frac{\tau^B-\tau''}{6}\right)= \frac{(\tau^B-\tau'')}{15360}\delta_V(\tau'')\leq \delta_X(\tau'').$$
Therefore, we apply \textbf{Corollary \ref{gazing coro}} with parameters 
$$(l,\tau^A,\tau'',\Delta^I_T)\rightarrow\left(l_{\frac{\delta_V(\tau'')}{40}}\left(\frac{\tau^B-\tau''}{6}\right),\tau^B,\tau'',\delta_T^B(\tau'')\right)$$ to deduce that
\begin{align*}
       &    f(t,x,v) \geq a^S\left(l_{\frac{\delta_V(\tau'')}{40}}\left(\frac{\tau^B-\tau''}{6}\right),\tau^B,\tau'',\delta_T^B(\tau'')\right)\mathbf{1}_{B(0,3\mathfrak{R}_{\min}(\tau''))}(v)\\
       & =a^S\left(l_{\frac{\delta_V(\tau'')}{40}}\left(\frac{\tau^B-\tau''}{6}\right),\tau^B,\tau'',\delta_T^B(\tau'')\right),
   \end{align*}
   which coincides with the definition of $b^B_{\tau''}(\tau^B)$ given above.
  Here, we used the fact $$|v| \leq |\overline{v}_i(\tau'')|+r^B_{k+1}(\tau'') \leq \mathfrak{R}_{\min}(\tau'')+2\mathfrak{R}_{\min}(\tau'') \leq 3\mathfrak{R}_{\min}(\tau'').$$ Combining both cases, we conclude that the induction step holds, and hence \eqref{boundary hard estimate} is valid for $1 \leq n \leq N_B(\tau'')$.
Finally, we set $n=N_B(\tau'')$ in \eqref{boundary hard estimate}, and use the fact that
$B(0,1) \subset B(\overline{v}_{i}(\tau''),r_{N_B(\tau'')}^B(\tau''))$ (because of \eqref{tiring small estimate}) to conclude that for $1\ \leq i \leq \mathcal{N}(N_B(\tau''),\tau'')$
 \begin{equation} 
 \begin{split}
        f(t,x,v) &\geq \min\left\{a^B_{N_B(\tau''),\tau''}(\tau^B),b^B_{\tau''}(\tau^B)\right\}\mathbf{1}_{B(\overline{v}_{i}(\tau''),r_{N_B(\tau'')}^B(\tau''))}(v),\\
     & \geq\min\left\{a^B_{N_B(\tau''),\tau''}(\tau^B),b^B_{\tau''}(\tau^B)\right\}\mathbf{1}_{B(0,1)}(v),\\
       &\forall (t,x) \in \left[\tau''+(\tau^B-\tau'')\left(1-\vartheta_{N_B(\tau'')}^B(\tau'')\right),\tau^B\right]\times B\left(x_i,\frac{\delta_X(\tau'')}{2^{N_B(\tau'')}}\right),
 \end{split}    
       \end{equation}
which implies \eqref{2025 10/08 04:02}.
    This completes the proof.
 
\end{proof}

Now, we introduce new notations:
\begin{definition}
  Given $0<\tau''<\tau^S$ as in \textbf{Proposition \ref{initial point lowerbound general sense fully specular}} , we define 

\begin{equation}
  \tilde{\tau}(\tau'',\tau^S):= \tau''+(\tau^S-\tau'')\left(1-\vartheta_{N_B(\tau'')}^B(\tau'')\right).
\end{equation}
Clearly, we see that $\tau''<\tilde{\tau}(\tau'',\tau^S)<\tau^S$.
\end{definition}

\begin{definition}
    
  Given $0<\tau''<\tau^S$, and a sequence $\{ \xi_n \} \in (0,1)^{\mathbb{N}}$,  we define the following numbers:
  \begin{equation}
r''_{n,S}:=
\begin{cases}
1,\ n=0\\
\sqrt{2}(1-\xi_{n})r''_{n-1,S},\ n \geq 1,
\end{cases}
  \end{equation}
 and
  \begin{equation}
a^S_{n,\tau''}(\tau^S):=
\begin{cases}
  \min\left\{a^B_{N_B(\tau''),\tau''}(\tau^S),b^B_{\tau''}(\tau^S)\right\},\ & n=0\\
(a^S_{n-1,\tau''}(\tau^S))^2 C_Q (r''_{n-1,S})^{3+\gamma}\xi_{n}^{\frac{1}{2}}\frac{\tau^S-\tilde{\tau}(\tau'',\tau^S)}{2^{n+1}r''_{n,S}}\\
  \times e^{-C_L\frac{(\tau^S-\tilde{\tau}(\tau'',\tau^S))\left\langle r''_{n,S} \right\rangle^{\gamma^+}}{2^{n}r''_{n,S}}}\ &n \geq 1,
\end{cases}
  \end{equation}
where $a^B_0(\tau^B)$ is defined in \textbf{Proposition \ref{initial point lowerbound general sense fully specular}}.    
\end{definition}

Now, we present a proposition similar to \textbf{Proposition \ref{final lower bound extension}}:
\begin{proposition}\label{final lower bound extension bouncing}
    Let $\Omega \subset \mathbb{R}^3$ satisfying \textbf{Assumption A}, kernel $B$ satisfying \textbf{Assumption B}, and let $f(t,x,v)$ be the continuous mild solution of \eqref{Boltzmann equation}--\eqref{boundary condition} with $\alpha=1$, $\tau''>0$ given from \textbf{Proposition \ref{initial point lowerbound general sense}}. Then, given $\tau^S \in (\tau'',\tau''+\delta_T^B(\tau'')]$, we have

   \begin{equation}\label{4/8 04:33}
       \begin{split}
           &  \ \ \ \ \   f(t,x,v) \geq a^S_{n,\tau''}(\tau^S)\mathbf{1}_{B(0,r''_{n,S})}(v),\\
       &\forall (t,x) \in \left[\tilde{\tau}(\tau'',\tau^S)+(\tau^S-\tilde{\tau}(\tau'',\tau^S))\left(1-\frac{1}{2^{n+1}(r''_{n,S})}\right),\tau^S\right]\times \Omega,
       \end{split}
   \end{equation}
for any $n \in \mathbb{N}$.
\end{proposition}

\begin{proof}
    The proof is the same as the one of \textbf{Proposition \ref{final lower bound extension}}. We prove by induction on $n$. The base case holds by applying \textbf{Proposition \ref{initial point lowerbound general sense fully specular}} with the choice $\tau^B:=\tau^S$: we have
\begin{equation}\label{2025 10/09 02:55}
 f(t,x,v) \geq \min\left\{a^B_{N_B(\tau''),\tau''}(\tau^S),b^B_{\tau''}(\tau^S)\right\}\mathbf{1}_{B(0,1)}(v)
     \end{equation}
for any $$(t,x,v) \in \left[\tilde{\tau}(\tau'',\tau^S),\tau^S\right]\times \Omega\times \mathbb{R}^3.$$  
    
    We now assume that the case $n=k$ holds. We consider
    \begin{equation}\label{Prop 5.4 txv assumption specular}
    \begin{split}
                &(t,x,v) \in \left[\tilde{\tau}(\tau'',\tau^S)+(\tau^S-\tilde{\tau}(\tau'',\tau^S))\left(1-\frac{1}{2^{k+2}(r''_{k+1,S})}\right),\tau^S\right] \\
        &\times \Omega \times B(0,r''_{k+1,S}).
    \end{split}
    \end{equation}
We recall that by the Duhamel formula \eqref{Duhamel formula specular}, and apply \textbf{Lemma \ref{up bound of L lemma}} to deduce that
    \begin{equation}\label{4/8 12:31}
    \begin{split}
        &f(t,x,v)\\
        \geq& \int_{\tilde{\tau}(\tau'',\tau^S)+(\tau^S-\tilde{\tau}(\tau'',\tau^S))(1-\frac{1}{2^{k+1}r''_{k+1,S}})}^{\tilde{\tau}(\tau'',\tau^S)+(\tau^S-\tilde{\tau}(\tau'',\tau^S))(1-\frac{1}{2^{k+2}r''_{k+1,S}})} e^{-C_L(t-s)\left\langle r''_{k+1,S} \right\rangle^{\gamma^+}} \\ 
        &Q^+[f(s,X_{s,t}(x,v),\cdot),f(s,X_{s,t}(x,v),\cdot)](V_{s,t}(x,v))\,ds.
    \end{split}
    \end{equation}

     Since we have
     $$\tilde{\tau}(\tau'',\tau^S)+(\tau^S-\tilde{\tau}(\tau'',\tau^S))\left(1-\frac{1}{2^{k+1}(r''_{k+1,S})}\right) \geq \tilde{\tau}(\tau'',\tau^S)+(\tau^S-\tilde{\tau}(\tau'',\tau^S))\left(1-\frac{1}{2^{k+1}(r''_{k,S})}\right),$$
     we can apply the induction assumption and deduce that
    \begin{equation}\label{Prop5.6 very smal esti}
        f(s,X_{s,t}(x,v),w) \geq a^S_{k,\tau''}(\tau^S)\mathbf{1}_{B(0,r''_{k,S})}(w),
    \end{equation}
for any $w \in \mathbb{R}^3$, $$s\in \left[\tilde{\tau}(\tau'',\tau^S)+(\tau^S-\tilde{\tau}(\tau'',\tau^S))\left(1-\frac{1}{2^{k+1}r''_{k+1,S}}\right),\tilde{\tau}(\tau'',\tau^S)+(\tau^S-\tilde{\tau}(\tau'',\tau^S))\left(1-\frac{1}{2^{k+1}r''_{k+2,S}}\right)\right]$$, and $(t,x,v)$ satisfying \eqref{Prop 5.4 txv assumption specular}.
    Thus, we deduce from \eqref{Prop5.6 very smal esti} that
    \begin{equation}\label{4/8 04:14}
    \begin{split}
        &Q^+[f(s,X_{s,t}(x,v),\cdot),f(s,X_{s,t}(x,v),\cdot)]\\
        \geq&  (a^S_{k,\tau''}(\tau^S))^2Q^+[\mathbf{1}_{B(0,r''_{k,S})},\mathbf{1}_{B(0,r''_{k,S})}].
    \end{split}
    \end{equation}
    
    Then, we use \textbf{Lemma \ref{spreading property lemma}} with $\xi=\xi_{k+1}$ to get
   \begin{equation}\label{4/8 04:15}
   \begin{split}
       &Q^+[\mathbf{1}_{B(0,r''_{k,S})},\mathbf{1}_{B(0,r''_{k,S})}]\\
       \geq & C_Q (r''_{k,S})^{3+\gamma}\xi_{k+1}^{\frac{1}{2}}\mathbf{1}_{B(0,r''_{k,S}\sqrt{2}(1-\xi_{k+1}))}=C_Q (r''_{k,S})^{3+\gamma}\xi_{k+1}^{\frac{1}{2}}\mathbf{1}_{B(0,r''_{k+1,S})}.
       \end{split}
   \end{equation}

    %Notice that $\mathbf{1}_{B(\overline{v}_{i}(\tau''),r''_{k+1})}(v)=1$ by the assumption of $v$. 
    As a result, we deduce from \eqref{4/8 12:31}, \eqref{4/8 04:14}, and \eqref{4/8 04:15} the following estimate: for any $(t,x,v)$ satisfying \eqref{Prop 5.4 txv assumption specular}, we have
      \begin{equation}
      \begin{split}
    &f(t,x,v)\\ 
    \geq& \int_{\tilde{\tau}(\tau'',\tau^S)+(\tau^S-\tilde{\tau}(\tau'',\tau^S))(1-\frac{1}{2^{k+1}r''_{k+1,S}})}^{\tilde{\tau}(\tau'',\tau^S)+(\tau^S-\tilde{\tau}(\tau'',\tau^S))(1-\frac{1}{2^{k+2}r''_{k+1,S}})} (a^S_{k,\tau''}(\tau^S))^2\\
    &\times  e^{-C_L(t-s)\langle r''_{k+1,S} \rangle^{\gamma^+}} C_Q (r''_{k,S})^{3+\gamma}\xi_{k+1}^{\frac{1}{2}}\mathbf{1}_{B(0,r''_{k+1,S})}(V_{s,t}(x,v))\,ds\\
    \geq& \int_{\tilde{\tau}(\tau'',\tau^S)+(\tau^S-\tilde{\tau}(\tau'',\tau^S))(1-\frac{1}{2^{k+1}r''_{k+1,S}})}^{\tilde{\tau}(\tau'',\tau^S)+(\tau^S-\tilde{\tau}(\tau'',\tau^S))(1-\frac{1}{2^{k+2}r''_{k+1,S}})} (a^S_{k,\tau''}(\tau^S))^2\\
    &\times  e^{-C_L(t-s)\langle r''_{k+1,S} \rangle^{\gamma^+}} C_Q (r''_{k,S})^{3+\gamma}\xi_{k+1}^{\frac{1}{2}}\,ds\\
  \geq& (a^S_{k,\tau''}(\tau^S))^2 C_Q (r''_{k,S})^{3+\gamma}\xi_{k+1}^{\frac{1}{2}}\frac{\tau^S-\tilde{\tau}(\tau'',\tau^S)}{2^{k+2}r''_{k+1,S}}\\
  &\times e^{-C_L\frac{(\tau^S-\tilde{\tau}(\tau'',\tau^S))\left\langle r''_{k+1,S} \right\rangle^{\gamma^+}}{2^{k+1}r''_{k+1,S}}}\\
  =&a^S_{k+1,\tau''}(\tau^S).
  \end{split}
   \end{equation}
    Hence, we showed that \eqref{4/8 04:33} holds for $n=k+1$ and conclude the proof by induction.
    
\end{proof}
Thanks to \textbf{Proposition \ref{final lower bound extension bouncing}}, we can use the exact same method used at the end of the fourth chapter to derive the Maxwellian lower bound for the fully specular reflection boundary condition.

\section{Lower bound for the non-cutoff case}

In this chapter, we investigate lower bounds for the solution of the Boltzmann equation in the non-cutoff case, that is, $\nu\geq 0$. Here, the domain $\Omega$ is not necessarily convex.
In contrast to the cutoff case studied in Chapter 2--4, the non-cutoff setting requires a refined decomposition of the collision operator to handle the angular singularity:
\begin{equation}
    \begin{split}
        Q[h_1,h_2](v)&=\int_{\mathbb{R}^3\times \mathbb{S}^2}B(|v-v_*|,\cos{\theta})[h_2(v')h_1(v'_*)-h_2(v)h_1(v_*)]\,dv_*\,d\sigma\\
        &=\int_{\mathbb{R}^3\times \mathbb{S}^2}B(|v-v_*|,\cos{\theta})[h_1(v'_*)(h_2(v')-h_2(v))]\,dv_*\,d\sigma\\
        &\hspace{1.5cm}-h_2(v)\int_{\mathbb{R}^3\times \mathbb{S}^2}B(|v-v_*|,\cos{\theta})(h_1(v_*)-h_1(v'_*))\,dv_*\,d\sigma\\
        &=:Q^1_b[h_1,h_2](v)-Q^2_b[h_1,h_2](v),
    \end{split}
\end{equation}
where

\begin{align*}
    &Q^1_b[h_1,h_2](v):=\int_{\mathbb{R}^3\times \mathbb{S}^2}B(|v-v_*|,\cos{\theta})[h_1(v'_*)(h_2(v')-h_2(v))]\,dv_*\,d\sigma,\\
    &Q^2_b[h_1,h_2](v):=h_2(v)\int_{\mathbb{R}^3\times \mathbb{S}^2}B(|v-v_*|,\cos{\theta})[h_1(v_*)-h_1(v'_*)]\,dv_*\,d\sigma.
\end{align*}

With this notation and (for some $\epsilon>0$) the decomposition of  $B(|v-v_*|,\cos{\theta})$
\begin{align}
    B(|v-v_*|,\cos{\theta})=\Phi(|v-v_*|)b(\cos{\theta})\mathbf{1}_{\theta \geq \epsilon}+\Phi(|v-v_*|)b(\cos{\theta})\mathbf{1}_{\theta < \epsilon},
\end{align}
we can consider the singular and non-singular parts of $Q$:

\begin{align*}
    Q[h_1,h_2](v)=Q^+_\epsilon[h_1,h_2](v)-Q^-_\epsilon[h_1,h_2](v)+Q^1_\epsilon[h_1,h_2](v)+Q^2_\epsilon[h_1,h_2](v),
\end{align*}
where

\begin{align*}
&Q^+_{\epsilon}[h_1,h_2](v):=\int_{\mathbb{R}^3\times \mathbb{S}^2}\Phi(|v-v_*|)b(\cos{\theta})\mathbf{1}_{\theta \geq \epsilon}[h_2(v')h_1(v'_*)]\,dv_*\,d\sigma,\\
&Q^-_\epsilon [h_1,h_2](v):=h_2(v)\int_{\mathbb{R}^3\times \mathbb{S}^2}\Phi(|v-v_*|)b(\cos{\theta})\mathbf{1}_{\theta \geq \epsilon}h_1(v_*)\,dv_*\,d\sigma,\\
&Q^1_{\epsilon}[h_1,h_2](v):=\int_{\mathbb{R}^3\times \mathbb{S}^2}\Phi(|v-v_*|)b(\cos{\theta})\mathbf{1}_{\theta < \epsilon}\left[h_1(v'_*)(h_2(v')-h_2(v))\right]\,dv_*\,d\sigma,\\
&Q^2_{\epsilon}[h_1,h_2](v):=h_2(v)\int_{\mathbb{R}^3\times \mathbb{S}^2}\Phi(|v-v_*|)b(\cos{\theta})\mathbf{1}_{\theta < \epsilon}[h_1(v_*)-h_1(v'_*)]\,dv_*\,d\sigma.
\end{align*}

We further introduce the following notations
\begin{align}
     &b^{CO}_{\epsilon}(\cos{\theta}):=\mathbf{1}_{\theta\geq\epsilon}b(\cos{\theta}), \,
     b^{NCO}_{\epsilon}(\cos{\theta}):=\mathbf{1}_{\theta<\epsilon}b(\cos{\theta}),\\
      &\label{2025 10/26 14:02}m_{b}:=\int_{\mathbb{S}^2}b(\cos{\theta})(1-\cos{\theta})\,d\sigma,\\
     &\label{2025 10/26 14:03}n_{b^{CO}_{\epsilon}}:=\int_{\mathbb{S}^2}b^{CO}_{\epsilon}(\cos{\theta})\,d\sigma, \,
      m_{b^{NCO}_{\epsilon}}:=\int_{\mathbb{S}^2}b^{NCO}_{\epsilon}(\cos{\theta})(1-\cos{\theta})\,d\sigma,
\end{align}

\begin{align*}
   & L_{\epsilon}(h)(v):=\int_{\mathbb{R}^3\times \mathbb{S}^2}\Phi(|v-v_*|)b(\cos{\theta})\mathbf{1}_{\theta \geq \epsilon}h(v_*)\,dv_*\,d\sigma,\\
   & S[h](v):=\int_{\mathbb{R}^3\times \mathbb{S}^2}\Phi(|v-v_*|)b(\cos{\theta})[h(v_*)-h(v'_*)]\,dv_*\,d\sigma,\\
   & S_{\epsilon}[h](v):=\int_{\mathbb{R}^3\times \mathbb{S}^2}\Phi(|v-v_*|)b(\cos{\theta})\mathbf{1}_{\theta < \epsilon}[h(v_*)-h(v'_*)]\,dv_*\,d\sigma.
\end{align*}

With this decomposition of the collision operator, we now introduce the definition of mild solutions in the non-cutoff setting:
\begin{definition}\label{non-cutoff defi name}
Suppose that $\Omega \subset \mathbb{R}^3$ satisfies \textbf{Assumption A} and the kernel $B$ satisfies \textbf{Assumption B} with $\nu \geq 0$. Given a non-negative continuous function $f_0$ on $\overline{\Omega} \times \mathbb{R}^3$, we call a non-negative continuous function $f$ defined on $[0,T) \times (\overline{\Omega} \times \mathbb{R}^3)$ with $|f(t,x,v)| \leq C(1+|v|)^{-r}$ for some constant $C>0$ and $r>3$ for any $0< t\leq T$, $(x,v) \in \overline{\Omega} \times \mathbb{R}^3$ a "continuous mild" solution to \eqref{Boltzmann equation}--\eqref{boundary condition} with initial data $f_0$ if $f$ is continuous on $\Gamma_{conti}$ and there exists a number $0<\epsilon_0<\frac{\pi}{4}$ such that for any $0<\epsilon<\epsilon_0$, $(t,x,v) \in [0,T) \times \Omega \times \mathbb{R}^3$, we have
\begin{equation}\label{Duhamel formula not cutoff}
\begin{split}
f(t,x,v)&= f_{0}(X_{0,t}(x,v),v)\exp\bigg[-\int_0^t(L_{\epsilon}+S_{\epsilon})[f(s,X_{s,t}(x,v),\cdot)](v)\,ds\bigg]\\
&+\int_0^t \exp\left(-\int_s^t (L_{\epsilon}+S_{\epsilon})[f(s',X_{s',t}(x,v),\cdot)](v)\,ds'\right)\\
&(Q^+_{\epsilon}+Q_{\epsilon}^1)[f(s,X_{s,t}(x,v),\cdot),f(s,X_{s,t}(x,v),\cdot)](v)\,ds,
\end{split}
\end{equation}
when $t \leq t_{\partial}(x,v)$, and
\begin{equation}\label{Duhamel formula B not cutoff}
\begin{split}
f(t,x,v)&= \alpha f(t-t_{\partial}(x,v),X_{t-t_{\partial}(x,v),t}(x,v),R(X_{t-t_{\partial}(x,v),t}(x,v),v))\\
&\exp\left[-\int_{t-t_{\partial}(x,v)}^t(L_{\epsilon}+S_{\epsilon})[f(s,X_{s,t}(x,v),\cdot)](v)\,ds\right]\\
&+(1-\alpha)\left( \int_{w\cdot n(X_{t-t_{\partial}(x,v),t}(x,v))>0} f(t,X_{t-t_{\partial}(x,v),t}(x,v),w)(w\cdot n(X_{t-t_{\partial}(x,v),t}(x,v))) dw \right)\\
&\frac{1}{2\pi T_B^2}e^{-\frac{|v|^2}{2T_B}} \exp\left[-\int_{t-t_{\partial}(x,v)}^t(L_{\epsilon}+S_{\epsilon})[f(s,X_{s,t}(x,v),\cdot)](v)\,ds\right]\\
&+\int_{t-t_{\partial}(x,v)}^t \exp\left(-\int_s^t (L_{\epsilon}+S_{\epsilon})[f(s',X_{s',t}(x,v),\cdot)](v)\,ds'\right)\\
&(Q^+_{\epsilon}+Q_{\epsilon}^2)[f(s,X_{s,t}(x,v),\cdot),f(s,X_{s,t}(x,v),\cdot)](v)\,ds
\end{split}
\end{equation}
when $t \geq t_{\partial}(x,v)$.
Here, we recall the definition of $t_{\partial}(x,v)$ from \textbf{Definition \ref{3/9 01:05}}.

\end{definition}

Before introducing the main result in the non-cutoff case, we introduce the following constants (here $\tilde{\gamma}:=(2+\gamma)^+$): 

\begin{equation} \label{Energy constant 5}
    e'_{f}(t,x):= \int_{v \in \mathbb{R}^3}|v|^{\tilde{\gamma}}f(t,x,v)\,dv,
\end{equation}

\begin{equation} \label{Energy constant 6}
    E'_{f}:= \sup\limits_{[0,T)\times \Omega}e'_{f}(t,x),
\end{equation}

\begin{equation} \label{Energy constant 7}
    w_{f}(t,x):= \|f(t,x,\cdot)\|_{W^{2,\infty}_v},
\end{equation}

\begin{equation} \label{Energy constant 8}
    W_{f}:= \sup\limits_{[0,T)\times \Omega}w_f(t,x).
\end{equation}

Now, we introduce our main theorem for the non-cutoff case:
\begin{theorem} \label{Main theorem 3}
     Suppose that $\Omega \subset \mathbb{R}^3$ satisfies \textbf{Assumption A} and that the kernel $B$ satisfies \textbf{Assumption B} with $\nu \geq0$, $\alpha \in [0,1]$. We consider a non-negative function $f_0$ that is continuous on $ \overline{\Omega} \times \mathbb{R}^3$. Let $f(t,x,v)$ be a continuous mild solution to \eqref{Boltzmann equation}--\eqref{boundary condition} on $[0,T) \times \overline{\Omega} \times \mathbb{R}^3$, with $T>0$ (and the initial condition $f_0$) which satisfies the following properties:
  \begin{enumerate}
        \item $M>0$;
        %\item $f$ is a continuous function on $[0,T]\times ( \overline{\Omega} \times \mathbb{R}^3 - \Gamma^0)$;
        \item $E_f< \infty$, where $p_\gamma > \frac{3}{3+\gamma} >0$, if $-3 < \gamma <0$;
         \item $W_f<\infty$ and $E'_f< \infty$.
    \end{enumerate}

 Then, the following lower bound holds:
    There exists $0 <\tau_0\leq T$ such that for any $\tau \in (0,\tau_0)$ and $K>2\frac{\log{(2+\frac{2\nu}{2-\nu})}}{\log{2}}$, there exist $\Delta_{\tau_0}>0$, $\rho>0$ and $\theta>0$ depending on $\tau_0$, $C_{\Phi},c_{\Phi}, \gamma, b_0, \nu, E_f$ ,$W_f$,$E'_f$, $M$ (and $L_{f,p}$ if $\gamma <0$ ), $\tau$, $K$, and the modulus of continuity of $f_0$, such that

    \begin{equation}
        f(t,x,v) \geq \frac{\rho}{(2\pi\theta)^{\frac{3}{2}}}e^{-\frac{|v|^K}{2\theta}},\ \forall \, t \in [\tau,\Delta_{\tau_0}),\ \forall \, (x,v) \in \overline{\Omega}\times \mathbb{R}^3.
    \end{equation}
    In the case when $\nu=0$, we can further take $K=2$.
    \end{theorem}

    To prove this theorem, we quote the following result extracted from \textbf{Corollary 2.2} in \cite{Mou 1}:
\begin{lemma}\label{up bound of S lemma}
    Consider $g$ a measurable function on $\mathbb{R}^3$, and assume that the collision operator satisfies \textbf{Assumption B} with $0\leq\nu<2$. Then, there exists $C_g^S>0$ which depends only on $m_b$, $C_{\Phi}$ and $e_g$ (and $l_{g,p}$, where $p>\frac{3}{3+\gamma}$, if $\gamma<0$) such that
\begin{equation}
    |S[g](v)| \leq C_g^S\left\langle v \right\rangle^{\gamma^+}.
\end{equation}
Here we recall the definition of $e_g$ in \eqref{Energy constant 1}, $l_{g,p}$ in \eqref{Energy constant 3} and define $m_b:=\int_{\mathbb{S}^2}b(\cos{\theta})(1-\cos{\theta})\,d\sigma$.
\end{lemma}

Similarly to the cut-off case, we would like to use \eqref{Duhamel formula not cutoff} and \eqref{Duhamel formula B not cutoff} to spread the initial lower bounds. %Thanks to the \textbf{Lemma \ref{up bound of L lemma}} and \textbf{Lemma \ref{up bound of S lemma}}, we can easily replicate the method. 
However, the additional non-cutoff term $Q^1_{\epsilon}$ is not necessarily nonnegative. Hence, we need to quote \textbf{Lemma 2.5} from \cite{Mou 1}:
\begin{lemma}\label{raw deal with R one lemma}
Suppose that the collision operator satisfies \textbf{Assumption B} with $0\leq\nu<2$. Then, there exists a constant $C>0$, which depends only on $\gamma,\ \nu,\ b_0$, such that for any measurable function $h_1,h_2$, we have 

\begin{equation}\label{Q1 upper bound}
    |Q^1_{b}\left(h_1,h_2\right)| \leq C m_b c_{\Phi} \|h_2\|_{L^1_{\tilde{\gamma}}}\|h_1\|_{W^{2,\infty}}.
\end{equation}
\end{lemma}

Now, we show that a result which is similar to the result of \textbf{Proposition \ref{ini to multi}} can be obtained:
\begin{proposition} \label{ini to multi non-cutoff}
    Suppose that $\Omega \subset \mathbb{R}^3$ satisfies \textbf{Assumption A} and the kernel $B$ satisfies \textbf{Assumption B} with $ 0\leq \nu <2$, $\alpha \in [0,1]$. We consider a mild continuous solution $f(t,x,v)$ of \eqref{Boltzmann equation}--\eqref{boundary condition}. Given $A > 0$, $0<\Delta_1\leq 1,$ $\Delta_2<1$, and $(\tau, x' ,v') \in [0, T) \times \Omega \times \mathbb{R}^3$ such that $B(x',\Delta_2) \subset \Omega$ and

    \begin{equation*}
        f(t,x,v) \geq A,\ \forall (t,x,v) \in [\tau, \tau+\Delta_1] \times B(x',\Delta_2) \times B(v', \Delta_2).
    \end{equation*}
    Then, there exist $\{\epsilon_i>0\}_{i=0}^{\infty}$ such that for $n \in \mathbb{N}\cup \{ 0 \}$, $t \in [\tau, \tau+\Delta_1]$, $x \in B(x',\frac{\Delta_2}{2^n}) $,

\begin{equation}\label{2025 04/09 05:46}
    \forall v \in \mathbb{R}^3,\ f(t,x,v) \geq \alpha^{NC}_n(\tau,t,\Delta_2,A,|v'|)\mathbf{1}_{ B(v',r_n(\Delta_2))}(v),
\end{equation}
where the numbers $\{r_n(\Delta_2)\}_{n=0}^{\infty} \in \mathbb{R}$, $\{t_n(\tau,t,\Delta_2,|v'|)\}_{n=0}^{\infty}\in \mathbb{R}$ and $\{\alpha^{NC}_n(\tau, t,\Delta_2,A,|v'|)\}_{n=0}^{\infty}\in \mathbb{R}$ are defined as follows:

\begin{equation}
r_0(\Delta_2):=\Delta_2,\ r_{n+1}(\Delta_2):=\frac{3\sqrt{2}}{4}r_{n}(\Delta_2),    
\end{equation}

\begin{equation}
t^{NC}_n(\tau,t,\Delta_2,|v'|):=\max\left\{ \tau, t-\frac{\Delta_2}{2^{n+1}(2r_n(\Delta_2)+|v'|)} \right\},
\end{equation}
\begin{equation}
    \alpha_0:=A
\end{equation}
\begin{equation}\label{alpha second line non-cutoff}
\begin{split}
&\alpha^{NC}_{n+1}(\tau,t,\Delta_2,A,|v'|)\\\
:=&\frac{1}{4}\int_{t^{NC}_n(\tau,t,\Delta_2,|v'|)}^t \exp\left(-(t-s)(C_{L,1}n_{b^{CO}_{\epsilon_n}}+C_{L,2}m_{b^{NCO}_{\epsilon_n}})\left\langle |v'|+2r_n(\Delta_2) \right\rangle^{\gamma^+}\right)\\
        &\hspace{0.5cm} C_{Q,1} \alpha^{NC}_n(\tau, s,\Delta_2,A,|v'|)^2 l_{b^{CO}_{\epsilon_n}} c_{\Phi}(r_n(\Delta_2))^{3+\gamma}\,ds.
\end{split}
\end{equation}
where $C_{L,1}$, $C_{L,2}$ and $C_{Q,1}$ are constants which depend on $\gamma,\ \nu,\ b_0,\ n_b$, $m_b$, $C_{\Phi}$, $E_f$, $W_f$, $E'_f$, $M$(and $L_{f,p}$, where $p>\frac{3}{3+\gamma}$, if $\gamma<0$).

\end{proposition}

\begin{proof}
 %The following proof is also a summary of proof of \textbf{Lemma 3.3} in \cite{Bri1}. 
We prove the proposition by induction on $n$. The case $n=0$ is exactly the assumption.  Assume that the statement of \textbf{Proposition \ref{ini to multi non-cutoff}} holds for $n=k$. Given $t \in [\tau,\tau+\Delta_1]$, $x \in B\left(x', \frac{\Delta_2}{2^{k+1}}\right)$, $v \in B\left(v',r_{k+1}(\Delta_2)\right) \subset B\left(0,|v'|+2r_{k}(\Delta_2)\right)$, we consider the second term on the right-hand side of \eqref{Duhamel formula not cutoff} to obtain the following lower bound which holds for $0<\epsilon<\epsilon_0$ ($\epsilon_0$ is mentioned in \textbf{Definition \ref{non-cutoff defi name}}):

\begin{equation}
    \begin{split}
        &f(t,x,v) \\
        =&f(t^{NC}_k(\tau,t,\Delta_2,|v'|),X_{t^{NC}_k(\tau,t,\Delta_2,|v'|),t}(x,v),v)\\
        &\exp\bigg[-\int_{t^{NC}_k(\tau,t,\Delta_2,|v'|)}^t(L_{\epsilon}+S_{\epsilon})[f(s,X_{s,t}(x,v),\cdot)](v)\,ds\bigg]\\
&+\int_{t^{NC}_k(\tau,t,\Delta_2,|v'|)}^t \exp\left(-\int_s^t (L_{\epsilon}+S_{\epsilon})[f(s',X_{s',t}(x,v),\cdot)](v)\,ds'\right)\\
&(Q^+_{\epsilon}+Q_{\epsilon}^1)[f(s,X_{s,t}(x,v),\cdot),f(s,X_{s,t}(x,v),\cdot)](v)\,ds,\\
         \geq & \int_{t^{NC}_k(\tau,t,\Delta_2,|v'|)}^t \exp\left(-\int_s^t (L_{\epsilon}+S_{\epsilon})[f(s',X_{s',t}(x,v),\cdot)](v))\,ds'\right)\\
        &(Q^+_{\epsilon}-|Q_{\epsilon}^1|)[f(s,X_{s,t}(x,v),\cdot),f(s,X_{s,t}(x,v),\cdot)](v)\,ds.
    \end{split}
\end{equation}

Now, by \textbf{Lemma \ref{up bound of L lemma}} and \textbf{Lemma \ref{up bound of S lemma}}, we have

\begin{equation}
|(L_{\epsilon}+S_{\epsilon})[f(s',X_{s',t}(x,v),\cdot)](v)| \leq (C_{L,1}n_{b^{CO}_{\epsilon}}+C_{L,2}m_{b^{NCO}_{\epsilon}})\left\langle |v'|+2r_k(\Delta_2) \right\rangle^{\gamma^+}
\end{equation}
for some constants $C_{L,1},C_{L,2}$ depending only on $n_b$, $m_b$, $C_{\Phi}$ and $E_f$  (and $L_{f,p}$ if $\gamma<0$).

Hence, the following estimate holds:
\begin{equation}
    \begin{split}
        &f(t,x,v) \\
         \geq & \int_{t^{NC}_k(\tau,t,\Delta_2,|v'|)}^t \exp\left(-(t-s)(C_{L,1}n_{b^{CO}_{\epsilon}}+C_{L,2}m_{b^{NCO}_{\epsilon}})\left\langle |v'|+2r_k(\Delta_2) \right\rangle^{\gamma^+}\right)\\
        &(Q^+_{\epsilon}-|Q_{\epsilon}^1|)[f(s,X_{s,t}(x,v),\cdot),f(s,X_{s,t}(x,v),\cdot)](v)\,ds,
    \end{split}
\end{equation}
for all $\frac{\pi}{4}>\epsilon>0$.

We notice that when $s \in \left[t^{NC}_k(\tau,t,\Delta_2,|v'|),t\right]$, we have
\begin{equation*}
    |x'-X_{s,t}(x,v)|=|x'-x+tv-sv|\leq \frac{\Delta_2}{2^{k+1}}+|t-s|(|v'|+2r_k(\Delta_2)) \leq  \frac{\Delta_2}{2^{k}},
\end{equation*}
which implies that $X_{s,t}(x,v) \in B\left(x', \frac{\Delta_2}{2^k}\right) \subset \Omega$ and we deduce by induction hypothesis that
\begin{equation}\label{2025 04/09 05:51}
    \forall v \in \mathbb{R}^3,\ f(s,X_{s,t}(x,v),v) \geq \alpha^{NC}_k(\tau,t,\Delta_2,A,|v'|)\mathbf{1}_{ B(v',r_k(\Delta_2))}(v).
\end{equation}

Next, we obtain from \textbf{Lemma \ref{spreading property lemma}} by setting $\xi=\frac{1}{4}$: for any $v \in B\left(v',r_{k+1}(\Delta_2)\right) $, we have
\begin{equation}
    \begin{split}
        &Q_{\epsilon}^+[f(s,X_{s,t}(x,v),\cdot),f(s,X_{s,t}(x,v),\cdot)](v)\\
        \geq & C_{Q,1} \alpha^{NC}_k(\tau, s,\Delta_2,A,|v'|)^2l_{b^{CO}_{\epsilon}} c_{\Phi}(r_k(\Delta_2))^{3+\gamma}\frac{1}{2}\mathbf{1}_{B\left(v',\frac{3\sqrt{2}}{4}r_k(\Delta_2)\right)}\\
        =&C_{Q,1} \alpha^{NC}_k(\tau, s,\Delta_2,A,|v'|)^2l_{b^{CO}_{\epsilon}} c_{\Phi}(r_k(\Delta_2))^{3+\gamma}\frac{1}{2},
    \end{split}
\end{equation}for some $ C_{Q,1}>0$ which only depends on $\gamma,\ \nu,\ b_0$.

Now, we use \eqref{Q1 upper bound} to derive:
\begin{equation}
    \begin{split}
        |Q_{\epsilon}^1[f(s,X_{s,t}(x,v),\cdot),f(s,X_{s,t}(x,v),\cdot)](v)|\leq C_{Q,2} m_{b^{NCO}_{\epsilon}} c_{\Phi}E'_{f}W_{f},
    \end{split}
\end{equation}
for some $ C_{Q,2}>0$, which depends only on $\gamma,\ \nu,\ b_0$.

As a result, we deduce that
\begin{equation}
    \begin{split}
        &f(t,x,v) \\
         \geq& \int_{t^{NC}_k(\tau,t,\Delta_2,|v'|)}^t \exp\left(-(t-s)(C_{L,1}n_{b^{CO}_{\epsilon}}+C_{L,2}m_{b^{NCO}_{\epsilon}})\left\langle |v'|+2r_k(\Delta_2) \right\rangle^{\gamma^+}\right)\\
        &\hspace{0.5cm} \left(C_{Q,1} \alpha^{NC}_k(\tau, s,\Delta_2,A,|v'|)^2l_{b^{CO}_{\epsilon}} c_{\Phi}(r_k(\Delta_2))^{3+\gamma}\frac{1}{2}-C_{Q,2} m_{b^{NCO}_{\epsilon}} c_{\Phi}E'_{f}W_{f}\right)\,ds\\
        \geq & \int_{t^{NC}_k(\tau,t,\Delta_2,|v'|)}^t \exp\left(-(t-s)(C_{L,1}n_{b^{CO}_{\epsilon}}+C_{L,2}m_{b^{NCO}_{\epsilon}})\left\langle |v'|+2r_k(\Delta_2) \right\rangle^{\gamma^+}\right)\\
        &\hspace{0.5cm} C_{Q,1} \alpha^{NC}_k(\tau, s,\Delta_2,A,|v'|)^2 l_{b^{CO}_{\epsilon}} c_{\Phi}(r_k(\Delta_2))^{3+\gamma}\frac{1}{2}\,ds\\
        -&\int_{t^{NC}_k(\tau,t,\Delta_2,|v'|)}^t \exp\left(-(t-s)(C_{L,1}n_{b^{CO}_{\epsilon}}+C_{L,2}m_{b^{NCO}_{\epsilon}})\left\langle |v'|+2r_k(\Delta_2) \right\rangle^{\gamma^+}\right)\\
        &\hspace{0.5cm} C_{Q,2} m_{b^{NCO}_{\epsilon}} c_{\Phi}E'_{f}W_{f}\,ds,
    \end{split}
\end{equation}
for all $\frac{\pi}{4}>\epsilon>0$.

Then, we notice that when $\frac{\pi}{4}>\epsilon>0$, we ahve
\begin{equation}
    l_{b^{CO}_{\epsilon}}:=\inf\limits_{\frac{1}{4}\pi\leq \theta \leq \frac{3}{4}\pi} b^{CO}_{\epsilon}(\cos{\theta})\geq l_b,
\end{equation}
Thus, we deduce from \eqref{2025 10/26 14:02} and \eqref{2025 10/26 14:03} that
\begin{equation}
\begin{split}  n_{b^{CO}_{\epsilon}}=\int_{\mathbb{S}^2}b^{CO}_{\epsilon}(\cos{\theta})\,d\sigma\sim\begin{cases}
         \frac{b_0}{\nu}\epsilon^{-\nu},\, \nu \in (0,2),\\
         b_0|\log{\epsilon}|,\, \nu=0,
    \end{cases}
\end{split}
\end{equation}

\begin{equation}
\begin{split}
    m_{b^{NCO}_{\epsilon}}=\int_{\mathbb{S}^2}b^{NCO}_{\epsilon}(\cos{\theta})(1-\cos{\theta})\,d\sigma    \sim\begin{cases}
         \frac{b_0}{2-\nu}\epsilon^{2-\nu},\, \nu \in (0,2),\\
         \frac{b_0}{2}\epsilon^2,\, \nu=0.
    \end{cases}
\end{split}
\end{equation}

Hence, for any $k \in \mathbb{N}$, we can choose $\frac{\pi}{4} >\epsilon_k>0$ such that 

\begin{equation}
    \begin{split}
        &\int_{t^{NC}_k(\tau,t,\Delta_2,|v'|)}^t \exp\left(-(t-s)(C_{L,1}n_{b^{CO}_{\epsilon_k}}+C_{L,2}m_{b^{NCO}_{\epsilon_k}})\left\langle |v'|+2r_k(\Delta_2) \right\rangle^{\gamma^+}\right)\\
        &\hspace{0.5cm} C_{Q,2} m_{b^{NCO}_{\epsilon_k}} c_{\Phi}E'_{f}W_{f}\,ds\\
        \leq &  \frac{1}{4}\int_{t^{NC}_k(\tau,t,\Delta_2,|v'|)}^t \exp\left(-(t-s)(C_{L,1}n_{b^{CO}_{\epsilon_k}}+C_{L,2}m_{b^{NCO}_{\epsilon_k}})\left\langle |v'|+2r_k(\Delta_2) \right\rangle^{\gamma^+}\right)\\
        &\hspace{0.5cm} C_{Q,1} \alpha^{NC}_k(\tau, s,\Delta_2,A,|v'|)^2 l_{b^{CO}_{\epsilon_k}} c_{\Phi}(r_k(\Delta_2))^{3+\gamma}\,ds
    \end{split}
\end{equation}

and therefore
\begin{equation}
  \begin{split}
        &f(t,x,v) \\
         \geq& \frac{1}{4}\int_{t^{NC}_k(\tau,t,\Delta_2,|v'|)}^t \exp\left(-(t-s)(C_{L,1}n_{b^{CO}_{\epsilon_k}}+C_{L,2}m_{b^{NCO}_{\epsilon_k}})\left\langle |v'|+2r_k(\Delta_2) \right\rangle^{\gamma^+}\right)\\
        &\hspace{0.5cm} C_{Q,1} \alpha^{NC}_k(\tau, s,\Delta_2,A,|v'|)^2 l_{b^{CO}_{\epsilon_k}} c_{\Phi}(r_k(\Delta_2))^{3+\gamma}\,ds\\
        =&\alpha^{NC}_{k+1}(\tau,t,\Delta_2,A,|v'|),
    \end{split}
\end{equation}
for $v \in B(v',r_{k+1}(\Delta_2))$ and we conclude the proof by induction.

\end{proof}

With the help of \textbf{Proposition \ref{ini to multi non-cutoff}}, we can use the exact same argument from \textbf{Proposition \ref{translation proposition}} to  \textbf{Proposition \ref{initial point lower bound near boundary}} to generate series of lower bounds on $\mathbb{Y}$. Then, we can work as in \textbf{Proposition \ref{positve integral estimate}} to deduce the following.  
\begin{proposition} \label{positve integral estimate non-cutoff}
    Suppose that $\Omega \subset \mathbb{R}^3$ satisfies \textbf{Assumption A} and that the kernel $B$ satisfies \textbf{Assumption B} with $\nu \geq 0$. Let $f(t,x,v)$ be a continuous mild solution of \eqref{Boltzmann equation}--\eqref{boundary condition} with $\alpha \in [0,1]$. Fix $\tau \in (0,\Delta^0)$, with $\Delta^0>0$ is given in \textbf{Proposition \ref{initial point lower bound on a single point}}. Then, there exist $b^{NC}(\tau)$, $\delta_T(\tau)>0$, which depend on $\gamma,\ \nu,\ b_0,\ n_b$, $m_b$ $C_{\Phi}$ and $E_f$, $W_f$, $E'_f$, and $M$(and $L_{f,p}$, where $p>\frac{3}{3+\gamma}$, if $\gamma<0$) such that for $t \in [\tau, \tau+\delta_T(\tau)]$, 

    \begin{equation}
        \forall x \in \partial \Omega, \ \int_{v_*\cdot n(x)>0}f(t,x,v_*)(v_*\cdot n(x))\,dv_*>b^{NC}(\tau).
    \end{equation}
\end{proposition}

Using the same method as in \cite{Bri1,Bri2} (notice that the proof does not require the convexity of $\Omega$), we derive the following. 
\begin{lemma}
 Suppose that $\Omega \subset \mathbb{R}^3$ satisfies \textbf{Assumption A} and the kernel $B$ satisfies \textbf{Assumption B} with $\nu \geq 0$. Let $f(t,x,v)$ be a continuous mild solution of \eqref{Boltzmann equation}--\eqref{boundary condition} with $\alpha \in [0,1]$. Fix $\tau \in (0,\Delta^0)$, where $\Delta^0>0$ is given in \textbf{Proposition \ref{initial point lower bound on a single point}} and $\{\xi_n\}_{n=1}^{\infty} \in (0,1)$. Then, there exist $r^{NC}_0>0$, $\{\Delta_i\}_{i=1}^{\infty}$ such that $\sum_{i=1}^{\infty}\Delta_i=1$ and that for any $n \in \mathbb{N} $, $ t \in [(\sum_{i=1}^n\Delta_{i})\tau,\tau]$, $\forall (x,v) \in \overline{\Omega} \times \mathbb{R}^3$, we have
\begin{equation}
    f(t,x,v) \geq a_n^{NC}(\tau)\mathbf{1}_{B(0,r''_n)}(v)
\end{equation}

Here, 
\begin{equation}
    r^{NC}_{n+1}:=\sqrt{2}\,(1-\xi_n)\,r^{NC}_n,
\end{equation}
\begin{equation}
    a_{n+1}^{NC}(\tau):=
    \begin{cases}
    C_Q\Delta_{n+1}e^{-\left[C_L\left(a^{NC}_n(\tau)\right)^2\left(r_n^{NC}\right)^{3+\gamma-\tilde{\gamma}}\xi_n^{\frac{1}{2}}\right]^{\frac{-\nu}{2-\nu}}\left(\sum\limits_{k \geq n+1}\Delta_k\right)\left(r_{n+1}^{NC}\right)^{\gamma^+}}\\
    \times\left(a^{NC}_n(\tau)\right)^2 \left(r_n^{NC}\right)^{\gamma+3}\xi_n^{\frac{5}{2}}, &\nu \in (0,2),\\
    C_Q\Delta_{n+1}e^{-C\log \left[C_L \left(a^{NC}_n(\tau)\right)^2\left(r_n^{NC}\right)^{3+\gamma-\tilde{\gamma}}\xi_n^{\frac{1}{2}}\right]\left(\sum\limits_{k \geq n+1}\Delta_k\right)\left(r_{n+1}^{NC}\right)^{\gamma^+}}\\
    \times\left(a^{NC}_n(\tau)\right)^2\left(r_n^{NC}\right)^{\gamma+3}\xi_n^{\frac{5}{2}}, &\nu =0.
     \end{cases}
\end{equation}

\end{lemma}
Finally, by proceeding as in \cite{Mou 1}, p.29--31, we can derive the "weaker than Maxwellian" lower bound of \textbf{Theorem \ref{Main theorem 3}}.

\setcounter{section}{0}
\renewcommand{\thesection}{A}
\setcounter{theorem}{0}
\renewcommand{\thetheorem}{A.\arabic{theorem}}

\section{Appendix: the characteristic line}

In this section, we recall some useful properties about the bounce trajectory of $\Omega$ used in the definition of the characteristic line $X_{s,t}(x,v)$ and $V_{s,t}(x,v)$. 
To begin with, we introduce some notations from Definition A.1 in \cite{Bri2}.

\begin{definition}
    Given an open bounded domain with a $C^1$ boundary, we define a partition of $\partial \Omega \times \mathbb{R}^3$ as follows:

    \begin{itemize}
  \item 
     \begin{equation*}
     \Omega_{rebounds}:=\{ (x,v) \in \partial \Omega \times \mathbb{R}^3 \mid v \cdot n(x)<0 \}. 
     \end{equation*}
  \item 
     \begin{equation*}
     \Omega_{rolling}:=\{ (x,v) \in \partial \Omega \times \mathbb{R}^3 \mid v \cdot n(x)=0,\ \exists \, \delta>0 \ s.t. \ x-vt \in \overline{\Omega},\ \forall \, t \in [0,\delta] \}. 
     \end{equation*}
  \item 
     \begin{equation*}
     \Omega_{stop}:=\{ (x,v) \in \partial \Omega \times \mathbb{R}^3 \mid v \cdot n(x)=0,\ \forall \, \delta>0 \ \exists \, t \in (0,\delta)\ s.t. \ x-vt \notin \overline{\Omega} \}. 
     \end{equation*}
  \item 
     \begin{equation*}
     \Omega_{line}:=\{ (x,v) \in \partial \Omega \times \mathbb{R}^3 \mid v \cdot n(x)>0 \}. 
     \end{equation*}
     
\end{itemize}

\end{definition}
\noindent
Due to \textbf{Proposition A.2} in \cite{Bri2}, there is no specular reflection trajectory that leads to any point in $ \Omega_{stop}$.
Given $(x,v) \in \overline{\Omega} \times \mathbb{R}^3$, we recall the definition of $t_{\partial}(x,v)$
    \begin{equation*}
        t_{\partial}(x,v)=\max\{t \geq 0 \mid x-vs \in \overline{\Omega}, \forall \, s \in [0,t]\}.
    \end{equation*}
The next proposition(\textbf{Proposition A.3} in \cite{Bri2}) gives a picture of a backward trajectory touching the boundary:

\begin{proposition}
    Given an open bounded domain in $\mathbb{R}^3$ with $C^1$ boundary. Then, we have
    
    (1) If there exists $t \in (0,t_{\partial}(x,v))$ such that $x-vt \in \partial\Omega$, then $(x-vt,v) \in \Omega_{rolling}$. 
    
    (2) $t_{\partial}(x,v)=0$ if and only if $(x,v) \in \Omega_{stop}\cup\Omega_{rebound}$.

    (3) $(x-vt_{\partial},v) \in \Omega_{stop}\cup\Omega_{rebound}$.
    
    \end{proposition}

We always need to consider some strange trajectory with a strange path. 
Thankfully, the following proposition (\textbf{Proposition A.4 in \cite{Bri2}}) shows that we do not need to worry about the case that the backward trajectory reaching $\Omega_{stop}$ or bounce infinite times in a finite time since those set are measure zero:     \begin{proposition} \label{non inf bounce}
   We consider an open bounded domain in $\mathbb{R}^3$ with $C^1$ boundary and $(x,v) \in \overline{\Omega} \times \mathbb{R}^3$. Then, for any $t \geq 0$ the trajectory of $(x,v)$ with specular reflection after time $t$ has at most a countable number of rebound and rolling. Moreover, the set with infinite rebound and rolling in finite time is zero measure in $\overline{\Omega} \times \mathbb{R}^3$.
\end{proposition}

%\begin{proof}
 %   \textcolor{red}{will be added, same as Prop A.4 in \cite{Bri2}}
%\end{proof}

\begin{remark}\label{non stop}
Similarly, with the above assumption the subset of $\overline{\Omega} \times \mathbb{R}^3$ with a backward trajectory leading to $\Omega_{stop}$ is also measure zero in $\overline{\Omega} \times \mathbb{R}^3$.
\end{remark}

Now, we introduce the characteristic line by defining:

\begin{equation}
(t_0,x_0,v_0):=(0,x,v),
  \end{equation}

  \begin{equation}
  \begin{split}
&(t_{k+1}, x_{k+1},v_{k+1})\\
:=&\begin{cases}
(\infty,x_k,v_k),\ &(x_k,v_k) \in \Omega_{stop},\\
(t_k+t_{\partial}(x_k,v_k),x_k-v_kt_{\partial}(x_k,v_k),R(x_k-v_kt_{\partial}(x_k,v_k),v_k)),\  &(x_k,v_k) \notin \Omega_{stop},
\end{cases}
\end{split}
  \end{equation}

  \begin{equation}
n(t,x,v):=\sup\{ k \in \mathbb{N}:\ t_k(x,v)\leq t \} .
  \end{equation}
  Note that $n(t,x,v)<\infty$ in case $\{ t_k \}$ is unbounded.

  Next, we define the last rebound from the backward trajectory of $(x,v)$:

   \begin{equation}
(t_{fin}, x_{fin},v_{fin}):=
\begin{cases}
(t,x_n,v_n),  & n(t,x,v) < \infty,\ \, t_{n(t,x,v)+1}=\infty,\\
(t_n,x_n,v_n), &n(t,x,v) < \infty,\ \, t_{n(t,x,v)+1}<\infty,\\
\lim\limits_{n \rightarrow \infty}(t_n,x_n,v_n), &n(t,x,v) = \infty.
\end{cases}
  \end{equation}

  Finally, for $0 \leq s \leq t$, we define the characteristic line :
  \begin{align*}
  \begin{cases}
     X_t(x,v)&:=x_{fin}(t,x,-v)-(t-t_{fin}(t,x,-v))v_{fin}(t,x,-v),\\
V_t(x,v)&:=-v_{fin}(t,x,-v). 
  \end{cases}
  \end{align*}

   \begin{align*}
  \begin{cases}
     X_{s,t}(x,v)&:=X_s(X_t(x,-v),-V_t(x,-v)),\\
V_{s,t}(x,v)&:=V_s(X_t(x,-v),-V_t(x,-v)). 
  \end{cases}
  \end{align*}

  %Note that by \textbf{Lemma \ref{non inf bounce}} and \textbf{Remark \ref{non stop}}, for almost surely, we have 
%\begin{equation*}
 %   (X_{s,t}(x,v),V_{s,t}(x,v))=(X_{t-s}(x,v),V_{t-s}(x,v)).
%\end{equation*}

%%%%%%%%%%%%%%%%%%%%%%%%%%%%%%%%%%%%%%%%%%%%%%%%%%%%%%
\section*{Acknowledgement}
%%%%%%%%%%%%%%%%%%%%%%%%%%%%%%%%%%%%%%%%%%%%%%%%%%%%%%
The author would like to expresses sincere gratitude to Professor Laurent Desvillettes for his valuable suggestions during the preparation of this manuscript, in the framework of a research sojourn in Université Paris Cité, IMJ-PRG.

This work is supported by the NSTC Graduate Student Study Abroad Program with grant number 113-2917-I-002-054.

%%%%%%%%%%%%%%%%%%%%%%%%%%%%%%%%%%%%%%%%%%%%%%%%%%%%%%
\section*{Data availability statement}
%%%%%%%%%%%%%%%%%%%%%%%%%%%%%%%%%%%%%%%%%%%%%%%%%%%%%%

Data sharing is not applicable to this article, as no data sets were generated or analyzed during the current study.

%%%%%%%%%%%%%%%%%%%%%%%%%%%%%%%%%%%%%%%%%%%%%%%%%%%%%%

\end{document}